\documentclass[11pt, letterpaper,reqno]{amsart}
\usepackage{amsthm}
\usepackage{amsmath}
\usepackage{amssymb}
\usepackage{mathrsfs}
\usepackage{enumitem}
\usepackage{array}
\usepackage[T1]{fontenc}
\usepackage{pifont}
\usepackage{float}
\usepackage{amsmath}
\usepackage[pdfpagelabels,hyperindex]{hyperref}
\hypersetup{colorlinks=true,              breaklinks=true,urlcolor= blue,               linkcolor= blue,citecolor= cyan}

\makeatother
\newtheorem{thm}{Theorem}[section]

\newtheorem{prop}[thm]{Proposition}
\newtheorem{cor}[thm]{Corollary}
\newtheorem{lem}[thm]{Lemma}
\newtheorem{pro}[thm]{Proposition}

\theoremstyle{definition}
\newtheorem{defn}[thm]{Definition}
\newtheorem{rem}[thm]{Remark}
\newtheorem{exa}[thm]{Example}

\numberwithin{equation}{section}

\renewcommand{\leq}{\leqslant}

\renewcommand{\geq}{\geqslant}

\newcommand{\ex}{\mathbb{E}}

\newcommand{\allchisum}{\sum_{\chi\bmod m}}

\begin{document}
\author[Youssef Sedrati]{Youssef Sedrati}
\address{INSTITUT ÉLIE CARTAN DE LORRAINE, UNIVERSITÉ DE LORRAINE, BP 70239,  54506 VANDOEUVRE LÉS-NANCY , FRANCE }
\email{youssef.sedrati@univ-lorraine.fr}


\title[Prime number race over function fields]{Inequities in the Shanks-Renyi Prime number \\ race over function fields}

\baselineskip=17pt

\begin{abstract}
Fix a prime $p >2$ and a finite field $\mathbb{F}_{q}$ with $q$ elements, where $q$ is a power of $p$. Let $m$ be a monic polynomial in the polynomial ring $\mathbb{F}_{q}[T]$ such that  $\deg(m)$ is large. Fix an integer $r\geq 2$, and let $a_1,\dots,a_r$ be distinct residue classes modulo $m$ that are relatively prime to $m$. In this paper, we derive an asymptotic formula for the natural density $\delta_{m;a_1,\dots,a_r}$ of the set of all positive integers $X$ such that $\sum\limits_{N=1}^{X} \pi_{q}(a_1,m,N) > \sum\limits_{N=1}^{X} \pi_{q}(a_2,m,N)>\dots> \sum\limits_{N=1}^{X} \pi_{q}(a_r,m,N)$, where $\pi_{q}(a_i,m,N)$ denotes the number of irreducible monic polynomials in $\mathbb{F}_{q}[T]$ of degree $N$ that are congruent to $ a_i \bmod m$, under the assumption of LI (Linear Independence Hypothesis). Many consequences follow from our results. First, we deduce the exact rate at which $\delta_{m;a_1,a_2}$ converges to $\frac{1}{2}$ as $\deg(m)$ grows, where $a_1$ is a quadratic non-residue and $a_2$ is a quadratic residue modulo $m$, generalizing the work of Fiorilli and Martin. Furthermore, similarly to the number field setting, we show that two-way races behave differently than races involving three or more competitors, once $\deg(m)$ is large. In particular, biases do appear in races involving three or more quadratic residues (or quadratic non-residues) modulo $m$. This work is a function field analog of the work of Lamzouri, who established similar results in the number field case. However, we exhibit some examples of races in function fields where LI is false, and where the associated densities vanish, or behave differently than in the number field setting.
\end{abstract}
\maketitle
\section{Introduction}

Chebyshev noticed in 1853 that there seems to be more primes of the form $4n+3$ than those of the form $4n+1$ in initial intervals of the integers. More generally, $\pi(x;q,a)$ tends to exceed $\pi(x;q,b)$ when $a$ is a quadratic non-residue and $b$ is a quadratic residue modulo $q$,
where $\pi(x;q,a)$ is the number of primes less than $x$ that are congruent to $a$ modulo $q$. This phenomenon is called "Chebyshev's bias". In $1994$, Rubinstein and Sarnak \cite{RS94} determined conditionally under some natural assumptions that the logarithmic density $\delta (q;a,b)$ of the set of real numbers $x \geq 2$ for which $\pi(x;q,a) > \pi(x;q,b)$, exists and is positive. They also computed some of these densities and showed in particular that $\delta (4;3,1) \approx 0.9959$. 
The assumptions they used are the Generalized Riemann Hypothesis GRH and the Linear Independence Hypothesis LI (which is the assumption that the non-negative imaginary parts of the zeros of all Dirichlet $L$-functions attached to primitive characters modulo $q$ are linearly independent over $\mathbb{Q}$). They proved, under these two hypotheses, that $\delta (q;a,b) =  \frac{1}{2}$ if $a$ and $b$ are both quadratic residues or quadratic non-residues modulo $q$, and otherwise that  $\delta (q;a,b) > \frac{1}{2}$ if $a$ is a quadratic non-residue and $b$ is a quadratic residue modulo $q$, thus confirming Chebyshev's observation. We mention the interesting articles of Granville and Martin \cite{GM}, and Ford and Konyagin \cite{FK2} for detailed review about the history of this subject. 
Under GRH and LI, Rubinstein and Sarnak also showed that $\delta(q;a,b)$ approaches $\frac{1}{2}$ when $q \rightarrow \infty$. Under the same assumptions, Fiorilli and Martin \cite{fiorilliMartin} established an asymptotic formula for the density $\delta(q;a,b),$ and then provided the exact rate at which $\delta(q;a,b)$ converges to $\frac{1}{2}$ as $q$ grows. They also used this asymptotic formula in order to compute many of these densities. 

Under GRH and LI, Rubinstein and Sarnak completely solved the Shanks-Rényi prime number race problem, by showing for any fixed integer $r \geq 2$ that the set of real numbers $x \geq 2$ such that $\pi(x;q,a_1) > \dots > \pi(x;q,a_r)$ has a positive logarithmic density $\delta (q;a_1,\dots,a_r)$. 
Feuerverger and Martin \cite{FeM1} were the first to compute densities with three competitors. They showed, under GRH and LI, that $\delta (8;3,5,7) = \delta (8;7,5,3) \ne \frac{1}{3!}$ even if $3$, $5$ and $7$ are quadratic non-residues modulo $8$. This proves, that unlike the case of two competitors, some densities $\delta (q;a_1,a_2,a_3)$ involving three competitors are asymmetric under permutations of the $a_i$, even if the $a_i$ are all quadratic non-residues modulo $q$. However, all the explicit examples they exhibit satisfy $r \leq 4$ and $q \leq 12$.  
Lamzouri \cite{lamzouri} has then established an asymptotic formula for the densities in races with three or more competitors when $q$ is large enough. He discovered that two-way races behave differently than races involving three or more competitors, if $q$ is sufficiently large. 

In this article we generalize the results of \cite{lamzouri} in the case of polynomial rings over finite fields. We shall first recall some definitions and results. To this end, we fix a finite field $\mathbb{F}_{q}$ with $q$ elements, where $q$ is a power of a prime $p > 2$. Define $\log_q x$, for a positive real number $x$, by $\log_q x := \frac{\log x}{\log q}$. Let $m$ be a monic polynomial in $\mathbb{F}_{q}[T]$ whose degree is at least two. Let $M$ be the degree of $m$. We denote by $\phi(m)$ the number of invertible congruence classes modulo $m$. The letter $P$ will always denote an irreducible monic polynomial in $\mathbb{F}_{q}[T].$ To abbreviate our notations, we will sometimes use the word prime to refer to an irreducible monic polynomial in $\mathbb{F}_{q}[T]$.  
Let $\mathcal{M}_{q}$, $\mathcal{P}_{q}$ be the sets of monic and prime polynomials in $\mathbb{F}_{q}[T]$, respectively. 
For a polynomial $f \in \mathbb{F}_{q}[T]$, we denote its norm by $|f|$ which is defined to be $q^{\deg(f)}$ if $f \ne 0$, and equals $0$ otherwise.
There is a striking similarity between the ring of integers and polynomial rings over finite fields (see \cite{rosen}). Using this analogy, Cha \cite{BCha} found that many results of \cite{RS94} persists for irreducible polynomials over finite fields under some analogous assumptions. In particular, he established, under a certain hypothesis,
that like in the number field case primes are biased toward quadratic non-residues. On the other hand, he proved unconditionally in the case of polynomial rings that there are some biases in unexpected directions.
Since then, other generalizations have been extensively studied by many authors. Cha and Seick \cite{CSK} generalized the results of \cite{FeM1}. Cha, Fiorilli and Jouve studied the prime number race for elliptic curves over the function field of a proper, smooth and geometrically connected curve over a finite field. Further related subjects have been studied since then in  \cite{DevMeng}, \cite{devin2020limiting}, \cite{chaIm}, \cite{perret2020roots}, \cite{bailleul2021chebyshev}.

Let $r \geq 2$ be an integer. For any positive integer $N$ and an element $a$ in $\mathbb{F}_{q}[T]$ prime to $m$, we let $\pi_q(a,m,N)$ be the prime counting function defined by 
$$
\pi_q(a,m,N) := \# \{ P \in \mathcal{P}_{q} \ | \ P  \equiv a \bmod m, \deg(P) = N \}.
$$
Define $\mathcal{A}_r(m)$ to be the set of $r$-tuples of distinct residue classes $(a_1,a_2,\dots, a_r)$ modulo $m$ which are coprime to $m$. For $(a_1,a_2,\dots, a_r) \in \mathcal{A}_r(m)$, let $P_{m;a_1,\dots,a_r}$ be the set of all positive integers $X$ with 
$$
\sum\limits_{N=1}^{X} \pi_q(a_1,m,N) > \sum\limits_{N=1}^{X} \pi_q(a_2,m,N) >\dots > \sum\limits_{N=1}^{X} \pi_q(a_r,m,N).
$$
Throughout this article, we will use the following definition. 
\begin{defn}[Linear Independence Hypothesis (LI)]
We say that $m$ satisfies LI if the multi-set  $$\bigcup\limits_{\substack{\chi \bmod m \\  \chi \neq \chi_0}}\lbrace \gamma \in [0,\pi] :  L(\frac{1}{2} + i\gamma, \chi) = 0  \rbrace \cup \lbrace\pi\rbrace$$ is linearly independent over $\mathbf{Q}$ (see Section \ref{background} for the definition of the $L$-functions).
\end{defn}
Assuming LI, Cha proved in \cite{BCha} that the limit 
\begin{equation*}
\begin{aligned}
\delta_{m;a_1,\dots,a_r} := \lim\limits_{X \rightarrow \infty}^{} \frac{\left| P_{m;a_1,\dots,a_r} \cap \left\{1,2, \dots, X \right\} \right|}{X}
\end{aligned}
\end{equation*}
exists and is positive.
He also established that $\delta_{m;a,b}\to 1/2$ as $\deg(m) \to \infty$, uniformly for all $(a,b) \in \mathcal{A}_{2}(m)$. In fact, he proved that in general all biases disappear when $ \deg(m) \to \infty$. Let
$$ \Delta_r(m):= \max_{(a_1,a_2,\dots,a_r)\in \mathcal{A}_r(m)}\left|\delta_{m;a_1,\dots,a_r}-\frac{1}{r!}\right|.$$
 Generalizing the work of Rubinstein and Sarnak \cite{RS94}, Cha showed under the assumption of LI, that for any fixed $r \geq 2$ we have
\begin{equation}
\Delta_r(m)\to 0 \text{ as } \deg(m) \to \infty.
\end{equation}
We establish asymptotic formulas for the densities $\delta_{m;a_1,\dots,a_r}$ when $r \geq 2$ is fixed and $\deg(m)$ is large enough. An interesting consequence of our results is that $\Delta_r(m)$ behaves in a completely different way than $\Delta_2(m)$ when $r \geq 3$. A similar behavior was uncovered in the work of Lamzouri \cite{lamzouri} in the number field setting. 
\begin{thm}
Assume LI. Let $r \geq 3$ be a fixed integer. If $\deg(m)$ is large enough, then 
\begin{equation}
\label{delta2}
\begin{aligned}
\Delta_2(m) &= \frac{1}{|m|^{1/2 + o(1)}},
\end{aligned}
\end{equation}
and 
\begin{equation}
\begin{aligned}
\label{deltar3}
\Delta_r(m)  &\asymp_{r} \frac{1}{\log_q |m|}.
\end{aligned}
\end{equation}
\end{thm}

Under the assumption of LI, we know from \cite[Theorem $6.1$]{BCha} that if $a_1$ and $a_2$ are both quadratic residues or both quadratic non-residues modulo $m$, then $\delta_{m;a_1,a_2} = \delta_{m;a_2,a_1} = \frac{1}{2}$.  We say in this case that the race $\{m;a_1,a_2\}$ is unbiased. In general, we give the following definition.  
\begin{defn} Let $(a_1,\dots,a_r)\in \mathcal{A}_r(m)$. The race $\{m;a_1,\dots,a_r\}$ is said to be unbiased if for every permutation $\sigma$ of the set $\{1,2,\dots,r\}$ we have
$$ \delta_{m;a_{\sigma(1)},\dots, a_{\sigma(r)}}=\delta_{m;a_1,\dots,a_r}=\frac{1}{r!}.$$
Furthermore, a race is said to be biased if this condition does not hold.
\end{defn}
Our next result shows that, unlike two-way races, biases appear in races involving three or more quadratic residues (or quadratic non-residues) modulo $m$, if $\deg(m) = M$ is large enough. 
\begin{thm}
\label{aboutQuadRaces}Assume LI. Let $r \geq 3$ be a fixed integer. Then there exists a positive number $m_{0}(r,q)$ such that if $\deg(m) \geq m_{0}(r,q)$ then there are two $r$-tuples $(a_1,\dots,a_r),(b_1,\dots,b_r) \in \mathcal{A}_{r}(m)$, with all of the $a_i$ being quadratic residues and all of $b_i$ quadratic non-residues modulo $m$, and such that both races $\{ m; a_1,\dots,a_r \}$ and $\{ m; b_1,\dots,b_r \}$ are biased. 
\end{thm}
Generalizing a result of Rubinstein and Sarnak [Proposition $3.1$, \cite{RS94}], 
Cha proved in \cite[Theorem $6.1$]{BCha} that if there exists a polynomial $\rho \not \equiv 1$ mod $m$ satisfying the following:
\begin{equation*}
\begin{aligned}
\rho^{3} \equiv 1 \bmod m, \quad a_2 \equiv a_1 \rho \bmod m, \quad \text{and} \quad a_3 \equiv a_1 \rho^2 \bmod m,
\end{aligned}    
\end{equation*}
then the race $\{m;a_1,a_2,a_3\}$ is unbiased under the assumption of LI.
We show that when LI is violated, this phenomenon does not hold. In fact, by taking $\rho = T \in \mathbb{F}_{3}[T]$, $m = T^2 + T + 1 \in \mathbb{F}_{3}[T]$, $a_1 = 2 \in \mathbb{F}_{3}[T]$, $a_2 = 2 T \in \mathbb{F}_{3}[T]$, and $a_3 = T + 1\in \mathbb{F}_{3}[T]$, we can show (see Section \ref{section10}) that $$\delta_{m;a_3,a_2,a_1} = \frac{1}{4} \ne \frac{1}{6}.$$ Moreover, we find out that a generalization of Littlewood's Theorem \cite{littlewood} in the function field setting is false. Indeed, for the same $m$ we have that 
$$\sum\limits_{N=1}^{X} \pi_q(T,m,N) < \sum\limits_{N=1}^{X} \pi_q(T+1,m,N),$$ for all large enough positive integers $X$. In particular, this implies $$\delta_{m;T,T+1} = 0.$$ We also give another example in the case of polynomials of degree $3$.  Choosing $m = T (T+1) (T+2) \in \mathbb{F}_{3}[T]$, we prove that $$\delta_{m;1,T^2 + 1} = 0,$$ and more generally that $\delta_{m;1,a} = 0$ for all quadratic non-residues $a$ mod $m$ (see Section \ref{section10} for the proof of these results).

In the next section we will discuss these results in details. In particular we shall describe the asymptotic formulas we prove for the densities $\delta_{m;a_1,\dots,a_r}$ according to whether $r=2$ or $r \geq 3$ and deduce further consequences.

\section{Detailed statement of results}

Before stating our results, we first define some notations that will be used throughout this paper. For any positive integer $N$, we define $\pi_q(N)$ by 
$$
\pi_q(N) := \# \{ P \in \mathcal{P}_q \ | \ \deg(P) = N \}.
$$
Let $(a_1,\dots,a_r) \in \mathcal{A}_{r}(m)$ and define $E_{m;a_1,\dots,a_r}(X)$, for a positive integer $X$, by 
$$
E_{m;a_1,\dots,a_r}(X) := \left( E_{m;a_1}(X),\dots, E_{m;a_r}(X)\right),
$$
where 
$$
E_{m;a_i}(X) := \frac{X}{q^{X/2}} \sum\limits_{N=1}^{X} \left( \phi(m) \pi_q(a_i,m,N) - \pi_q(N) \right). 
$$
Cha showed the existence of a certain limiting distribution $\mu_{m;a_1,\dots,a_r}$ that is constructed from $E_{m;a_1,\dots,a_r}(X)$. Moreover, it follows from his work that 
$$
\delta_{m;a_1,\dots,a_r} = \mu_{m;a_1,\dots,a_r} \left\{  (x_1,\dots,x_r) \in \mathbb{R}^{r}: x_1 > x_2 > \dots > x_r \right\}.
$$
For a non-principal Dirichlet character $\chi$ modulo $m$, we denote by $\{\gamma_{\chi}\}$ the multiset of inverse zeros of the Dirichlet $L$-function associated to the character $\chi$ (see Section \ref{background}). Let $\chi_{0}$ denote the principal character modulo $m$. Define $S = \bigcup\limits_{\substack{\chi \bmod m \\ \chi \ne \chi_{0}}} \gamma_{\chi}$ and $\{ U(\gamma_{\chi}) \}_{\gamma_{\chi} \in S }$ be a multiset of independent random variables uniformly distributed on the unit circle. Let $a$ be in $\mathbb{F}_{q}[T]$ such that $(a,m)=1.$ It follows from the results of \cite{BCha} that under the assumption of LI, the limiting distribution $\mu_{m;a_1,\dots,a_r}$ is the probability measure corresponding to the random variable $X_{m;a_1,\dots,a_r} = (X_{m;a_1},\dots,X_{m,a_r})$, 
with 
\begin{equation*}
\begin{aligned}
X_{m;a} &= - C_{m}(a) X^{'} + \sum\limits_{\substack{\chi \bmod m \\ \chi \ne \chi_{0}} }^{} \sum\limits_{\Im{\left( \gamma_{\chi} \right) > 0}}^{} 2 \Re(\chi(a) U(\gamma_{\chi})) \left|  \frac{\gamma_{\chi}}{\gamma_{\chi} - 1} \right|,
\end{aligned}
\end{equation*}
where $X^{'}$ is a random variable independent from the $U(\gamma_{\chi})$, and which takes the values $\frac{\sqrt{q}}{q-1}$ and $\frac{q}{q-1}$ with equal probability $\frac{1}{2}$,
and $$C_m(a):=-1+ \sum_{\substack{b^2\equiv a \bmod m\\ b \in (\mathbb{F}_{q}[T]/(m))^{*}}} 1.$$
We remark that for $(a,m)=1$, the function $C_{m}(a)$ takes only two values : $C_{m}(a) = -1$ if $a$ is a quadratic non-residue modulo $m$, and $C_{m}(a) = C_{m}(1)$ if $a$ is quadratic residue modulo $m.$
Moreover, it is easy to check that $|C_m(a)| = |m|^{o(1)}$. \\
Let $\text{Cov}_{m;a_1,\dots,a_r}$ be the covariance matrix of $X_{m;a_1,\dots,a_r}$. Then a straightforward computation shows the following lemma (see the proof in Section \ref{background}). 
\begin{lem}
\label{varAlea}
The entries of  $\textup{Cov}_{m;a_1,\dots,a_r}$ are
$$\textup{Cov}_{m;a_1,\dots,a_r}(j,k)= \begin{cases}  N_m + \frac{1}{4} \left( \frac{q}{q-1} - \frac{\sqrt{q}}{q-1} \right)^{2} C_{m}(a_j)^{2}&\text{ if } j=k \\ B_m(a_j,a_k) + \frac{1}{4} \left( \frac{q}{q-1} - \frac{\sqrt{q}}{q-1} \right)^{2} C_{m}(a_j) C_{m}(a_k) &\text{ if } j\neq k,\end{cases}$$
where
\begin{equation}
\begin{aligned}
N_m & := 2 \sum\limits_{\substack{\chi  \bmod m \\ \chi \ne \chi_{0}} }^{} \sum\limits_{\Im{\left( \gamma_{\chi} \right) > 0}}^{} \left| \frac{\gamma_{\chi}}{\gamma_{\chi} - 1} \right|^{2}, 
\end{aligned}
\end{equation}
and
\begin{equation}
\begin{aligned}
B_m(a_j,a_k) &:= \sum\limits_{\substack{\chi  \bmod m \\ \chi \ne \chi_{0}}}^{} \sum\limits_{\Im{\left( \gamma_{\chi} \right)} > 0}^{} (\chi(a_j/a_k) + \chi(a_k/a_j)) \left| \frac{\gamma_{\chi}}{\gamma_{\chi} - 1}\right|^{2}.
\end{aligned}
\end{equation}
\end{lem}
We prove asymptotic formulas for $N_{m}$ and $B_{m}(a,b)$ for $(a,b) \in \mathcal{A}_{2}(m)$ in Section \ref{section4}. This leads us to deduce later (see Lemma \ref{formuleNm} and Corollary \ref{majorBm}) that 
\begin{equation}
\label{asymptotics}
\begin{aligned}
N_{m} \sim \frac{q}{q-1} \phi(m) \log_q |m| \text{ and} \  B_{m}(a,b) \ll \phi(m),
\end{aligned}
\end{equation}
for all $(a,b) \in \mathcal{A}_{2}(m).$ \\
Under the assumption of LI, in the case $r=2$, we establish later the following asymptotic formula for the densities $\delta_{m;a,b}$. 
\begin{thm}
\label{asymDelta2}
Let $m \in \mathcal{M}_{q}$ be of large degree.
Assume LI. Let $(a,b) \in \mathcal{A}_{2}(m)$, then 
\begin{equation}
\begin{aligned}
\delta_{m;a,b} = \frac{1}{2} - \frac{(\sqrt{q} + q)}{2(q-1)} \frac{\left( C_{m}(a) - C_{m}(b) \right) }{\sqrt{2 \pi V_{m}(a,b)}}+ O \left( \frac{C_{m}(1)^2 \log_q |m|}{\phi(m)} \right),
\end{aligned}
\end{equation}
where $V_{m}(a,b) = 2 \left( N_m - B_m(a,b) \right)$. 
\end{thm}
Combining this result with \eqref{asymptotics} and using that $\phi(m) = |m|^{1 + o(1)}$ we deduce \eqref{delta2}. \\
We now state our result for $\delta_{m;a_1,\dots,a_r}$ when $r \geq 3$.
Let
\begin{equation}
\begin{aligned}
C_m := \max\limits_{1 \leq i \leq r} \left| C_m(a_i) \right| \text{ and }  B_{m} := \max\limits_{1 \leq j  < k \leq r}^{} \left| B_{m}(a_j,a_k)\right|
\end{aligned}
\end{equation}
for all $(a_1,\dots,a_r) \in \mathcal{A}_{r}(m).$ Furthermore, 
for $1 \leq j \ne k \leq r$ we recall the following integrals which appear in the work of Lamzouri [see \cite{lamzouri}].
\begin{equation*}
\begin{aligned}
\alpha_j(r) &:= (2 \pi)^{-r/2} \int_{x_1 >\dots>x_r}^{} x_j \exp{\left( - \frac{x_1^2 +\dots+ x_r^{2}}{2}\right)} {d}x_1\dots{d}x_r, \\
\lambda_j(r) &:= (2 \pi)^{-r/2} \int_{x_1 >\dots>x_r}^{} (x_j^2 - 1) \exp{\left( - \frac{x_1^2 +\dots+ x_r^{2}}{2}\right)} {d}x_1\dots{d}x_r,
\end{aligned}
\end{equation*}
and 
\begin{equation*}
\begin{aligned}
\beta_{j,k}(r) &:= (2 \pi)^{-r/2} \int_{x_1 >\dots>x_r}^{} x_j x_k \exp{\left( - \frac{x_1^2 +\dots+ x_r^{2}}{2}\right)} {d}x_1 \dots{d}x_r.
\end{aligned}
\end{equation*}

Then we have 
\begin{thm} 
\label{forAsymDelta3}
Fix an integer $r \geq 3$ and let $m \in \mathcal{M}_{q}$ of large degree. Assume LI, then 
\begin{equation*}
\begin{aligned}
\delta_{m;a_1,\dots,a_r}&= \frac{1}{r!} - \frac{(q + \sqrt{q})}{2 \sqrt{N_{m}} (q-1)} \sum\limits_{j=1}^{r}  \alpha_j(r) C_{m}(a_j) +  \frac{1}{N_{m}} \sum\limits_{1 \leq j < k \leq r}^{}  \beta_{j,k}(r) B_{m}(a_j,a_k)   \\
&\quad+ \frac{1}{4 N_{m}}  \frac{q + q^2}{(q-1)^{2}} \left(\sum\limits_{j=1}^{r} \lambda_{j}(r) C_{m}(a_j)^{2} + 2 \sum\limits_{1 \leq j < k \leq r}^{} \beta_{j,k}(r) C_{m}(a_j) C_{m}(a_k)  \right)   \\
&\quad+ O_{r} \left( \frac{1}{N_{m}} + \frac{C_m B_m}{N_{m}^{3/2}} + \frac{B_{m}^{2}}{N_{m}^{2}}\right).
\end{aligned}
\end{equation*}
\end{thm}
We compare this result with \cite[Theorem $2.1$]{lamzouri} where a similar asymptotic formula is proved for the densities in races with three or more competitors in the number field case. The major difference is that in the function field case we observe some new factors depending on $q$ (the number of element of the field $\mathbb{F}_{q}$). 
This is due to the factor $\mathcal{B}_{m;a_1,\dots,a_r}(t)$ of the explicit formula for the Fourier Transform $\hat{\mu}_{m;a_1,\dots,a_r}(t)$ (see below Section \ref{section5} Equation \eqref{forExplicitCha}).
A direct corollary follows.
\begin{cor}
\label{forAsymCor}
Under the same hypotheses of Theorem \ref{forAsymDelta3} we have 
\begin{equation*}
\begin{aligned}
\delta_{m;a_1,\dots,a_r} &= \frac{1}{r!} - \frac{q+ \sqrt{q}}{2 \sqrt{N_m} (q-1)} \sum\limits_{j=1}^{r} \alpha_j(r) C_{m}(a_j) 
+ \frac{1}{N_{m}} \sum\limits_{1 \leq j < k \leq r}^{} \beta_{j,k}(r) B_{m}(a_j,a_k) \\
&\quad+O_{r} \left( \frac{C_m^2}{N_m} + \frac{B_m^2}{N_m^2} \right).
\end{aligned}
\end{equation*}
\end{cor}
We recall Lemma $4.5$ of \cite{lamzouri}.
\begin{lem}[{\cite[Lemma $4.5$]{lamzouri}}]
 We have $\beta_{1,2}(2)=0.$
Moreover, one has
$$ \alpha_1(3)= \frac{1}{4\sqrt{\pi}}, \quad \alpha_2(3)= 0, \quad \alpha_3(3)= -\frac{1}{4\sqrt{\pi}},$$
and
$$ \beta_{1,2}(3)= \beta_{2,3}(3)= \frac{1}{4\pi\sqrt{3}}, \quad \beta_{1,3}(3)= -\frac{1}{2\pi\sqrt{3}}.$$
\end{lem}
Hence, in particular for $r=3$, by combining the previous lemma and Corollary \ref{forAsymCor} we get the following
\begin{cor} 
\label{cor2}
Under the same assumptions of Theorem \ref{forAsymDelta3} we have 
\begin{equation*}
\begin{aligned}
\delta_{m;a_1,a_2,a_3} &= \frac{1}{6} + \frac{q + \sqrt{q}}{8 \sqrt{\pi N_{m}} (q-1)} \left( C_{m}(a_3) - C_{m}(a_1) \right) \\
&\quad+ \frac{1}{4 \pi \sqrt{3} N_{m}} \left( B_{m}(a_1,a_2) + B_{m}(a_2,a_3) - 2 B_{m}(a_1,a_3) \right) + O\left( \frac{C_{m}^2}{N_m} + \frac{B_m^2}{N_m^2}\right).
\end{aligned}
\end{equation*}
\end{cor}
It is easy to check that 
\begin{equation*}
\begin{aligned}
\sum\limits_{j=1}^{r} \alpha_{j}(r) &= \sum\limits_{j=1}^{r} \lambda_{j}(r) = \sum\limits_{1 \leq j < k \leq r}^{} \beta_{j,k}(r) = 0.
\end{aligned}
\end{equation*}
Moreover, we also know that if all the $a_i$ are quadratic non-residues modulo $m$ then $C_m(a_i) = -1$ for all $1 \leq i \leq r$, and if all the $a_i$ are quadratic residues modulo $m$ then $C_m(a_i) = C_m(1)$. 
Hence we deduce from Theorem \ref{forAsymDelta3} the following corollary.
\begin{cor} 
\label{cor3}
Let $r \geq 3$ be an integer and $m \in \mathcal{M}_{q}$ be of large degree. Assume LI. Then for any $(a_1,\dots,a_r) \in \mathcal{A}_{r}(m)$ such that all the $a_i$ are quadratic residues modulo $m$ or all are quadratic non-residues modulo $m$, we have 
\begin{equation*}
\begin{aligned}
\delta_{m;a_1,\dots,a_r} &= \frac{1}{r!} + \frac{1}{N_{m}} \sum\limits_{1 \leq j < k \leq r}^{} \beta_{j,k}(r) B_{m}(a_j,a_k) + O_{r} \left( \frac{1}{N_{m}} + \frac{C_{m} B_{m}}{N_{m}^{3/2}} + \frac{B_{m}^{2}}{N_{m}^{2}}\right).
\end{aligned}
\end{equation*}
\end{cor}
\begin{rem}
First, we note that our asymptotic formula for the densities $\delta_{m;a_1,\dots,a_r}$ when $r \geq 3$ is still valid in the case $r=2$, but does not imply Theorem \ref{asymDelta2} since the error term may exceed the main term. This is due to the fact that $\beta_{1,2}(2) = 0$, and hence the term involving $B_{m}(a_1,a_2)$ is missing in the case $r=2$ in the asymptotic formula for $\delta_{m;a_1,a_2}$. On the other hand, when $r \geq 3$, we will show later that the contribution of the terms $B_{m}(a_i,a_j)$ in the densities $\delta_{m;a_1,\dots,a_r}$ may be $\gg_{r,q}$ $\frac{1}{\log |m|}$.
This is the main difference between the cases $r=2$ and $r \geq 3$ that explains why $\Delta_{r}(m)$ for $r \geq 3$ behaves differently than $\Delta_2(m)$. 
\end{rem}
The following result is a strong version of Theorem \ref{aboutQuadRaces}, and is the analogue of \cite[Theorem $2.8$]{lamzouri} in the function fields setting. It shows that, like in the number field case, biases appear in races involving three or more quadratic residues (or quadratic non-residues) modulo $m$, once $\deg(m)$ is large enough.  
\begin{thm}
\label{biasedQuadratic}
Let $r \geq 3$ be an integer, $m \in \mathcal{M}_{q}$ be of large degree. Assume LI. Then there exist residue classes $(a_1,\dots,a_r),(b_1,\dots,b_r) \in \mathcal{A}_{r}(m)$ where $a_1,\dots,a_r$ are all quadratic residues modulo $m$ and  $b_1,\dots,b_r$ are all quadratic non-residues modulo $m$, and there exists a permutation $\sigma$ of the set $\{ 1,\dots,r \}$ such that 
\begin{equation*}
\begin{aligned}
\delta_{m;a_1,\dots,a_r} = \delta_{m;b_1,\dots,b_r} < \frac{1}{r!} - \frac{c_{1}(q,r)}{(\log |m|)^{3}} \ \ &\text{and} \ \ \delta_{m;a_{\sigma(1)},a_{\sigma(2)},\dots,a_{\sigma(r)}} = \delta_{m;b_{\sigma(1)},b_{\sigma(2)},\dots,b_{\sigma(r)}} > \frac{1}{r!} + \frac{c_{1}(q,r)}{(\log |m|)^3},
\end{aligned}
\end{equation*}
for some constant $c_1(q,r) > 0$ which depends on $q$ and $r$.
\end{thm}
Using Theorem \ref{forAsymDelta3} we prove the following result which implies \eqref{deltar3}. 
\begin{thm}
\label{biasedRaces}
Let $r \geq 3$ be an integer, $m \in \mathcal{M}_{q}$ be of large degree. Assume LI, then for all $(a_1,\dots,a_r) \in \mathcal{A}_{r}(m)$ we have
\begin{equation*}
\begin{aligned}
\left| \delta_{m;a_1,\dots,a_r} - \frac{1}{r!} \right| &\ll_{r}    \frac{1}{\log_q |m|}.
\end{aligned}
\end{equation*}
Furthermore, there exist residue classes $(b_1,\dots,b_r),(d_1,\dots,d_r) \in \mathcal{A}_{r}(m)$ such that 
\begin{equation*}
\begin{aligned}
\delta_{m;b_1,\dots,b_r} > \frac{1}{r!} + \frac{c_2(r)}{\log_q |m|} &\text{  and} \ \ \delta_{m;d_1,\dots,d_r} < \frac{1}{r!} - \frac{c_2(r)}{\log_q |m|}, 
\end{aligned}
\end{equation*}
where $c_2(r)>0$ is a constant depending only on $r$. 
\end{thm}
This theorem motivates us to determine for which residue classes $a_1,\dots,a_r$ modulo $m$ the distance $\left| \delta_{m;a_{\sigma(1)},\dots,a_{\sigma(r)}} - \frac{1}{r!} \right|$ can be $\gg_{r,q} \frac{1}{\log |m|}$ for some permutation $\sigma$ of the set $\left\{1,\dots,r \right\}$. To reach this goal, let us start by giving the following definition. 
\begin{defn}
Fix an integer $r\geq 3$ and let $m \in \mathcal{M}_{q}$ be of large degree. We say that a race $\{m; a_1,\dots, a_r\}$ is $m$-extremely biased if there exists a permutation $\sigma$ of the set $\{1,\dots,r\}$ such that 
$$\left|\delta_{m; a_{\sigma(1)}, \dots, a_{\sigma(r)}}-\frac{1}{r!}\right|\gg_{r,q} \frac{1}{ \log |m|}.$$
\end{defn}
When the degree of the residue classes $a_1,\dots,a_r$ is bounded and $\deg(m)$ is large then we can give a criteria for $m$-extremely biased races $\{m,a_1,\dots,a_r \}.$  
\begin{thm}
\label{criteriaExtreme}
Fix an integer $r\geq 3$ and let $m \in \mathcal{M}_{q}$ be of large degree. Assume LI. Let $A \geq 1$ be a real number. Then if $(a_1,a_2,\dots,a_r) \in \mathcal{A}_{r}(m)$  with  $\deg(a_i) \leq A$ for all $1 \leq i \leq r$, the race $\{ m;a_1,\dots,a_r\}$ is $m$-extremely biased if and only if: 
\begin{enumerate}
    \item There exist $1 \leq j \ne k \leq r$ such that $a_j + a_k = 0,$
    \item There exist $1 \leq j \ne k \leq r$ such that $a_j/a_k$ is a power of a prime (irreducible monic polynomial in $\mathbb{F}_{q}[T]$).
\end{enumerate}
Furthermore, if neither $(1)$ nor $(2)$ holds for any permutation $ \sigma$ of the set $\{1,\dots,r \}$, then 
\begin{equation*}
\begin{aligned}
\left|\delta_{m;a_{\sigma(1)},\dots,a_{\sigma(r)}} - \frac{1}{r!} \right| &= 
\begin{cases}
         O_{A,r,q} \left( \frac{M^2}{|m|} \right) &
         \text{if all the } a_{i} \text{ are quadratic residues (or quadratic} \\
         & \text{non-residues) $\bmod \ m$,}\\
     O_{\epsilon,r,q} \left( |m|^{-1/2 + \epsilon}\right)  & \text{otherwise.}
\end{cases}
\end{aligned}
\end{equation*}
\end{thm}

In order to understand the behavior of $\delta_{m;a_1,\dots,a_r}$, we should investigate the terms $B_{m}(a_i,a_j)$ for $1 \leq i < j \leq r$ in more detail.
The following theorem determines the order of magnitude of $\left| B_{m}(a,b) \right|$ for a generic pair $(a,b) \in \mathcal{A}_{2}(m).$
We show that on average $\left| B_{m}(a,b) \right| \asymp \log_q |m|$. This generalizes \cite[Theorem $2.9$]{lamzouri}.  
\begin{thm}
\label{FirstMomBm}
Let $r \geq 3$ be an integer, $m \in \mathcal{M}_{q}$ be of large degree. Assume LI. Then 
\begin{equation*}
\begin{aligned}
\frac{q}{q-1} \log_q |m| + O \left( \log \log_q |m| \right) \leq \frac{1}{\left| \mathcal{A}_{2}(m) \right|} \sum\limits_{(a,b) \in \mathcal{A}_{2}(m)}^{} \left| B_{m}(a,b) \right|  &\leq  \frac{17 q}{q-1} \log_q |m|  + O \left(\log \log_q |m| \right).     
\end{aligned}
\end{equation*}
\end{thm}
Let $n$ be a positive integer. Define $\mathcal{D}_r(n)$ to be the set of ordered $r$-tuples $(e_1,e_2,\dots, e_r)$ of distinct residue classes modulo $n$ that are coprime to $n$. Feuerverger and Martin \cite{FeM1} conjectured that there should exist a ``bias factor'' $F_n(e_1,\dots,e_r)$, defined as a linear combination of the $G_n(e_j) := -1 + \sum\limits_{\substack{ b^2 \equiv e_j \bmod n \\ 1 \leq b \leq n}}^{} 1$ such that
\begin{equation}
 F_n(e_1,\dots,e_r)>F_n(j_1,\dots,j_r) \implies \tilde{\delta} (n;e_1,\dots,e_r) > \tilde{\delta} (n;j_1,\dots,j_r),
\end{equation}
where $\tilde{\delta}(n;e_1,\dots,e_r)$ is the logarithmic density of the set of real numbers $x \geq 2$ such that 
$$
\pi(x;n,e_1) > \pi(x;n,e_2) > \dots > \pi(x;n,e_r). 
$$
Lamzouri \cite[Theorem $2.11$]{lamzouri} showed (conditionally on GRH and LI for Dirichlet $L$-functions) that this conjecture does not hold when $r \geq 3$. This motivates us to generalize this result to the function field setting. We prove the following theorem. 
\begin{thm}
\label{generMartin}
Let $m \in \mathcal{M}_{q}$ be of large degree. Assume LI. 
Fix an integer $r \geq 3$  and let $(\kappa_1,\kappa_2,\dots,\kappa_r) \in \mathbb{R}^{r}$ such that $(\kappa_1,\dots,\kappa_r) \ne (0,\dots,0)$. Then there exist two $r$-tuples \\ $(a_1,\dots,a_r)$,$(b_1,\dots,b_r) \in \mathcal{A}_{r}(m)$ such that 
\begin{equation*}
\sum\limits_{1 \leq j \leq r}^{} \kappa_j C_m(a_j) > \sum\limits_{1 \leq j \leq r}^{} \kappa_j C_m(b_j) \quad \text{and} \quad \delta_{m;a_1,\dots,a_r} < \delta_{m;b_1,\dots,b_r}.     
\end{equation*}
\end{thm}
This shows that this conjecture is also false in the function fields case. In the other direction, combining Theorems \ref{forAsymDelta3} and \ref{FirstMomBm}, we establish the generalization of the result obtained in \cite[Theorem $2.12$]{lamzouri}. 
\begin{thm}
\label{genra212}
Fix an integer $r \geq 3$ and let $m \in \mathcal{M}_{q}$ of large degree. Assume LI. Then there is a set $\Omega_r(m) \subset \mathcal{A}_{r}(m)$ with $\left| \Omega_r(m) \right| = o \left(\left| \mathcal{A}_{r}(m) \right|\right)$ such that for all $r$-tuples $(a_1,\dots,a_r),(b_1,\dots,b_r) \in \mathcal{A}_{r}(m) \setminus{\Omega_r(m)}$ we have
\begin{equation*}
- \sum\limits_{j=1}^{r} \alpha_j(r) C_{m}(a_j) > - \sum\limits_{j=1}^{r} \alpha_j(r) C_m(b_j)   \implies    
\delta_{m;a_1,\dots,a_r} > \delta_{m;b_1,\dots,b_r}.
\end{equation*}
\end{thm}

In the next section, we give some preliminary results that will be used throughout this paper. In Section \ref{section4}, we establish the two asymptotic formulas for $N_m$ and $B_m(a,b)$. We also study the order of magnitude of $B_m(a,b)$,  proving Theorem \ref{FirstMomBm}. In Section \ref{section5}, we investigate properties of the Fourier and Laplace Transforms of $\mu_{m;a_1,\dots,a_r}.$ We derive asymptotic formulas for the densities in the case $r =2$ and $r \geq 3$ in Sections \ref{section6} and \ref{section7} respectively. We also show Theorem \ref{genra212} in Section \ref{section7}. The next section, Section \ref{section8}, is devoted to constructing explicitly biased races and proving Theorems \ref{biasedQuadratic}, \ref{biasedRaces} and \ref{generMartin}. In Section \ref{section9}, we find a criteria for $m$-extremely biased races establishing Theorem \ref{criteriaExtreme}. Finally, in Section \ref{section10} we give some examples of races where LI is false.

\section{Preliminaries}\label{covar}

\subsection{Background for Function Fields}

\label{background}
Recall that for any positive integer $n$, the prime counting function $\pi_q(n)$ is defined by 
\[\pi_q(n)=\#\{a \in \mathcal{P}_q \,\,|\,\,\deg(a)=n \}.\]
The prime number theorem for polynomials gives the following estimate about $\pi_q(n)$ (cf. \cite[Theorem 2.2]{rosen}):
\begin{equation}\label{PNT2}
\pi_q(n)=\frac{q^n}{n}+O\left(\frac{q^{n/2}}{n}\right).
\end{equation}
Let $m$ be an element of $\mathbb{F}_{q}[T]$ of positive degree.  From \cite[Chapter 4]{rosen}, we say that a function $\chi: \mathbb{F}_{q}[T] \to\mathbb{C}$ is a Dirichlet character modulo $m$ if it satisfies the following 
\begin{enumerate}
\item $\chi(a+bm)=\chi(a)$ for all $a,b\in\mathbb{F}_{q}[T]$,
\item $\chi(a)\chi(b)=\chi(ab)$ for all $a,b\in\mathbb{F}_{q}[T]$,
\item $\chi(a)\neq0 \Leftrightarrow (a,m)=1$.
\end{enumerate}
We recall the orthogonality relations for Dirichlet characters. Let $\chi$, $\psi$ be two Dirichlet characters modulo $m$ and $a$ and $b$ two elements in $\mathbb{F}_q[T]$ relatively prime to $m$. Then 
\begin{equation}
\label{firstOrthogo}
\begin{aligned}
\sum\limits_{ \chi \bmod m}^{} \chi(a) \overline{\chi(b)}= \phi(m) \delta(a,b),
\end{aligned}
\end{equation}
and
\begin{equation}
\label{secondOrthogo}
\begin{aligned}
\sum\limits_{c \in (\mathbb{F}_{q}[T]/m \mathbb{F}_{q}[T])^{*}}^{} \chi(c) \overline{\psi(c)}= \phi(m) \delta^{'}(\chi,\psi),
\end{aligned}
\end{equation}
where $\delta(a,b) = \left\{
    \begin{array}{ll}
        1 & \text{if } a \equiv b \bmod m, \\
        0 & \text{otherwise.}
    \end{array}
\right.$ and $\delta^{'}(\chi,\psi) = \left\{
    \begin{array}{ll}
        1 & \text{if } \chi = \psi, \\
        0 & \text{otherwise.}
    \end{array}
\right.$
\\

The Dirichlet $L$-function associated to a Dirichlet character $\chi$ modulo $m$ is defined by 
$$
L(s,\chi) = \sum\limits_{f \in \mathcal{M}_{q}}^{} \frac{\chi(f)}{|f|^s} = \prod_{P \in \mathcal{P}_{q}}\left(1-\frac{\chi(P)}{|P|^s}\right)^{-1} 
 \text{ for } \Re{(s)} > 1.
$$
We know from \cite[Proposition $4.3$]{rosen} that for a non-principal Dirichlet character $\chi$ modulo $m$, the Dirichlet $L$-function $L(s,\chi)$ is a polynomial in $q^{-s} = u$ of degree at most $\deg(m) - 1$. This shows that if $\chi$ is non-principal, then $L(s,\chi)$ can be analytically continued to an entire function on all of $\mathbb{C}$. Weil proved the Riemann Hypothesis for function fields over finite fields. 
We make the following change of variable $u =q^{-s}$ and we define 
$\mathcal{L}(u,\chi) := L(s,\chi).$ We recall Proposition $6.4$ of \cite{BCha} (which follows from Weil's Theorem \cite{weil1948varietes}).

\begin{prop}[{\cite[Proposition $6.4$]{BCha}}]
\label{propCha}
Let $m \in \mathcal{M}_{q}$ of degree $M \geq 1$. Let $\chi^{*}$ be the primitive Dirichlet character modulo a polynomial $m(\chi^{*})$ which induces a non-principal Dirichlet character $\chi$ modulo $m$. Also, let $M(\chi^{*})$ be the degree of $m(\chi^{*})$. Then we have \footnote{Note that \cite[Proposition $6.4$]{BCha} contains a typo, we have $\mathcal{L}(u,\chi) = \mathcal{L}(u,\chi^{*}) 
\prod\limits_{\substack{P|m \\ P \nmid m(\chi^{*})}} (1 - \chi^{*}(P) u^{\deg(P)})$.} 
\begin{equation}
\begin{aligned}
\label{sameInverse}
\mathcal{L}(u,\chi) &= \mathcal{L}(u,\chi^{*}) 
\prod\limits_{\substack{P|m \\ P \nmid m(\chi^{*})}} (1 - \chi^{*}(P) u^{\deg(P)}).
\end{aligned}
\end{equation}
 If $\chi^{*}$ is even ($\chi^{*}(-1) = 1$), then 
\begin{equation}
\begin{aligned}
\label{even}
\mathcal{L}(u,\chi^{*}) &= (1 - u) \prod\limits_{i=1}^{M(\chi^{*}) - 2} (1 - \gamma_{\chi_i} u)  ,
\end{aligned}
 \end{equation}
and, otherwise,
\begin{equation}
\label{odd}
\begin{aligned}
\mathcal{L}(u,\chi^{*}) &= \prod\limits_{i=1}^{M(\chi^{*}) - 1} (1 - \gamma_{\chi_i} u),  
\end{aligned}
\end{equation}
for some complex numbers $\gamma_{\chi_{i}}$ with $\left| \gamma_{\chi_i} \right| = \sqrt{q}.$
\end{prop}
These $\gamma_{\chi_{i}}$'s are called the inverse zeros of modulus $\sqrt{q}$ of the Dirichlet $L$-function associated to the character $\chi$.
\begin{rem}
\label{eqInverses}
Note that it follows from  \eqref{sameInverse} that the inverse zeros of the Dirichlet $L$-function associated to a character $\chi$ modulo $m$ are the same as the inverse zeros of the Dirichlet $L$-function associated to a character $\chi^{*}$ modulo $m(\chi^{*})$. 
\end{rem}
\subsection{Estimate for sums over irreducible monic polynomials}
Let $\Lambda(m) = \log |P|$ if $m = P^{t}$, a prime power, and $0$ otherwise. 
\begin{lem}[Mertens Theorem in function fields]
Let $\chi$ be a non-principal Dirichlet character modulo $m$, and let $N \in \mathbb{N^{*}}$. Then 
\begin{equation*}
\label{mertens}
\begin{aligned}
\sum\limits_{|P| \leq N}^{} \frac{\Lambda(f)}{|f|} &=    \log N  + O(1).
\end{aligned}
\end{equation*}
\end{lem}

\begin{proof}
Combining the facts that 
\begin{equation*}
\begin{aligned}
\sum\limits_{|f| \leq N}^{} \frac{\Lambda(f)}{|f|} &= \sum\limits_{k=1}^{\lfloor \frac{\log N}{\log q} \rfloor} \frac{1}{q^k} \sum\limits_{\deg(f) = k}^{} \Lambda(f),   
\end{aligned}
\end{equation*}
and 
\begin{equation*}
\begin{aligned}
\sum\limits_{\deg(f) = k}^{} \Lambda(f) &=  \sum\limits_{\deg(P) = k}^{} \log |P| +  O( k q^{k/2})\\
&= k (\log q) \pi_q(k) + O( k q^{k/2}) = (\log q) q^k + O( k q^{k/2}),
\end{aligned}
\end{equation*}
we deduce our estimate. 
\end{proof}

\begin{lem}
Let $m \in \mathcal{M}_{q}$ be of large degree. Then 
\begin{equation}
\label{majorSom}
\begin{aligned}
\sum\limits_{P | m}^{} \frac{\log |P|}{|P| - 1} &\ll  \log \log_q |m|,  
\end{aligned}
\end{equation}
and
\begin{equation}
\label{majorSomme2}
\begin{aligned}
\frac{|m|}{\phi(m)} &\ll \log_q |m|.   
\end{aligned}
\end{equation}
\end{lem}

\begin{proof}
Let $M = \deg(m)$ and $\alpha \in \mathbb{R}_{+}^{*}$. Let us consider $A_{\alpha} := \#\left\{ P | m \  \text{such that} \ \deg(P) \geq \left[ \frac{M}{\alpha} \right] + 1 \right\}$, clearly $A_{\alpha} \leq \alpha$. In particular, if we take $\alpha = \frac{M}{\log M}$, then $A_{\frac{M}{\log M}} \leq \frac{M}{\log M}.$ Hence, we have 
\begin{equation*}
\begin{aligned}
\sum\limits_{P | m}^{} \frac{\log |P|}{|P| - 1} &= \sum\limits_{\substack{P | m \\ \deg(P) \leq \log M}}^{} \frac{\log |P|}{|P| - 1} +
\sum\limits_{\substack{P | m \\ \deg(P) > \log M}}^{} \frac{\log |P|}{|P| - 1} \\
&\leq  
\sum\limits_{\substack{P | m \\ \deg(P) \leq \log M}}^{} \frac{\log |P|}{|P| - 1} +  \log q \frac{M}{\log M} \frac{\log M}{q^{\log M} - 1}.
\end{aligned}
\end{equation*}
Thus from Lemma \ref{mertens}, we deduce that 
\begin{equation*}
\sum\limits_{P | m}^{} \frac{\log |P|}{|P| - 1} \ll   \log M.
\end{equation*}

We now establish  \eqref{majorSomme2}.
From \cite[Proposition $1.7$]{rosen}, we have that 
\begin{equation*}
\frac{|m|}{\phi(m)} = \prod\limits_{P | m}^{} \left( 1 - \frac{1}{|P|} \right)^{-1}.
\end{equation*}
Hence the desired estimate follows from \cite[Theorem $3$]{RosenMertens}.
\end{proof}
\begin{lem}[{\cite[Lemma $6.3$]{BCha}}]
\label{derivLogL}
Let $\chi$ be a non-principal Dirichlet character modulo $m$. Then 
\begin{equation}
\frac{L'}{L}(1,\chi) = O\left(\log \log_q |m| \right).
\end{equation}
\end{lem}

\subsection{Arithmetic sums over characters in function fields}

In this subsection, we will establish some preliminary arithmetic identities that will be needed later. Most of these results are obtained by using the same arguments as in \cite[Subsection $3.1$]{fiorilliMartin} but we give the proofs for the sake of completeness. 
\begin{lem}
\label{sommeLambda}
Let $m \in \mathcal{M}_{q}$ of degree $\geq 1$. We have 
\begin{equation}
\label{sommeLambda1in2}
\sum\limits_{f | m}^{} \Lambda \left( m/f \right) \phi(f) = \phi(m) \sum\limits_{P | m}^{} \frac{\log |P|}{|P| - 1}. 
\end{equation}
If $s$ is a proper divisor of $m$, then 
\begin{equation}
\label{sommeLambda2in1}
\sum\limits_{f | s}^{} \Lambda \left( m/f \right) \phi(f) = \phi(m) \frac{\Lambda \left( m/s \right)}{\phi \left(m/s \right)}.
\end{equation}
\end{lem}

\begin{proof}
First, note that the sum of $\Lambda \left( m/f \right) \phi(f)$ over $f | m$ is the same as the sum of $\Lambda \left( m/f \right) \phi(f)$ over the divisors $f$ such that $m/f$ is a power of a particular irreducible monic polynomial in $\mathbb{F}_{q}[T].$ We use the notation $P^{r} || m$ to denote that $r$ is the largest integer such that $P^{r} | m$. Then if $P^{r} || m$, we write $m = Q P^{r}$ with $Q \in \mathcal{M}_{q}$ such that $P \nmid Q$. We then get a contribution to the sum only when $f = Q P^{r-k}$ for some $1 \leq k \leq r$. Therefore 
\begin{equation*}
\begin{aligned}
\sum\limits_{f | m}^{} \Lambda \left( m/f \right) \phi(f) &= \sum\limits_{P^{r} || m}^{} \sum\limits_{k=1}^{r} \Lambda \left( P^{k} \right) \phi(Q P^{r-k}), \\
 &= \sum\limits_{P^{r} || m}^{} \phi(Q) \log |P| |P|^{r-1},
\end{aligned}
\end{equation*}
since $(Q,P) = 1$ and $\sum\limits_{k=1}^{r} \phi(P^{r-k}) = |P|^{r-1}$.
Using the fact that  $\phi(Q) = \phi(m) / \phi(P^{r}) \quad \text{(because }  P \nmid Q \text{)}$, we deduce that 
\begin{equation*}
\begin{aligned}
\sum\limits_{f | m}^{} \Lambda \left( m/f \right) \phi(f) &= \sum\limits_{P^{r} || m}^{} \frac{\phi(m)}{\phi(P^{r})} \log |P| |P|^{r-1}
    = \phi(m) \sum\limits_{P^{r} || m}^{} \frac{\log |P|}{|P| - 1}
    = \phi(m) \sum\limits_{P | m}^{} \frac{\log |P|}{|P| - 1}.
\end{aligned}
\end{equation*} 

Let us establish now the second identity. If $m/s$ has at least two distinct irreducible monic polynomial factors in $\mathbb{F}_{q}[T]$, then clearly the right hand side of \eqref{sommeLambda1in2} is equal to 0, and the left hand side is also equal to 0 because $\Lambda(m/f) = 0$ for every divisor $f$ of $s$ and hence all of the terms $\Lambda(m/f)$ will be $0$. Therefore we need only to consider the case where $m/s$ is a prime power. \\
We write $m = Q P^{r}$ with $P \nmid Q$, and since $m/s = P^{t}$ then $s = Q P^{r-t}$. So the only terms that contribute to the sum are those where $f = Q P^{r-k}$ with $t \leq k \leq r$. Therefore by a similar calculation as before we obtain that 
\begin{equation*}
\begin{aligned}
\sum\limits_{f | s}^{} \Lambda \left( m/f \right) \phi(f) = \sum\limits_{k=t}^{r} \Lambda(P^{k}) \phi(Q.P^{r-k}) &= \phi(Q) \log |P| \sum\limits_{k=t}^{r} \phi(P^{r-k}). \\
\end{aligned}
\end{equation*}
Combining the facts that $\phi(Q) = \frac{\phi(m)}{\phi(P^r)}$, $\sum\limits_{k=t}^{r} \phi(P^{r-k}) = |P|^{r-t}$ and $m/s = P^t$ we deduce that 
\begin{equation*}
\begin{aligned}
\sum\limits_{f | s}^{} \Lambda \left( m/f \right) \phi(f) &= \frac{\phi(m)}{\phi(P^r)} \log |P| |P|^{r-t} = \phi(m) \frac{\log |P|}{|P|^{t-1} \left( |P| - 1\right)}
=
\phi(m) \frac{\Lambda(m/s)}{\phi(m/s)}.
\end{aligned}
\end{equation*} 
\end{proof}

\begin{pro} 
\label{formuleExplicite}
Let $m \in \mathcal{M}_{q}$ of degree $\geq 1$. Let $\chi^{*}$ be a primitive Dirichlet character modulo $m(\chi^{*})$ inducing a non principal character $\chi$ modulo $m.$ Then 
\begin{equation*}
\label{firstTool}
\allchisum \log \left| m\left(\chi^{*} \right)\right| = \phi(m) \left( \log |m| - \sum\limits_{P | m}^{} \frac{\log |P|}{|P| - 1}\right),
\end{equation*}
while if $a\not\equiv 1 \mod m$ is a reduced residue, then 
\begin{equation*}
\label{secondTool}
\allchisum  \chi(a) \log \left| m\left(\chi^{*} \right)\right| = - \phi(m) \frac{\Lambda \left( m/ \left( m, a-1 \right) \right)}{\phi \left( m/ \left( m, a-1 \right) \right)}.
\end{equation*}
\end{pro}

\begin{proof}
To prove these two results, we claim that it suffices to show that for any reduced residue $a \bmod m$ we have 
\begin{equation}
    \label{2in1}
    \allchisum \chi(a) \log \left| m(\chi^{*}) \right| = \log |m| \allchisum \chi(a) - \sum\limits_{f | m}^{} \Lambda(m/f) \sum\limits_{\chi \bmod f}^{} \chi(a).
\end{equation}
In fact, if $a \equiv 1 \bmod m$ we obtain by using Identity \eqref{2in1} that 
\begin{equation*}
\begin{aligned}
\allchisum \log \left|m(\chi^{*}) \right| &= \phi(m) \log |m| - 
\sum\limits_{f | m}^{} \Lambda(m/f) \phi(f).
\end{aligned}
\end{equation*}
Thus the first part of our proposition follows from \eqref{sommeLambda1in2}. \\
Otherwise, if $a\not\equiv 1 \mod m$ we have by the orthogonality relation \eqref{firstOrthogo} that  $\sum\limits_{\chi\mod m}^{} \chi(a) = 0$. Thus from \eqref{2in1}
we deduce that 
\begin{equation*}
\allchisum \chi(a) \log \left|m(\chi^{*}) \right| = - \sum\limits_{f | m}^{} \Lambda(m/f) \sum\limits_{\chi \bmod f}^{} \chi(a)  =  - \sum\limits_{f | m}^{} \Lambda(m/f) \phi(f) \iota_{f}(a),    
\end{equation*}
where $\iota_{f}(a) = 1$, if $a \equiv 1 \bmod f$, and equal $0$ otherwise. Therefore 
\begin{equation*}
\allchisum \chi(a) \log \left|m(\chi^{*}) \right| = - \sum\limits_{f | (m,a-1)}^{} \Lambda(m/f) \phi(f).  
\end{equation*}
Since $(m,a-1)$ is a proper divisor of $m$ (because $a\not\equiv 1 \mod m$) then the second result follows from \eqref{sommeLambda2in1}. 

Now, let us prove \eqref{2in1}. Given a Dirichlet character $\chi$ modulo $m$ and $f | m$, we know that 
\begin{equation*}
     \chi \text{ is induced by a character } \chi^{*} \bmod f
     \Longleftrightarrow f  \text{ is a multiple of } m(\chi^{*}). 
\end{equation*}
Then 
\begin{equation*}
    \sum\limits_{f | m}^{} \Lambda(m/f) \sum\limits_{\chi \bmod f}^{}  \chi(a) = \allchisum \chi(a) \sum\limits_{\substack{ f | m \\ m(\chi^{*}) | f}}^{} \Lambda(m/f).
\end{equation*}
So by making the change of variables $c = m/f$ we have 
\begin{equation*}
    \sum\limits_{f | m}^{} \Lambda(m/f) \sum\limits_{\chi \bmod f}^{}  \chi(a) = \allchisum \chi(a) \sum\limits_{c | m / \left( m \left( \chi^{*} \right) \right)}^{} \Lambda(c). 
\end{equation*}
Moreover, we know that 
\begin{equation*}
\sum\limits_{c | m / \left( m \left( \chi^{*} \right) \right)}^{} \Lambda(c) = \log \left| \frac{m}{m\left( \chi^{*} \right)} \right|.
\end{equation*}
Therefore 
\begin{equation*}
\sum\limits_{f | m}^{} \Lambda(m/f) \sum\limits_{\chi \bmod f}^{}  \chi(a) = \allchisum \chi(a) \log \left| \frac{m}{m(\chi^{*})} \right|= \log |m| \allchisum \chi(a) - \allchisum \chi(a) \log \left| m \left( \chi^{*} \right) \right|.
\end{equation*}
Hence, the result \eqref{2in1} follows.
\end{proof}
We end this section by establishing Lemma \ref{varAlea}.
\begin{proof}[Proof of Lemma \ref{varAlea}]
First, note that $\ex(X_{m;a})=-C_m(a) \ex(X^{'}) =- \frac{1}{2} \left( \frac{\sqrt{q}}{q-1} + \frac{q}{q-1}\right) C_m(a)$ since $\ex(U(\gamma_{\chi}))=0$ for all $\gamma_{\chi}$. Thus $\textup{Cov}_{m;a_1,\dots,a_r}(j,k)$ equals
\begin{align*}
&\ex\left(\Big(X_{m;a_j}+\frac{1}{2} \left( \frac{\sqrt{q}}{q-1} + \frac{q}{q-1}\right) C_m(a_j) \Big)\Big(X_{m;a_k}+\frac{1}{2} \left( \frac{\sqrt{q}}{q-1} + \frac{q}{q-1}\right) C_m(a_k)\Big)\right).
\end{align*} 
We define
\begin{align*}
Y_{a_j,a_k} &= \ex\left(\sum_{\substack{\chi\bmod m \\ \chi\neq \chi_0 }}\sum_{\Im(\gamma_{\chi})>0} \sum_{\substack{\psi\bmod m \\ \psi\neq \chi_0}}\sum_{\Im(\widetilde{\gamma}_{\psi})>0} \frac{\big(\chi(a_j)U(\gamma_{\chi})+ \overline{\chi(a_j)U(\gamma_{\chi})}\big)\big(\psi(a_k)U(\widetilde{\gamma}_{\psi})+ \overline{\psi(a_k)U(\widetilde{\gamma}_{\psi})}\big) |\gamma_{\chi} \widetilde{\gamma}_{\psi} |}{\left| (\gamma_{\chi} -1) (\widetilde{\gamma}_{\psi} -1) \right|} \right).
\end{align*}
Therefore 
\begin{align*}
\textup{Cov}_{m;a_1,\dots,a_r}(j,k) &= C_{m}(a_j) C_{m}(a_k) \ex((X^{'})^{2}) + Y_{a_j,a_k} - \left( \frac{\sqrt{q}}{q-1} + \frac{q}{q-1} \right) C_{m}(a_j) C_{m}(a_k) \ex(X^{'}) \\
&+ \frac{1}{4} \left( \frac{\sqrt{q}}{q-1} + \frac{q}{q-1} \right)^{2} C_{m}(a_j) C_{m}(a_k).
\end{align*}
Since $\ex(U(\gamma_{\chi})U(\widetilde{\gamma}_{\psi}))=0$ for all $\gamma_{\chi},\widetilde{\gamma}_{\psi}$  and
$$\ex\left(U(\gamma_{\chi})\overline{U(\widetilde{\gamma}_{\psi})}\right)=\begin{cases} 1& \text{ if } \chi=\psi \text{ and } \gamma_{\chi}=\widetilde{\gamma}_{\psi}\\
0 & \text{ otherwise},\end{cases}$$
we deduce that
\begin{align*}
\textup{Cov}_{m;a_1,\dots,a_r}(j,k) &= \sum\limits_{\substack{\chi  \bmod m \\ \chi \ne \chi_{0}}}^{} \sum\limits_{\Im{\left( \gamma_{\chi} \right)} > 0}^{} (\chi(a_j/a_k) + \chi(a_k/a_j)) \left| \frac{\gamma_{\chi}}{\gamma_{\chi} - 1}\right|^{2} \\
&+ \frac{1}{2} \left( \frac{q}{(q-1)^2} + \frac{q^2}{(q-1)^2} \right) C_{m}(a_j) C_{m}(a_k)  - \frac{1}{4} \left( \frac{\sqrt{q}}{q-1} + \frac{q}{q-1} \right)^{2} C_{m}(a_j) C_{m}(a_k).
\end{align*}
Thus
\begin{align*}
\textup{Cov}_{m;a_1,\dots,a_r}(j,k) = \sum_{\substack{\chi\bmod m \\ \chi\neq \chi_0 }}\sum_{\Im(\gamma_{\chi})>0} \left( \chi\left(a_j/a_k\right)+\chi\left(a_k/a_j\right) \right) \left| \frac{\gamma_{\chi}}{\gamma_{\chi} - 1}\right|^{2} + \frac{1}{4} \left( \frac{q}{q-1} - \frac{\sqrt{q}}{q-1} \right)^{2} C_{m}(a_j) C_{m}(a_k),
\end{align*}
which implies the result.
\end{proof}

\section{Asymptotic formulas for $N_{m}$ and $B_{m}(a,b)$} 

\label{section4}
The purpose of this section is to derive asymptotic formulas for $N_{m}$ and $B_{m}(a,b)$ and deduce related consequences. To achieve this goal, we use the following definition.
\begin{defn}
For any non-principal Dirichlet character $\chi \bmod m$, we define 
$$I(\chi) := \frac{1}{2}\sum\limits_{\gamma_{\chi}}^{} \left| \frac{\gamma_{\chi}}{\gamma_{\chi} - 1} \right|^{2}.$$
\end{defn}
We mention that Cha has already established in the proof of  \cite[Theorem $6.5$]{BCha} the following asymptotic formula for $I(\chi)$:
\begin{equation}
\label{asymptCha}
I(\chi) = \frac{q}{2(q-1)} M(\chi^{*}) + O \left( \log M(\chi^{*}) \right).
\end{equation}
It turns out that this is not enough for our purposes. Instead we manage to prove an exact formula of $I\left( \chi \right)$ without error terms. We then adapt the techniques used in \cite{lamzouri} to the function field setting. This enables us later to prove asymptotic formulas for $N_{m}$ and $B_{m}(a,b)$ and to show important estimates. 
\begin{pro}
\label{formuleI}
Let $m \in \mathcal{M}_{q}$ of degree $\geq 1$. Let $\chi^{*}$ be the primitive Dirichlet character modulo a polynomial $m(\chi^{*})$ which induces a non-principal Dirichlet character $\chi$ modulo $m$. Also, let $M(\chi^{*})$ be the degree of $m(\chi^{*})$. Then 
\begin{equation*}
\begin{aligned}
2 I(\chi) := \sum\limits_{\gamma_{\chi}}^{} \left| \frac{\gamma_{\chi}}{\gamma_{\chi} - 1} \right|^{2} &= 
\frac{q}{q-1} M(\chi^{*})  + \frac{2 q}{q-1} \frac{1}{\log q} \Re{\left(\frac{L^{'}}{L}(1,\chi^{*}) \right)} \\
&\quad-  \frac{  q^2 + q}{2(q-1)^2} \chi^{*}(-1) - \frac{  3 q^2 -  q}{2(q-1)^2}.
\end{aligned}
\end{equation*}
\end{pro}

\begin{proof}
\textbf{-} Let us first consider the case $\chi^{*}$ is odd. \\
Define $G(u,\chi^{*}) = \mathcal{L}(u,\chi^{*}) \overline{\mathcal{L}(\overline{u},\chi^{*})}$ so by \eqref{odd}, it follows that 
\begin{equation*}
G(u,\chi^{*}) = \prod\limits_{i=1}^{M(\chi^{*}) -1} (1 - \gamma_{\chi_i} u) (1 - \overline{\gamma_{\chi_i}} u).     
\end{equation*}
Therefore 
\begin{equation}
\label{res1}
- \frac{G'}{G}(1,\chi^{*}) = \sum\limits_{i=1}^{M(\chi^{*}) - 1}   \left(  \frac{ \gamma_{\chi_{i}}}{1 - \gamma_{\chi_{i}} } +
\frac{\overline{\gamma_{\chi_{i}}}}{1 - \overline{\gamma_{\chi_{i}}}}
\right) = \sum\limits_{i=1}^{M(\chi^{*}) - 1}   \frac{\gamma_{\chi_i} + \overline{\gamma_{\chi_i}}}{\left| \gamma_{\chi_i} - 1 \right|^{2}} - 4 I(\chi^{*}).   
\end{equation}
Moreover, we know that 
\begin{equation}
\label{res2}
M(\chi^{*}) - 1 + \sum\limits_{i=1}^{M(\chi^{*}) - 1}   \frac{\gamma_{\chi_i} + \overline{\gamma_{\chi_i}}}{\left| \gamma_{\chi_i} - 1 \right|^{2}} =\sum\limits_{i=1}^{M(\chi^{*}) - 1}   
\frac{1 + q}{\left| \gamma_{\chi_i} - 1 \right|^{2}}
= \frac{2 (1+q)}{q} I(\chi^{*}).
\end{equation}
Combining \eqref{res1} and \eqref{res2} it follows that 
\begin{equation}
\label{finalRes}
I(\chi^{*}) = \frac{q}{2(q-1)} \left( \frac{G'}{G}(1,\chi^{*}) - (M(\chi^{*}) - 1) \right).
\end{equation}
Using that $\gamma_{\chi} \gamma_{\overline{\chi}} = q$ together with \eqref{odd}, it is easy to derive the following functional equation 
\begin{equation*}
G(u,\chi^{*}) = q^{(M(\chi^{*}) - 1)} u^{2 (M(\chi^{*}) - 1)} G(1/qu, \chi^{*}).     
\end{equation*}
Thus 
\begin{equation}
\label{derivLogG}
\frac{G'}{G}(1,\chi^{*}) = 2 \left( M(\chi^{*}) - 1 \right) - \frac{1}{q} \frac{G'}{G}(1/q,\chi^{*}) = 2 \left( M(\chi^{*}) - 1 \right) - \frac{1}{q} \left( \frac{\mathcal{L}'}{\mathcal{L}}(1/q,\chi^{*}) + \overline{\frac{\mathcal{L}'}{\mathcal{L}}(1/q, \chi^{*})} \right).      
\end{equation}
Moreover, by using the fact that $u = q^{-s}$ we get
\begin{equation}
\label{changeVar}
\frac{-1}{q} \frac{\mathcal{L}'}{\mathcal{L}}(1/q,\chi^{*}) = \frac{1}{\log q} \frac{L'}{L}(1,\chi^{*}).    
\end{equation}
Therefore, by combining \eqref{finalRes}, \eqref{derivLogG} and \eqref{changeVar} we deduce that 
\begin{equation}
\label{resultOdd}
I(\chi^{*}) = \frac{q}{q-1} \left( \frac{M(\chi^{*}) - 1}{2} \right) + \frac{q}{2 (q-1)} \frac{1}{\log q} \left( \frac{L'}{L}(1,\chi^{*}) + \overline{\frac{L'}{L}(1,\chi^{*})} \right).  
\end{equation} 
\textbf{-} Now, let us consider the case that $\chi^{*}$ is even. \\
We define 
\begin{equation*}
H(u,\chi^{*}) = \prod\limits_{i=1}^{M(\chi^{*}) - 2} ( 1 - \gamma_{\chi_i} u) ( 1 - \overline{\gamma_{\chi_i}} u).      
\end{equation*}
By the same argument as in \eqref{finalRes}, we find that 
\begin{equation}
\label{formulaEven}
I(\chi^{*}) = \frac{q}{2(q-1)} \left( \frac{H'}{H}(1,\chi^{*}) - (M(\chi^{*}) - 2) \right).    
\end{equation}
We can check easily that 
\begin{equation*}
\frac{H'}{H}(1,\chi^{*}) = 2 \left( M(\chi^{*}) - 2 \right) - \frac{1}{q} \frac{H'}{H}(1/q,\chi^{*}).
\end{equation*}
Using \eqref{even}, we have $H(u,\chi^{*}) (1-u)^{2} = \mathcal{L}(u,\chi^{*}) \overline{\mathcal{L}(\overline{u},\chi^{*} )}$, and thus  
\begin{equation*}
\frac{H'}{H} (1/q, \chi^{*}) = \frac{\mathcal{L}'}{\mathcal{L}}(1/q, \chi^{*}) + \overline{\frac{\mathcal{L}'}{\mathcal{L}}(1/q, \chi^{*})} + \frac{2}{1 - \frac{1}{q}}.   
\end{equation*}
Therefore 
\begin{equation}
\label{derivHformula}
\frac{H'}{H} (1,\chi^{*}) = 2 \left( M(\chi^{*}) - 2 \right)
- \frac{1}{q} \left( \frac{\mathcal{L}'}{\mathcal{L}}(1/q,\chi^{*}) + \overline{\frac{\mathcal{L}'}{\mathcal{L}}(1/q,\chi^{*})} \right) - \frac{2}{q-1}.
\end{equation}
Hence it follows from \eqref{formulaEven} and \eqref{derivHformula} that 
\begin{equation}
\label{secondRes}
I(\chi^{*}) = \frac{q}{q-1} \left( \frac{M(\chi^{*}) - 2}{2} \right) + \frac{q}{2 (q-1)} \frac{1}{\log q} \left( \frac{L'}{L}(1,\chi^{*}) + \overline{\frac{L'}{L}(1,\chi^{*})} \right) - \frac{q}{(q-1)^2}.
\end{equation}
\\
By combining \eqref{finalRes} and \eqref{secondRes} we deduce that 
\begin{equation*}
I(\chi^{*}) = \frac{q}{q-1} \frac{M(\chi^{*}) - 1}{2} + \frac{q}{q-1} \frac{1}{\log q} \Re\left({\frac{L'}{L} (1,\chi^{*})} \right) + \frac{1}{2} \left( \chi^{*}(-1) + 1 \right) \left( \frac{-q}{2(q-1)} - \frac{q}{(q-1)^2} \right).    
\end{equation*}
Hence by using Remark \ref{eqInverses}, Proposition \ref{formuleI}  follows. 
\end{proof}
The asymptotic formula of $N_{m}$ follows then directly by using the previous proposition. 
\begin{lem} 
\label{formuleNm}
Let $m \in \mathcal{M}_{q}$ be of degree $M$, and assume LI. Then 
\begin{equation*}
\begin{aligned}
N_m &= \frac{q}{q-1} \phi(m) M  + O\left( \phi(m) \log M \right).
\end{aligned}
\end{equation*}
\end{lem}

\begin{proof}
For $\chi \ne \chi_{0}$, since $\chi^{*}(-1) = \chi(-1)$, then it follows from Proposition \ref{formuleI} that 
\begin{equation*}
2 I(\chi) = \frac{q}{q-1} M\left( \chi^{*} \right) + \frac{2 q}{q-1} \frac{1}{\log q} \Re{\left(\frac{L^{'}}{L}(1,\chi^{*})\right)} - \frac{ q^2 + q}{2(q-1)^2} \chi(-1) - \frac{3 q^2 - q}{2(q-1)^2}.
\end{equation*}
Since $N_{m} = 2 \sum\limits_{\substack{\chi  \bmod m \\ \chi \ne \chi_{0}} }^{} \sum\limits_{\Im{\left( \gamma_{\chi} \right) > 0}}^{} \left| \frac{\gamma_{\chi}}{\gamma_{\chi} - 1} \right|^{2} = \sum\limits_{\substack{\chi  \bmod m \\ \chi \ne \chi_{0}}}^{} \left( I(\chi) + I(\overline{\chi}) \right)$, then we have 
\begin{equation*}
\begin{aligned}
N_{m} =  \frac{q}{q-1} \sum\limits_{\substack{ \chi  \bmod m \\ \chi \ne \chi_{0}}}^{} M\left( \chi^{*} \right) + \frac{2 q}{q-1} \frac{1}{\log q} \sum\limits_{\substack{ \chi  \bmod m \\ \chi \ne \chi_{0}}}^{} \Re{\left(\frac{L^{'}}{L}(1,\chi^{*})\right)}  
+ \frac{q^2 + q}{2(q-1)^{2}} -  \frac{ 3 q^2 - q}{2(q-1)^2} \left( \phi(m) - 1 \right).
\end{aligned}
\end{equation*}
From Lemma \ref{derivLogL} and Proposition \ref{formuleExplicite} we finally deduce that 
\begin{equation*}
N_{m} = \frac{q}{q-1} \phi(m) M + O\left( \phi(m) \log M \right). 
\end{equation*}
\end{proof}
\begin{rem}
Let $m \in \mathcal{M}_{q}$ and $\chi_{m}$ be the primitive quadratic Dirichlet character modulo $m$. We denote by $mult(\chi_{m})$ the multiplicity of $\pm \sqrt{q}$ as an inverse zero of $\mathcal{L}(u,\chi_{m})$.
Wanlin LI showed in \cite{wanli} the existence of a family of polynomials $m \in \mathcal{M}_{q}$ satisfying $mult(\chi_{m}) >0$. This is the reason why we assume LI in the previous lemma. 
\end{rem}

Before giving an asymptotic formula for $B_{m}(a,b)$ (Proposition \ref{forExpBm}), we begin by establishing some useful lemmas.   
\begin{lem}  [\cite{BCha}, p $1371$] 
\label{formuleCha}
Let $\chi$ be a non-principal character modulo $m$ (of degree $M$), then for all  $x \geq 1$ we have 
\begin{equation*}
\frac{L'}{L} (1, \chi) = - \sum\limits_{f \in \mathcal{M}_{q}}^{} \frac{\chi(f) \Lambda(f)}{|f|} \exp{\left( - q^{\deg(f)} / x \right)} + O\left( \frac{M}{x^{1/4}} \right).
\end{equation*}
\end{lem}

\begin{lem} 
\label{lemFIM}
Let  $m \in \mathcal{M}_{q}$,  $P \in \mathcal{P}_{q}$, $e\in \mathbb{N^{*}},$ and $r \in \mathbb{F}_{q}[T] $ such that $(r,m) = 1$. \\
If $P \nmid m$ then 
\begin{equation*}
\sum\limits_{\chi  \bmod m}^{} \chi(r) \left( \chi^{*}(P^e) - \chi(P^e) \right) = 0.
\end{equation*}
Otherwise, if $P | m$ then 
\begin{equation*}
\sum\limits_{\chi  \bmod m}^{} \chi(r) 
\left( \chi^{*}(P^e) - \chi(P^e) \right) = \left\{
\begin{array}{ll}
     \phi(m/P^{\nu}) &  \text{if } r P^e \equiv 1  \bmod m/P^{\nu} ,\\
     0 & \text{otherwise,} 
\end{array}
\right.
\end{equation*}
where $\nu \geq 1$ is the integer such that $P^{\nu} || m.$
\end{lem} 

\begin{proof}
Since $P \nmid m$, we have for every character $\chi \bmod m$  that $\chi^{*}(P^{e}) = \chi(P^{e})$. Hence the first assertion follows. \\
Now, if $P | m$ then $\chi(P^e) = 0$ for every $\chi$. Therefore using the
fact that $\chi(r) = \chi^{*}(r)$ for every $\chi \bmod m$ due to the hypothesis that $(r,m) = 1$, we have 
\begin{equation*}
\sum\limits_{\chi \bmod m}^{} \chi(r) 
\left( \chi^{*}(P^e) - \chi(P^e) \right) = \sum\limits_{\chi \bmod m}^{} \chi(r) \chi^{*}(P^e) = \sum\limits_{\chi \bmod m}^{}
\chi^{*}(r P^{e}).
\end{equation*}
We also know that $\chi^{*}(P^e) = 0$ for any character $\chi \bmod m$ such that 
$P | m(\chi^{*})$. Thus 
\begin{equation*}
\sum\limits_{\chi \bmod m}^{}
\chi^{*}(r P^{e}) = \sum\limits_{\substack{\chi \bmod m \\ m(\chi^{*}) | m/P^{\nu}}}^{}
\chi^{*}(r P^{e}),
\end{equation*}
since $(P^{e}, m/P^{\nu}) = 1.$ By using the orthogonality relation \eqref{firstOrthogo}, the second result follows.
\end{proof}

\begin{lem} 
\label{elemantaryBm}
Let $m \in \mathcal{M}_{q}$ be of large degree $M$ and $x \geq |m|$ be a real number. Then 
\begin{equation*}
\sum\limits_{\substack{f \in \mathcal{M}_{q} \\ \deg(f) \geq 1 \\ (f,m) = 1}}^{} \frac{\Lambda(f)}{|f|} \exp{\left( - |f|/x \right)} = \log(x) + O\left( \log(\log x)\right). \\
\end{equation*}
\end{lem}

\begin{proof}
By \eqref{majorSom} we have 
\begin{equation*}
\begin{aligned}
\sum\limits_{\substack{f \in \mathcal{M}_{q} \\ \deg(f) \geq 1 \\ (f,m) > 1}}^{} \frac{\Lambda(f)}{|f|} \exp{\left( - |f|/x \right)} &\leq \sum\limits_{P | m}^{} \sum\limits_{k=1}^{\infty} \frac{\log |P|}{|P|^{k}} = 
\sum\limits_{P | m}^{} \frac{\log |P|}{|P| - 1} \ll \log M \ll \log \log x.
\end{aligned}
\end{equation*}
Thus it is enough to evaluate $\sum\limits_{\substack{f \in \mathcal{M}_{q} \\ \deg(f) \geq 1 }}^{} \frac{\Lambda(f)}{|f|} \exp{\left( - |f|/x \right)}.$ \\
We split the above sum into three parts : $|f| > x\log^{2} x$, $x \log \log x < |f| \leq x \log^{2} x$, and $|f| \leq x \log \log x$. 
Since $\exp{\left( - |f|/x \right)} \leq \frac{1}{|f|^{2}}$ for $|f| > x \log^{2} x$ then the contribution of the first part is
\begin{equation*}
\sum\limits_{|f| > x \log^{2} x}^{} \frac{\Lambda(f)}{|f|} \exp{ \left( - |f|/x \right) } \leq  \sum\limits_{|f| > x \log^{2} x}^{} \frac{1}{|f|^{2}} \ll \frac{1}{x}.
\end{equation*}
By using Lemma \ref{mertens} it follows that the contribution of the second part is 
\begin{equation*}
\sum\limits_{x \log \log x < |f| \leq x \log^{2} x}^{} \frac{\Lambda(f)}{|f|} \exp{ \left( - |f|/x \right) } \leq
\frac{1}{\log x} \sum\limits_{|f| \leq x \log^{2} x}^{} \frac{\Lambda(f)}{|f|}  \ll 1.
\end{equation*}
Finally combining the fact that $e^{-t} = 1 + O(t) $ for all $t>0$,
with \eqref{PNT2} and Lemma \ref{mertens} we deduce that the contribution of the last part equals 
\begin{equation*}
\sum\limits_{|f| \leq x \log\log x}^{} \frac{\Lambda(f)}{|f|} \exp{\left( - |f| / x \right)} = \sum\limits_{|f| \leq x \log\log x}^{} \frac{\Lambda(f)}{|f|} + O\left(
\frac{1}{x} \sum\limits_{|f| \leq x \log\log x}^{} \Lambda(f) \right) = \log x + O\left( \log \log x \right).
\end{equation*}
\end{proof}

The following proposition will be very useful in later proofs. 
\begin{pro}  
\label{forExpBm}
Let $m \in \mathcal{M}_{q}$ be of large degree $M$ and $(a,b) \in \mathcal{A}_{2}(m)$, and let $x = \left( \phi(m) M \right)^{4}$. Assume LI, then 
\begin{equation*}
\begin{aligned}
B_m(a,b) &= \frac{8q}{(q-1)\log q} \log \phi(m) - \frac{q}{q-1} \frac{\phi(m)}{\log q} \frac{\Lambda(m/(m,a-b))}{\phi(m/(m,a-b))} -  \frac{q^2 + q}{2(q-1)^2} \phi(m) l_{m}(a,b) \\ &\quad+ \frac{q}{ (q-1) \log q} \left( - \phi(m) \sum\limits_{\substack{1 \leq |f| \leq 2x \log x \\ af \equiv b \bmod m}}^{} \Lambda(f) \frac{\exp{\left( - |f|/ x \right)}}{|f|}  - \phi(m) \sum\limits_{\substack{1 \leq |f| \leq 2x \log x \\ bf \equiv a \bmod m}}^{} \Lambda(f) \frac{\exp{\left( - |f|/ x \right)}}{|f|} \right. \\
&\left. \quad - \phi(m) \sum\limits_{P^v || m}^{}   \sum\limits_{n=1}^{\infty} \sum\limits_{\substack{ 1 \leq e \leq 2 \log x \\ \deg(P^e) = n \\ a P^{e} \equiv b \bmod m/P^v}}^{}  \frac{\log |P|}{|P|^{e+v-1} (|P| - 1)}   \right.\\
&\left.\quad - \phi(m) \sum\limits_{P^v || m}^{}   \sum\limits_{n=1}^{\infty} \sum\limits_{\substack{ 1 \leq e \leq 2 \log x \\ \deg(P^e) = n \\ b P^{e} \equiv a \bmod m/P^v}}^{}  \frac{\log |P|}{|P|^{e+v-1} (|P| - 1)}  \right) + O\left( \log M \right),
\end{aligned}
\end{equation*}
where $l_{m}(a,b) = 1$, if $a+b \equiv 0 \bmod m$, and $0$, otherwise.
\end{pro}

\begin{proof}
We know that $B_{m}(a,b) = 2 \sum\limits_{\substack{\chi  \bmod m \\ \chi \ne \chi_{0}}}^{} \Re{ \left( \chi(a/b) \right)} I(\chi^{*})$. So from Propositions \ref{formuleExplicite} and \ref{formuleI} we have that 
\begin{equation*}
\begin{aligned}
B_m(a,b) &= - \frac{q}{q-1} \frac{\phi(m)}{\log q} \frac{\Lambda(m/(m,a-b))}{\phi(m/(m,a-b))} -   \frac{q^2 +q}{2(q-1)^2} \phi(m) l_{m}(a,b) \\
&\quad+ \frac{q}{ (q-1) \log q} \sum\limits_{\substack{\chi  \bmod m \\ \chi \ne \chi_{0}}}^{} \left( \chi(a/b) + \chi(b/a) \right) \Re{\left(\frac{L'}{L}(1,\chi^{*}) \right)} + \frac{3 q^2 - q}{2(q-1)^2} +  \frac{q^2 + q}{2(q-1)^2}.
\end{aligned}
\end{equation*}
In order to get our result, it is sufficient to evaluate
$\sum\limits_{\substack{\chi  \bmod m \\ \chi \ne \chi_{0}}}^{} \left( \chi(a/b) + \chi(b/a) \right) \frac{L'}{L}(1,\chi^{*}).$ \\
By using Lemma \ref{formuleCha} we find that  
\begin{equation*}
\frac{L'}{L} (1, \chi^{*}) = - \sum\limits_{f \in \mathcal{M}_{q}}^{} \frac{\chi^{*}(f) \Lambda(f)}{|f|} \exp{\left( - q^{\deg(f)} / x \right)} + O\left( \frac{1}{\phi(m) } \right).
\end{equation*}
Therefore $\sum\limits_{\substack{\chi  \bmod m \\ \chi \ne \chi_{0}}}^{} \left( \chi(a/b) + \chi(b/a) \right)\Re{\left( \frac{L'}{L}(1,\chi^{*}) \right)}$ equals
\begin{equation*}
\begin{aligned}
- \Re\left(\sum\limits_{f \in \mathcal{M}_{q}}^{}  \frac{\Lambda(f)}{|f|} \exp{\left( - q^{\deg(f)} / x \right) \sum\limits_{\substack{\chi  \bmod m \\ \chi \ne \chi_{0}}}^{} \left( \chi(a/b) \chi^{*}(f) +  \chi(b/a) \chi^{*}(f)
\right)  }\right) + O(1).
\end{aligned}
\end{equation*}
We know by Lemma \ref{lemFIM} that   
\begin{equation*}
\begin{aligned}
\sum\limits_{\chi \bmod m}^{} \chi(a/b) \chi^{*}(P^e) &= \left\{
\begin{array}{cc}
     \phi(m) & \text{if} \ P \nmid m \ \text{and} \ aP^{e} \equiv b \bmod m,\\
     \phi(m/P^{v}) & \text{if} \ P^{v} || m \ \text{and} \ aP^{e} \equiv b \bmod m/P^{v}, \\
     0 & \text{otherwise.}
\end{array}
\right.
\end{aligned}
\end{equation*}
Moreover, it is clear that $aP^{e} \equiv b \bmod m$ implies that $P \nmid m$, and by using Lemma \ref{elemantaryBm}
 it follows that  $\sum\limits_{\substack{\chi  \bmod m \\ \chi \ne \chi_{0}}}^{} \left( \chi(a/b) + \chi(b/a) \right) \Re \left(\frac{L'}{L}(1,\chi^{*}) \right)$ is 
\begin{equation*}
 8 \log \phi(m) - \phi(m) \sum\limits_{n=1}^{\infty} \sum\limits_{\substack{\deg(f) = n \\ af \equiv b \bmod m }}^{} \Lambda(f) \frac{\exp{\left( - |f|/ x \right)}}{|f|}  
\end{equation*}
\begin{equation*}
- \phi(m) \sum\limits_{n=1}^{\infty} \sum\limits_{\substack{\deg(f) = n \\ bf \equiv a \bmod m }}^{} \Lambda(f) \frac{\exp{\left( - |f|/ x \right)}}{|f|} 
- \sum\limits_{P^v || m}^{} \phi\left( \frac{m}{P^v} \right) \sum\limits_{n=1}^{\infty} \sum\limits_{\substack{ e\geq 1 \\ \deg(P^e) = n \\ a P^{e} \equiv b \bmod m/P^v}}^{}  \frac{\log |P|}{|P|^{e}} \exp{\left( - |P|^{e} /x \right)} 
\end{equation*}
\begin{equation*}
- \sum\limits_{P^v || m}^{} \phi\left( \frac{m}{P^v} \right) \sum\limits_{n=1}^{\infty} \sum\limits_{\substack{ e\geq 1 \\ \deg(P^e) = n \\ b P^{e} \equiv a \bmod m/P^v}}^{}  \frac{\log |P|}{|P|^{e}} \exp{\left( - |P|^{e} /x \right)} + O\left( \log M  \right).
\end{equation*}
Since if $|f| \geq 2 x \log x$ we have $\exp{\left( - |f|/x \right)} \leq \frac{1}{|f|}$, it follows that  
\begin{equation*}
\begin{aligned}
\sum\limits_{\substack{ |f| > 2 x \log x \\ b f \equiv a \bmod m}}^{} \Lambda(f) \frac{\exp{\left( - |f|/x \right)}}{|f|}
+ \sum\limits_{P^v || m}^{} \sum\limits_{n=1}^{\infty} \sum\limits_{\substack{ e > 2 \log x \\ \deg(P^e) = n \\ a P^{e} \equiv b \bmod m/P^v}}^{}  \frac{\log |P|}{|P|^{e}} \exp{\left( - |P|^{e} /x \right)}  &\ll \sum\limits_{|f| > 2 x \log x}^{} \frac{\Lambda(f)}{|f|^2} \ll \frac{1}{|m|^{2}}.
\end{aligned}
\end{equation*}
We notice that $\phi(m/P^{v}) = \phi(m)/ \left( |P|^{v-1} (|P| - 1) \right)$ (because $(P^{v}, m/P^{v}) = 1$) and \\
$1 - \exp{(-t)} \leq 2 t$ for all $t > 0$. Thus, we deduce that 
\begin{equation*}
\begin{aligned}
\sum\limits_{P^{v} || m}^{} \phi\left(\frac{m}{P^{v}} \right) \sum\limits_{n=1}^{\infty} 
\sum\limits_{\substack{ 1 \leq e \leq 2 \log x \\ a P^{e} \equiv b \bmod m/P^v\\ \deg(P^e) = n}}^{} \frac{\log |P|}{|P|^{e}} \left( 1 - \exp{ \left( - |P|^{e} / x \right)}\right)
&\ll \frac{\phi(m)}{x} \sum\limits_{P^{v} || m}^{} 
\sum\limits_{n=1}^{\infty}  \sum\limits_{\substack{ 1 \leq e \leq 2 \log x \\ \deg(P^e) = n \\ a P^{e} \equiv b \bmod m/P^v}}^{} \frac{\log |P|}{|P| - 1} \\
&\ll \phi(m) \frac{\log x}{x} \sum\limits_{P | m}^{}  \frac{\log |P|}{|P| - 1} \ll  \frac{\log M}{(\phi(m) M)^{2}}.
\end{aligned}
\end{equation*}
Hence by combining the above estimates our proposition follows. 
\end{proof}

\begin{rem}
We have that $B_{m}(a,b) < 0$ if $|B_{m}(a,b)| > \frac{9q}{(q-1) \log q} \log \phi(m).$ 
\end{rem}
A direct consequence of Proposition \ref{forExpBm} is the following corollary. 
\begin{cor} 
\label{majorBm}
For any $(a,b) \in \mathcal{A}_{2}(m)$, we have  $|B_m(a,b)| \ll \phi(m).$
\end{cor}
To deduce this corollary, we need the following lemma. 
\begin{lem}
\label{lemElemMajBm}
Let $m \in \mathcal{M}_{q}$ be of large degree $M$, $(a,b) \in \mathcal{A}_{2}(m)$, and denote by $s$ the residue of  $ab^{-1}$ mod $m$ with the least degree. Moreover, put $x = \left( \phi(m) M \right)^{4}$. Then 
\begin{equation*}
\sum\limits_{\substack{ 1 \leq |f| \leq 2x \log x \\ bf \equiv a \bmod m}}^{} \frac{\Lambda(f)}{|f|} \exp{\left( - |f|/x \right) }= \frac{\Lambda(s)}{|s|} + O_q \left( \frac{M^2}{q^{M}} \right). 
\end{equation*}
\end{lem}

\begin{proof}
We know that 
\begin{equation*}
\begin{aligned}
\sum\limits_{\substack{1 \leq |f| \leq 2x \log x \\ f \equiv s \bmod m}}^{} \Lambda(f) \frac{\exp{\left( - |f|/ x \right)}}{|f|} &= \frac{\Lambda(s)}{|s|} \exp{\left( - |s|/x \right)} +  \sum\limits_{\substack{1 \leq |f| \leq 2x \log x \\ f \equiv s \bmod m \\ f \ne s}}^{} \frac{\Lambda(f)}{|f|} \exp{\left( - |f| / x \right)}. 
\end{aligned}
\end{equation*}
Moreover, if $f \ne s$ with $f \equiv s \bmod m$, then $f = s + r m$ with $\deg(r) \geq 0$. Thus 
\begin{equation*}
\sum\limits_{\substack{1 \leq |f| \leq 2x \log x \\ f \equiv s \bmod m \\ f \ne s}}^{} \frac{\Lambda(f)}{|f|} \exp{\left( - |f| / x \right)} \ll_q \frac{M}{q^{M}} \left(
\sum\limits_{0 \leq \deg(r) \leq 24 M}^{} \frac{1}{|r|} \right) \ll_q \frac{M^2}{q^M}.
\end{equation*}
By using the fact that $e^{-t} = 1 + O(t) $ for all $t>0$, we obtain 
\begin{equation*}
\frac{\Lambda(s)}{|s|} \exp{\left( - |s|/x \right)} 
= \frac{\Lambda(s)}{|s|} + O_q \left( \frac{1}{\phi(m)^{4} M^{3}} \right).
\end{equation*}
Hence, the result follows upon collecting the above estimates since we have that $\frac{|m|}{\phi(m)} \ll M$ from \eqref{majorSomme2}. 
\end{proof}
\begin{proof}[Proof of Corollary \ref{majorBm}]
We denote by $s$ the residue of  $ab^{-1}$ mod $m$ of the least degree. First, note that $\frac{\Lambda(s)}{|s|} \leq \frac{\log q}{q}$ for $\deg(s) \geq 1$. 
Furthermore, we have that  $\Lambda(m/(m,a-b))/\phi(m/(m,a-b)) \ne 0$  if and only if $m/(m,a-b) = P^{l}$ with $P \in \mathcal{P}_{q}$, and $l \in \mathbb{N^{*}}.$ Thus, in this case 
\begin{equation*}
\frac{\Lambda(m/(m,a-b))}{\phi(m/(m,a-b)} = \frac{\log |P|}{|P|^{l-1} \left( |P| - 1 \right)} \leq \frac{\log |P|}{|P| - 1} \leq \log 2.
\end{equation*}
Moreover, we know that  
\begin{equation*}
\sum\limits_{P^{\nu} || m}^{} \sum\limits_{n=1}^{\infty} \sum\limits_{\substack{ 1 \leq e \leq 2 \log x \\ b.P^{e} \equiv a \bmod m/P^v\\ \deg(P^{e}) = n }}^{} \frac{\log |P|}{|P|^{e + \nu - 1} \left( |P| - 1 \right)} \leq \sum\limits_{P|m}^{} \frac{\log |P|}{\left( |P| - 1\right)^{2}} \ll 1.
\end{equation*}
Therefore by combining these different estimates with Proposition \ref{forExpBm} and Lemma \ref{lemElemMajBm}, the corollary follows.  
\end{proof}

Another consequence of Proposition \ref{forExpBm} is Theorem \ref{FirstMomBm} which affirms that the lower and upper bounds for the first moment of $|B_{m}(a,b)|$ over pairs of residue classes $(a,b) \in \mathcal{A}_{2}(m)$ have the same order of magnitude. 

\begin{proof}[Proof of Theorem \ref{FirstMomBm}]
We first prove the lower bound of our theorem. By definition of $B_{m}(a,b)$, we have that
\begin{equation*}
\begin{aligned}
\sum\limits_{(a,b) \in \mathcal{A}_{2}(m)}^{} B_m(a,b) &= \sum\limits_{\substack{\chi  \bmod m \\ \chi \ne \chi_{0}}}^{} \sum\limits_{\Im{\left( \gamma_{\chi} \right)} > 0}^{} \sum\limits_{\substack{a \bmod m \\ (a,m) = 1}}^{} \sum\limits_{\substack{b \ne a \bmod m \\ (b,m) = 1}}^{}  (\chi(a/b) + \chi(b/a)) \left| \frac{\gamma_{\chi}}{\gamma_{\chi} - 1}\right|^{2}.
\end{aligned}   
\end{equation*}
Using the orthogonality relations for Dirichlet characters modulo $m$, it is easy to check that 
\begin{equation*}
\sum\limits_{\substack{a \bmod m \\ (a,m) = 1}}^{} \sum\limits_{\substack{b \ne a \bmod m \\ (b,m) = 1}}^{}  (\chi(a/b) + \chi(b/a)) = - 2 \phi(m).
\end{equation*}
Thus we obtain that $\sum\limits_{(a,b) \in \mathcal{A}_{2}(m)}^{} B_{m}(a,b) = - \phi(m) N_{m}$, and since $\left| \mathcal{A}_{2}(m) \right| = \phi(m)^2 - \phi(m) \leq \phi(m)^2$, we have 
\begin{equation*}
\frac{1}{\left| \mathcal{A}_{2}(m) \right|} \sum\limits_{(a,b) \in \mathcal{A}_{2}(m)}^{} \left| B_{m}(a,b) \right| \geq - \frac{1}{\left| \mathcal{A}_{2}(m) \right|} \sum\limits_{(a,b) \in \mathcal{A}_{2}(m)}^{} B_{m}(a,b) = \frac{N_{m}}{\phi(m) - 1}.
\end{equation*}
Moreover, we know that $N_{m} = \phi(m) \left( \frac{q}{q-1}  M + O\left(\log M \right) \right)$, therefore 
\begin{equation*}
\frac{1}{\left| \mathcal{A}_{2}(m) \right|} \sum\limits_{(a,b) \in \mathcal{A}_{2}(m)}^{} \left| B_{m}(a,b) \right| \geq \frac{q}{q-1} M + O\left(\log M \right).
\end{equation*}

Next, it remains to establish the upper bound of our theorem. Let $(a,b) \in \mathcal{A}_{2}(m)$ and $d = (m,a-b)$. Then $a-b = ds$ with $0 \leq \deg(s) \leq \deg(m/d)$ and $(s,m/d) = 1.$ So for any choice of $d$ and $s$ satisfying these conditions there are at most $\phi(m)$ pairs $(a,b) \in \mathcal{A}_{2}(m)$ such that $a-b = ds.$ Thus we have 
\begin{equation*}
\begin{aligned}
\sum\limits_{(a,b) \in \mathcal{A}_{2}(m)}^{} \frac{\Lambda \left(m/(m,a-b) \right)}{\phi\left( m/(m,a-b)\right)}&\leq \phi(m) \sum\limits_{d | m}^{} \frac{\Lambda(m/d)}{\phi(m/d)} \left( \sum\limits_{\substack{0 \leq \deg(s) \leq \deg(m/d) \\ (m/d,s) = 1}}^{} 1 \right) \\
&\leq \phi(m) \sum\limits_{d|m}^{} \Lambda(m/d) = (\log q) \phi(m) M. 
\end{aligned}
\end{equation*}
We choose $x = \left( \phi(m) M \right)^{4}$, then by Lemma \ref{elemantaryBm} it follows that
\begin{equation*}
\begin{aligned}
\sum\limits_{(a,b) \in \mathcal{A}_{2}(m)}^{} \sum\limits_{\substack{1 \leq |f| \leq 2x \log x \\ f \equiv ab^{-1} \bmod m}}^{} \frac{\Lambda(f)}{|f|} \exp{\left( - |f|/x \right)} &= \sum\limits_{\substack{1 \leq |f| \leq 2x \log x \\ 
(f,m) = 1}}^{} \frac{\Lambda(f)}{|f|} \exp{ \left( - |f|/ x \right)} \left( \sum\limits_{ \substack{(a,b) \in \mathcal{A}_{2}(m) \\ a b^{-1} \equiv f \bmod m}}^{} 1 \right) \\
&\leq \phi(m) \sum\limits_{\substack{1 \leq |f| \leq 2x \log x \\ (f,m) = 1}}^{} \frac{\Lambda(f)}{|f|} \exp{\left( - |f|/x \right)}  \\
&\leq \phi(m) \sum\limits_{\substack{f \in \mathcal{M}_{q} \\ \deg(f) \geq 1 \\ (f,m) = 1}}^{} \frac{\Lambda(f)}{|f|} \exp{\left( - |f|/x \right)} \\
&\leq 4 (\log q) \phi(m) M + O\left( \phi(m) \log M \right).
\end{aligned}
\end{equation*}
On the other hand, we obtain 
\begin{equation*}
\begin{aligned}
\sum\limits_{(a,b) \in \mathcal{A}_{2}(m)}^{} \sum\limits_{P^{v} || m}^{} \sum\limits_{\substack{1 \leq e \leq 2 \log x \\ aP^{e} \equiv b \bmod m/P^v \\ \deg(P^e) = n}}^{} \frac{\log |P|}{|P|^{e+v-1} \left( |P| - 1 \right)}  &\leq \phi(m) \sum\limits_{P | m}^{} \frac{\log |P|}{(|P| - 1)^{2}} \ll \phi(m).  
\end{aligned}
\end{equation*}
Hence our theorem follows from Proposition \ref{forExpBm} combined with these different estimates. 
\end{proof}

\section{The Fourier and Laplace transforms of $\mu_{m;a_1,\dots,a_r}$}

\label{section5}The aim of this section is to study properties of the Fourier and Laplace transforms of $\mu_{m;a_1,\dots,a_r}$.  We first recall the explicit formula for the Fourier transform $\hat{\mu}_{m;a_1,\dots,a_r}$ obtained by Cha in  \cite[Theorem $3.4$]{BCha}. Under the assumption of LI, we have that for all $t= \left( t_1,\dots,t_r \right) \in \mathbb{R}^{r}$ 
\begin{equation}
\label{forExplicitCha}
\begin{aligned}
\hat{\mu}_{m;a_1,\dots,a_r}(t) &= \mathcal{B}_{m;a_1,\dots,a_r}(t) \prod\limits_{\substack{\chi  \bmod m \\ \chi \ne \chi_{0} }}^{} \prod\limits_{\Im{\left( \gamma_{\chi}\right)} > 0 }^{} 
J_{0} \left( \left| \frac{2 \gamma_{\chi}}{\gamma_{\chi} - 1} \right|
\left| \sum\limits_{j=1}^{r} \chi(a_j) t_j \right| \right) ,
\end{aligned}    
\end{equation}
where 
$$ J_{0}(z) = \sum\limits_{n=0}^{\infty} \frac{(-1)^n (z/2)^{2n}}{(n!)^2} $$
is the Bessel function of order $0$, and 
\begin{equation*}
\mathcal{B}_{m;a_1,\dots,a_r}(t) = \frac{1}{2} \left[ 
\exp{\left( i \frac{\sqrt{q}}{q-1} \sum\limits_{j=1}^{r} C_{m}(a_j) t_j \right)} + 
\exp{\left( i \frac{q}{q-1} \sum\limits_{j=1}^{r} C_{m}(a_j) t_j \right)}
\right]. 
\end{equation*}
First, we give the following definition. 
\begin{defn} 
For any non-principal Dirichlet character $\chi$ modulo $m$, we define 
\begin{equation*}
F(x,\chi) := \prod_{\Im{(\gamma_{\chi})} > 0}^{} J_{0} \left( 2 x \left| \frac{\gamma_{\chi}}{\gamma_{\chi} - 1} \right|  \right),  
\end{equation*}
for all $x \in \mathbb{R}$.
\end{defn}
\begin{lem}
\label{lemFourierTrans}
Let $\chi$ be a non-principal Dirichlet character modulo $m$. Assume LI, then 
\begin{equation*}
\left| F(x,\chi) F(x,\overline{\chi}) \right| \leq \exp{ \left( - \frac{\log x}{2} (M(\chi^{*}) - 2) \right)},
\end{equation*}
for all $x > 2 \sqrt{q}$.
\end{lem}

\begin{proof}
We know that $\left| J_0(x) \right| \leq \min{\left\{ 1, \sqrt{\frac{2}{\pi x}} \right\}}$ for all $x > 2 \sqrt{q}$, therefore for $x > 2 | \gamma_{\chi}|$ we have 
\begin{equation*}
\begin{aligned}
\left| F(x,\chi) F(x,\overline{\chi}) \right| &\leq \prod_{\gamma_{\chi}}^{} J_{0} \left( 2 x \left| \frac{\gamma_{\chi}}{\gamma_{\chi} - 1} \right|  \right) \leq \prod_{\gamma_{\chi}}^{} \left| \frac{\gamma_{\chi} - 1}{\gamma_{\chi}}\right|^{\frac{1}{2}} \frac{1}{\sqrt{\pi x}}.
\end{aligned}
\end{equation*}

Moreover, since $q > 2$ we have $\left| \frac{\gamma_{\chi} - 1}{\gamma_{\chi}} \right| \leq 1 + \frac{1}{|\gamma_{\chi}|} < 2$. 
Thus 
\begin{equation*}
\begin{aligned}
\left| F(x,\chi) F(x,\overline{\chi}) \right| &\leq \prod_{\gamma_{\chi}}^{} \sqrt{\frac{2}{\pi x}} \leq \prod_{\gamma_{\chi}}^{}  \frac{1}{\sqrt{x}} \leq  \exp{\left( - \frac{\log x}{2} \#\{ \gamma_{\chi} \} \right)}.
\end{aligned}
\end{equation*}
Hence it follows from Proposition \ref{propCha} that 
\begin{equation*}
\left| F(x,\chi) F(x,\overline{\chi}) \right| \leq \exp{\left( - \frac{\log x}{2} ( M(\chi^{*}) - 2) \right)}.
\end{equation*}
\end{proof}

Now, we are able to give a decreasing upper bound for $\hat{\mu}_{m;a_1,\dots,a_r}(t)$ in the following proposition. 
\begin{pro} 
\label{majorMesur}
Let $r \geq 2$ be a fixed integer, $m \in \mathcal{M}_{q}$ be of large degree $M$ and let $\epsilon \in ]0,\frac{1}{4}[$ be a real number. Assume LI. Then for all $(a_1,\dots,a_r) \in \mathcal{A}_{r}(m)$ we have 
\begin{equation*}
\begin{aligned}
\left| \hat{\mu}_{m;a_1,\dots,a_r}(t_1,\dots,t_r) \right| &\leq
\exp{\left( - c_1(r) \epsilon^{2} \phi(m) M \right)},
\end{aligned}
\end{equation*}
for $t= \left( t_1,\dots,t_r \right) \in \mathbb{R}^{r}$ with $\epsilon \leq \left\lVert t \right\rVert \leq 4 \sqrt{q}$ and
\begin{equation*}
\begin{aligned}
\left| \hat{\mu}_{m;a_1,\dots,a_r}(t_1,\dots,t_r) \right| &\leq 
\exp{ \left( - c_2(r) \phi(m) M \log \left\lVert t \right\rVert \right)},
\end{aligned}
\end{equation*}
for $\left\lVert t \right\rVert > 4 \sqrt{q}$, where $c_1(r)$ and $c_2(r)$ are positive constants depending only on $r$.
\end{pro}

\begin{proof}
By using the explicit formula \eqref{forExplicitCha}, we have that for all $t= \left( t_1,\dots,t_r \right) \in \mathbb{R}^{r}$
$$
\left| \hat{\mu}_{m;a_1,\dots,a_r}(t) \right| \leq \prod_{ \substack{\chi  \bmod m \\ \chi \ne \chi_{0} }}^{} \prod_{\Im{(\gamma_{\chi})} > 0}^{} \left| J_{0} \left( 2 \left| \frac{\gamma_{\chi}}{\gamma_{\chi} - 1} \right| \left| \sum\limits_{j=1}^{r} \chi(a_j) t_j \right| \right) \right|.
$$
Let $V_m$ be the set of all non-principal Dirichlet characters modulo $m$ such that  $$\left| \sum\limits_{j=1}^{r} \chi(a_j) t_j \right| \geq \frac{\left\lVert t \right\rVert}{2}.$$
First, note that $\chi \in V_m$ if and only if $\overline{\chi} \in V_m$.
Let $\epsilon \leq \left\lVert t \right\rVert \leq 4 \sqrt{q}$. If $\chi \in V_{m}$, then $2 \left| \sum\limits_{j=1}^{r} \chi(a_j) t_j \right| \geq \left\lVert t \right\rVert \geq \epsilon$. Furthermore, we know that $\epsilon \left| \frac{\gamma_{\chi}}{\gamma_{\chi} - 1} \right| \leq \epsilon \left( 1 + \frac{1}{\sqrt{q} - 1} \right) \leq 1$, $J_{0}$ is a decreasing positive function on $\left[ 0,1 \right]$ and $\left| J_{0}(x) \right| \leq J_{0}(1)$ for $x \geq 1$. This implies 
\begin{equation*}
\begin{aligned}
    \left| \hat{\mu}_{m;a_1,\dots,a_r}(t_1,\dots,t_r) \right| &\leq \prod\limits_{\chi \in V_m}^{} \prod\limits_{\Im{\left( \gamma_{\chi} \right)} > 0}^{} \left| J_{0} \left( 2 \left| \sum\limits_{j=1}^{r} \chi(a_j) t_j \right| \left| \frac{\gamma_{\chi}}{\gamma_{\chi} - 1} \right| \right) \right| \\
    &\leq \prod\limits_{\chi \in V_m}^{} \prod\limits_{\Im{\left( \gamma_{\chi} \right)} > 0}^{} \left| J_0 \left( \epsilon \left| \frac{\gamma_{\chi}}{\gamma_{\chi} - 1} \right| \right) \right|.
\end{aligned}
\end{equation*}
Moreover, for $x \leq 1$ we have that $J_0(x) \leq \exp{\left( - x^{2}/4 \right)},$ therefore 
\begin{equation}
\label{introDm}
\begin{aligned}
\left| \hat{\mu}_{m;a_1,\dots,a_r}(t_1,\dots,t_r) \right| \leq 
\exp{\left( \frac{- \epsilon^{2}}{4}  D_{m} \right)},
\end{aligned}
\end{equation}
where $D_m :=  \sum\limits_{\chi \in V_{m}}^{} \sum\limits_{\Im{\left( \gamma_{\chi} \right)} > 0}^{} \left| \frac{\gamma_{\chi}}{\gamma_{\chi} - 1} \right|^{2}$.
Since $\chi \in V_{m} \Longleftrightarrow \overline{\chi} \in V_{m}$, it is clear that $D_{m} = \frac{1}{2} \sum\limits_{\chi \in V_{m}}^{} \left( I(\chi) + I(\overline{\chi}) \right)$. Hence from \eqref{asymptCha}, it follows that
\begin{equation}
\begin{aligned}
\label{formuleDm}
D_{m} &= \frac{q}{2(q-1)} \sum\limits_{\chi \in V_{m}}^{} M(\chi^{*}) + O\left( \sum\limits_{\chi \in V_{m}}^{} \log M(\chi^{*}) \right) \\
&= \frac{q}{2(q-1)}  \sum\limits_{\chi \in V_{m}}^{} M(\chi^{*}) + O\left( \phi(m) \log M \right).
\end{aligned}
\end{equation}
By using Proposition \ref{firstTool} we have 
\begin{equation}
\begin{aligned}
\label{mesureLemmF11}
\sum\limits_{\chi \in V_{m}}^{} M(\chi^{*}) &= \sum\limits_{\chi \bmod m}^{} M(\chi^{*}) - \sum\limits_{\chi \not\in V_{m}}^{} M(\chi^{*})\\
&= \phi(m) M - \frac{\phi(m)}{\log q} \sum\limits_{P | m}^{} \frac{\log |P|}{|P| - 1} - \sum\limits_{\chi \not\in V_{m}}^{} M(\chi^{*}) \\
 &\geq |V_{m}| M - \frac{\phi(m)}{\log q} \sum\limits_{P | m}^{} \frac{\log |P|}{|P| - 1}.
\end{aligned}
\end{equation}
Therefore, it remains to find an adequate lower bound for $|V_{m}|$ to establish the first part of our proposition. \\
Let 
\begin{equation*}
\begin{aligned}
 S(t) &= \sum\limits_{\substack{ \chi  \bmod m \\ \chi \ne \chi_{0}}}^{} \left| \sum\limits_{j=1}^{r} \chi(a_j) t_j \right|^{2}= \sum\limits_{\chi \bmod m}^{} \left| \sum\limits_{j=1}^{r} \chi(a_j) t_j \right|^{2} - \left( \sum\limits_{j=1}^{r} t_j\right)^{2}\\
&= \sum\limits_{j=1}^{r} \sum\limits_{k=1}^{r} t_j t_k \sum\limits_{\chi \bmod m}^{} \chi(a_j) \overline{\chi}(a_k) - \left( \sum\limits_{j=1}^{r} t_j \right)^{2}.
\end{aligned}
\end{equation*}
Thus from the Cauchy-Schwartz inequality, we obtain 
\begin{equation}
\begin{aligned}
\label{mesureLem21}
S(t) &= \phi(m) \sum\limits_{j=1}^{r} t_j^{2} - \left( \sum\limits_{j=1}^{r} t_j \right)^{2} \geq (\phi(m) - r) \sum\limits_{j=1}^{r} t_j^{2} = (\phi(m) - r) \left\lVert t \right\rVert^{2}.
\end{aligned}
\end{equation}
Furthermore, since $ \left| \sum\limits_{j=1}^{r} \chi(a_j) t_j \right|^{2} \leq r \lVert t \rVert^{2}$ we get that 
\begin{equation}
\begin{aligned}
\label{mesureLem22}
S(t) &= \sum\limits_{\chi \in V_m}^{} \left| \sum\limits_{j=1}^{r} \chi(a_j) t_j \right|^{2} + \sum\limits_{\chi \notin V_m}^{} \left| \sum\limits_{j=1}^{r} \chi(a_j) t_j \right|^{2} \leq r \left|V_m \right| \left\lVert t \right\rVert^{2} + \frac{\phi(m)}{4} \left\lVert t \right\rVert^{2}. 
\end{aligned}
\end{equation}
Hence, combining \eqref{mesureLem21} and \eqref{mesureLem22} we deduce that if $\deg(m) = M$ is large enough then 
\begin{equation}
\label{mesureLemmF3}
\begin{aligned}
\left| V_m \right| \geq \frac{\phi(m)}{2r}.
\end{aligned}
\end{equation}
Thus for $M$ large enough, we deduce from \eqref{majorSom}, \eqref{mesureLemmF11}, and \eqref{mesureLemmF3} that 
\begin{equation}
\label{sumMChi}
\begin{aligned}
\sum\limits_{\chi \in V_{m}}^{} M(\chi^{*}) \geq \frac{ \phi(m)M}{4 r}.
\end{aligned}
\end{equation}
Therefore it follows from \eqref{formuleDm} and \eqref{sumMChi} that for $M$ large enough we have
\begin{equation}
\label{InegDm}
\begin{aligned}
D_{m} &\geq \frac{\phi(m) M}{8 r} + O\left( \phi(m) \log M \right).
\end{aligned}
\end{equation}
Hence by combining \eqref{introDm} and \eqref{InegDm} the first part of our proposition follows.

Now assume that $ \lVert t \rVert > 4 \sqrt{q}$. 
Since $\left| J_{0}(x) \right| \leq 1$ for all $x \in \mathbb{R}$, then 
\begin{equation*}
\begin{aligned}
\left| \hat{\mu}_{m;a_1,\dots,a_r}(t_1,\dots,t_r) \right|^{2} &\leq \prod\limits_{\chi \in V_{m}}^{} \left| F\left( \left| \sum\limits_{j=1}^{r} \chi(a_j) t_j \right|, \chi \right) F\left( \left| \sum\limits_{j=1}^{r} \chi(a_j) t_j \right|, \overline{\chi} \right) \right|.
\end{aligned}
\end{equation*}
Thus for $\chi \in V_{m}$ we have  $\left| \sum\limits_{j=1}^{r} \chi(a_j) t_j \right| \geq \frac{\lVert t \rVert}{2} > 2 \sqrt{q}$. Hence we deduce by using Lemma \ref{lemFourierTrans} that 
\begin{equation}
\label{mesTailproof}
\begin{aligned}
\left| \hat{\mu}_{m;a_1,\dots,a_r}(t_1,\dots,t_r) \right|^{2} &\leq \prod\limits_{\chi \in V_{m}}^{} \exp{\left( -\frac{1}{2} \left( \log \lVert t \rVert - \log 2 \right) (M(\chi^{*}) - 2) \right) }. 
\end{aligned}
\end{equation}
Therefore, by combining the fact that $\log 2 \leq \frac{\log \lVert t \rVert}{2}$ with \eqref{sumMChi} and \eqref{mesTailproof} we deduce
 the second part of our proposition. 
\end{proof}

The following result is an asymptotic formula for the Fourier transform in the range $\lVert t \rVert \ll \phi(m)^{-1/2}$. This is a crucial ingredient in the proof of Theorem \ref{forAsymDelta3}.
\begin{pro}
\label{transFourrier}
Fix an integer $r \geq 2$ and let $m \in \mathcal{M}_{q}$ be of large degree $M.$ Assume LI. Then for every constant $A = A(r) > 0$, there exists $L(A) > 0$ such that for all $L \geq L(A)$ and $t = (t_1,\dots,t_r) \in \mathbb{R}^{r}$ with $\lVert t \rVert \leq A M^{1/2}$ we have 
\begin{equation*}
\begin{aligned}
&\hat{\mu}_{m;a_1,\dots,a_r}\left( \frac{t_1}{\sqrt{N_{m}}}, \dots, \frac{t_r}{\sqrt{N_{m}}}\right) = \exp{\left( - \frac{t_1^{2} + \dots + t_r^{2}}{2} \right)}\left( 1 + \frac{i}{2 \sqrt{N_{m}}} \frac{\sqrt{q} + q}{q-1} \left(  \sum\limits_{j=1}^{r} C_{m}(a_j) t_j \right) \right. \\
&- \left. \frac{1}{4 N_{m}} \frac{q + q^{2}}{(q-1)^{2}} \left( \sum\limits_{j=1}^{r} C_{m}(a_j)^{2} t_j^{2} \right) - \frac{1}{N_{m}}  \sum\limits_{1 \leq j < k \leq r}^{} \left( B_{m}(a_j,a_k) + \frac{1}{2} \frac{q + q^2}{(q - 1)^{2}} C_{m}(a_j) C_{m}(a_k)  \right) t_j t_k  
\right. \\
&+ \frac{Q_{4}(t_1,\dots,t_r)}{N_{m}}  +\left. \sum\limits_{s=0}^{1} \sum\limits_{d=0}^{2}   \sum\limits_{\substack{0 \leq l \leq L \\ 2l \geq 3 - 2s - d}}^{} \frac{1}{2} \left( \frac{q^{d/2} + q^d}{(q-1)^{d}} \right)
\frac{C_m^{d} B_m^{l}}{N_m^{d/2 + l + s}} P_{s,d,l}(t_1,\dots,t_r) + O\left( \frac{r^{L} B_{m}^{L} \lVert t \rVert^{2L}}{L! N_{m}^{L}}\right) \right),
\end{aligned}
\end{equation*}
where $Q_4$ is a homogeneous polynomial of degree $4$ with bounded coefficients and the $P_{s,d,l}$ are homogeneous polynomials of degree $d + 2l + 4s$ whose coefficients are bounded uniformly uniformly by a function of $l$. 
\end{pro}

\begin{proof}
From the explicit formula \eqref{forExplicitCha},we have 
\begin{equation*}
\begin{aligned}
\log \hat{\mu}_{m;a_1,\dots,a_r} \left( \frac{t_1}{\sqrt{N_{m}}}, \dots, \frac{t_r}{\sqrt{N_{m}}} \right) &= \log \mathcal{B}_{m;a_1,\dots,a_r}\left( \frac{t_1}{\sqrt{N_{m}}},\dots, \frac{t_r}{\sqrt{N_{m}}} \right) \\ 
&\quad+ \sum\limits_{\substack{\chi  \bmod m \\ \chi \ne \chi_{0} }}^{} \sum\limits_{\Im{\left( \gamma_{\chi}\right)} > 0 }^{} \log J_{0} \left( \frac{2}{\sqrt{N_{m}}} \left| \frac{\gamma_{\chi}}{\gamma_{\chi} - 1} \right| \left| \sum\limits_{j=1}^{r} \chi(a_j) t_j \right| \right).
\end{aligned}
\end{equation*}
By \cite[Lemma $2.8$]{fiorilliMartin}  we know that for $|s| \leq 1$
\begin{equation*}
\begin{aligned}
\log J_0(s) = - \sum\limits_{n=1}^{\infty} u_{2n} s^{2n},
\end{aligned}
\end{equation*}
where $u_2 = 1/4$ and $u_{2n} \ll \left( \frac{5}{12} \right)^{2n}$ for $n \geq 2$. Thus 
\begin{equation}
\label{contrib1}
\begin{aligned}
\log \hat{\mu}_{m;a_1,\dots,a_r} \left( \frac{t_1}{\sqrt{N_{m}}}, \dots, \frac{t_r}{\sqrt{N_{m}}} \right) &= \log \mathcal{B}_{m;a_1,\dots,a_r}\left( \frac{t_1}{\sqrt{N_{m}}},\dots, \frac{t_r}{\sqrt{N_{m}}} \right) \\
&\quad- \sum\limits_{n=1}^{\infty} \frac{u_{2n} 2^{2n}}{N_{m}^{n}}
\sum\limits_{\substack{\chi  \bmod m \\ \chi \ne \chi_{0} }}^{} \sum\limits_{\Im{\left( \gamma_{\chi}\right)} > 0 }^{} 
\left| \frac{\gamma_{\chi}}{\gamma_{\chi} - 1} \right|^{2n}
\left| \sum\limits_{j=1}^{r} \chi(a_j) t_j \right|^{2n}.
\end{aligned}
\end{equation}
 The contribution of the term $n = 1$ to the RHS of \eqref{contrib1} equals
\begin{equation*}
\begin{aligned}
- \frac{1}{N_{m}} \sum\limits_{\substack{\chi  \bmod m \\ \chi \ne \chi_{0} }}^{} \sum\limits_{\Im{\left( \gamma_{\chi}\right)} > 0 }^{} \left| \frac{\gamma_{\chi}}{\gamma_{\chi} - 1} \right|^{2} 
\sum\limits_{1 \leq j,k \leq r}^{} \chi(a_j) \overline{\chi}(a_k) t_j t_k &= - \frac{1}{2} \left( \sum\limits_{j=1}^{r} t_j^{2} \right) - \frac{1}{N_{m}} \sum\limits_{1 \leq j < k \leq r}^{} B_{m}(a_j,a_k) t_j t_k.
\end{aligned}
\end{equation*}
The term $n=2$ contributes $\frac{Q_4(t_1,\dots,t_r)}{N_{m}}$, where
\begin{equation*}
\begin{aligned}
Q_4(t_1,\dots,t_r) &: = - \frac{16 u_4}{N_{m}} \sum\limits_{\substack{\chi  \bmod m \\ \chi \ne \chi_{0} }}^{} \sum\limits_{\Im{\left( \gamma_{\chi}\right)} > 0 }^{} \left| \frac{\gamma_{\chi}}{\gamma_{\chi} - 1} \right|^{4}
\left| \sum\limits_{j=1}^{r} \chi(a_j) t_j \right|^{4} \\
&= - \frac{16 u_4}{N_{m}} \sum\limits_{1 \leq j_1,j_2,j_3,j_4 \leq r}^{} \sum\limits_{\substack{\chi  \bmod m \\ \chi \ne \chi_{0} }}^{} \sum\limits_{\Im{\left( \gamma_{\chi}\right)} > 0 }^{} \chi(a_{j_1}) \chi(a_{j_2}) \overline{\chi}(a_{j_3}) \overline{\chi}(a_{j_4}) \left| \frac{\gamma_{\chi}}{\gamma_{\chi} - 1}\right|^{4} t_{j_1} t_{j_2} t_{j_3} t_{j_4}.
\end{aligned}
\end{equation*}
Since
$\sum\limits_{\substack{\chi  \bmod m \\ \chi \ne \chi_{0} }}^{} \sum\limits_{\Im{\left( \gamma_{\chi}\right)} > 0 }^{} \left| \frac{\gamma_{\chi}}{\gamma_{\chi} - 1}\right|^{4} \leq \left( \sqrt{2} + 1 \right)^{2} N_{m}$,
then $Q_4(t_1,\dots,t_r)$ is a homogeneous polynomial of degree $4$ with bounded coefficients. Moreover, since
$ \sum\limits_{\substack{\chi  \bmod m \\ \chi \ne \chi_{0} }}^{} \sum\limits_{\Im{\left( \gamma_{\chi}\right)} > 0 }^{} \left| \frac{\gamma_{\chi}}{\gamma_{\chi} - 1}\right|^{2n} \leq 16^{n-1} N_{m}$ then the contribution of the terms $n \geq 3$ to the RHS of \eqref{contrib1} is $\ll \lVert t \rVert^{6} / N_{m}^{2} \ll_{r} \frac{M}{\phi(m)^{2}}$. Therefore  
\begin{equation}
\label{logTransFour}
\begin{aligned}
\log \hat{\mu}_{m;a_1,\dots,a_r} \left( \frac{t_1}{\sqrt{N_{m}}}, \dots, \frac{t_r}{\sqrt{N_{m}}} \right) &= \log \mathcal{B}_{m;a_1,\dots,a_r}\left( \frac{t_1}{\sqrt{N_{m}}},\dots, \frac{t_r}{\sqrt{N_{m}}} \right) \\
&\quad- \frac{1}{2} \left( \sum\limits_{j=1}^{r} t_j^{2} \right)
- \frac{1}{N_{m}} \sum\limits_{1 \leq j < k \leq r}^{} B_{m}(a_j,a_k) t_j t_k \\
&\quad+ \frac{Q_4(t_1,\dots,t_r)}{N_{m}} + O_{r} \left( \frac{M}{\phi(m)^{2}} \right). 
\end{aligned}
\end{equation}
In order to establish the asymptotic formula of $\hat{\mu}_{m;a_1,\dots,a_r} \left( \frac{t_1}{\sqrt{N_{m}}}, \dots, \frac{t_r}{\sqrt{N_{m}}} \right)$, we will focus on the exponentiation of the RHS of \eqref{logTransFour}.
We have 
\begin{equation*}
\begin{aligned}
\mathcal{B}_{m;a_1,\dots,a_r}\left( \frac{t_1}{\sqrt{N_{m}}},\dots, \frac{t_r}{\sqrt{N_{m}}} \right) &= \frac{1}{2} \left[  \exp{\left( i \frac{\sqrt{q}}{(q-1) \sqrt{N_{m}}} \sum\limits_{j=1}^{r} C_{m}(a_j) t_j \right)} \right. \\
&\left. \quad+ \exp{\left( i \frac{q}{(q-1) \sqrt{N_{m}}} \sum\limits_{j=1}^{r} C_{m}(a_j) t_j \right)} \right].  \\
\end{aligned}
\end{equation*}
Hence in our range of $t$, it follows that 
\begin{equation*}
\begin{aligned}
\mathcal{B}_{m;a_1,\dots,a_r}\left( \frac{t_1}{\sqrt{N_{m}}},\dots, \frac{t_r}{\sqrt{N_{m}}} \right) &= \frac{1}{2} 
\left( \sum\limits_{d=0}^{2}  \frac{1}{d! N_{m}^{d/2}} \frac{q^{d/2} + q^d}{(q-1)^{d}} \left( i \sum\limits_{j=1}^{r} C_{m}(a_j) t_j \right)^{d} \right) \\
&\quad+   O_{r} \left( \frac{C_{m}^{3}}{\phi(m)^{3/2}} \right). 
\end{aligned}
\end{equation*}
Moreover, we know that 
\begin{equation*}
\begin{aligned}
\exp{\left( Q_{4}(t_1,\dots,t_r)/N_{m} \right)} &= 1 + Q_{4}(t_1,\dots,t_r)/N_{m} + O_{r} \left( \frac{M^{2}}{\phi(m)^{2}} \right),  
\end{aligned}
\end{equation*}
and that
\begin{equation*}
\begin{aligned}
\exp{\left( - \frac{1}{N_{m}} \sum\limits_{1 \leq j < k \leq r}^{} B_{m}(a_j,a_k) t_j t_k \right)} &= \sum\limits_{l=0}^{\infty} \frac{\left( - \sum\limits_{1 \leq j < k \leq r}^{} B_{m}(a_j,a_k) t_j t_k \right)^{l}}{l! N_{m}^{l}}. 
\end{aligned}
\end{equation*}
Thus dividing $\hat{\mu}_{m;a_1,\dots,a_r} \left( \frac{t_1}{\sqrt{N_{m}}}, \dots, \frac{t_r}{\sqrt{N_{m}}} \right)$ by $\exp{ \left( - \frac{\sum\limits_{i=1}^{r} t_{i}^{2}}{2} \right)}$, we obtain
\begin{equation}
\label{summands}
\begin{aligned}
\frac{1}{2} \sum\limits_{s=0}^{1} \sum\limits_{d=0}^{2}  \sum\limits_{l=0}^{\infty}
\frac{Q_{4}(t_1,\dots,t_r)^{s}}{d! l! N_{m}^{d/2 + s + l}} \left( \frac{q^{d/2} + q^d}{(q - 1)^{d}} \right) \left( i \sum\limits_{j=1}^{r} C_{m}(a_j) t_j \right)^{d} &  \left( - \sum\limits_{1 \leq j < k \leq r}^{} B_{m}(a_j,a_k) t_j t_k \right)^{l} \\ &+ O_{r} \left( \frac{C_m^{3}}{\phi(m)^{3/2}} \right).
\end{aligned}
\end{equation}

We collect the summands above according to $D = d  + 2s + 2l.$ The contribution of the terms $0 \leq D \leq 2$ to the main term of \eqref{summands} is equal to 
\begin{equation*}
\begin{aligned}
&1 + \frac{i}{2 \sqrt{N_{m}}} \left( \frac{q + \sqrt{q}}{q-1}\right) \sum\limits_{j=1}^{r} C_{m}(a_j) t_j  - \frac{1}{4N_{m}} \left( \frac{q + q^2}{(q-1)^2}\right) \left( \sum\limits_{j=1}^{r} C_{m}(a_j) t_j \right)^{2}\\
&- \frac{1}{N_{m}} \left( \sum\limits_{1 \leq j < k \leq r}^{} B_{m}(a_k,a_k) t_j t_k \right) + \frac{Q_{4}(t_1,\dots,t_r)}{N_{m}}.
\end{aligned}
\end{equation*}
Let $P_{s,d,l}(t_1,\dots,t_r)$ be the homogeneous polynomial of degree $d + 2l + 4s$ defined by 
\begin{equation*}
P_{s,d,l} (t_1,\dots,t_r) =
\frac{1}{d! l!} C_{m}^{-d} B_{m}^{-l} Q_{4}(t_1,\dots,t_r)^{s} \left( i \sum\limits_{j=1}^{r} C_{m}(a_j) t_j \right)^{d} \left(- \sum\limits_{1 \leq j < k \leq r}^{} B_{m}(a_j,a_k) t_j t_k \right)^{l}. 
\end{equation*}
Then the contribution of the terms with $D \geq 3$ to \eqref{summands} is
\begin{equation*}
\sum\limits_{s=0}^{1} \sum\limits_{d=0}^{2}  \sum\limits_{\substack{l \geq 0 \\ 2l \geq 3 - 2s - d}}^{}  \frac{1}{2} \left( \frac{q^{d/2} + q^d}{(q - 1)^{d}}\right) \frac{C_{m}^{d} B_{m}^{l}}{N_{m}^{d/2 + l + s}} P_{s,d,l}(t_1,\dots,t_r).
\end{equation*}
Since $r$ is fixed and $s,d \leq 2$, then it is clear that the coefficients of $P_{s,d,l}$ are uniformly bounded by a function of $l$. Moreover, since $C_{m} = |m|^{o(1)}$ and $\frac{|m|}{\phi(m)} \ll  M$ (from \eqref{majorSomme2}), then 
\begin{equation*}
\begin{aligned}
\frac{1}{2} \left( \frac{q^{d/2} + q^{d}}{(q-1)^d} \right) \frac{C_{m}^{d} B_{m}^{l}}{N_{m}^{d/2 + l + s}} P_{s,d,l}(t_1,\dots,t_r) &\ll  \frac{r^{d/2 + l} C_{m}^{d} B_{m}^{l}}{d! l! N_{m}^{d/2 + l + s}} \lVert t \rVert^{d + 2l + 4s} \ll \frac{r^{l} B_{m}^{l} \lVert t \rVert^{2l} }{l! N_{m}^{l}}.
\end{aligned}
\end{equation*}
Furthermore, by using Corollary \ref{majorBm} we know that there exists $c > 0$ such that $B_{m} \leq c \phi(m)$. Thus for $\lVert t \rVert \leq A M^{1/2}$ we have  $r^{2} B_{m} \lVert t \rVert^{2} / (l N_{m}) \ll r^{2}  A^{2} / l.$ 
Hence for a suitably large constant $L(A) > 0$ (which also depends on $r$), we deduce that for $L \geq L(A)$ we have
\begin{equation*}
\sum\limits_{s=0}^{1}  \sum\limits_{d=0}^{2} \sum\limits_{\substack{l \geq L \\ 2l \geq 3 - 2s - d}}^{} \frac{1}{2} \left( \frac{q^{d/2} + q^d}{(q-1)^d} \right) \frac{C_{m}^{d} B_{m}^{l}}{N_{m}^{d/2 + l + s}} P_{s,d,l}(t_1,\dots,t_r) \ll  \frac{r^{L} B_{m}^{L} \lVert t \rVert^{2L}}{L! N_{m}^{L}}, 
\end{equation*}
which completes the proof.    
\end{proof}

For $s = (s_1,s_2,\dots,s_r) \in \mathbb{R}^{r}$, we define
\begin{equation*}
\begin{aligned}
\mathcal{L}_{m;a_1,\dots,a_r}(s_1,\dots,s_r) &:= \int_{x \in \mathbb{R}^{r}}^{} e^{s_1 x_1 + \dots + s_r x_r} {d}\mu_{m;a_1,\dots,a_r}(x_1,\dots,x_r),
\end{aligned}
\end{equation*}
if this integral converges. By using similar arguments to those used by Cha to obtain the explicit formula \eqref{forExplicitCha} of $\hat{\mu}_{m;a_1,\dots,a_r}$, 
it follows that under the assumption of LI,  $\mathcal{L}_{m;a_1,\dots,a_r}(s)$ exists for all $s \in \mathbb{R}^{r}$ and 
\begin{equation}
\label{similFourrier}
\begin{aligned}
\mathcal{L}_{m;a_1,\dots,a_r}(s) &= \mathcal{A}_{m;a_1,\dots,a_r}(s) \prod_{ \substack{\chi  \bmod m \\ \chi \ne \chi_{0} }}^{} \prod_{\Im{(\gamma_{\chi})} > 0}^{} I_{0} \left( 2 \left| \frac{\gamma_{\chi}}{\gamma_{\chi} - 1} \right| \left| \sum\limits_{j=1}^{r} \chi(a_j) s_j \right| \right),
\end{aligned}
\end{equation}
where $I_{0}(t) = \sum\limits_{n=0}^{\infty} \left( t/2 \right)^{2n} / n!^{2}$ is the modified Bessel function of order $0$ and 
\begin{equation*}
\begin{aligned}
\mathcal{A}_{m;a_1,\dots,a_r}(s) = \frac{1}{2} \left[ \exp{\left( - \frac{\sqrt{q}}{q - 1} \sum\limits_{j=1}^{r} C_{m}(a_j) s_j \right)} + \exp{\left( - \frac{q}{q-1} \sum\limits_{j=1}^{r} C_{m}(a_j) s_j \right)} \right]. 
\end{aligned}
\end{equation*}
The next step is to bound the tail of the measure $\mu_{m;a_1,\dots,a_r}$. To this end, we relate this tail with the Laplace Transform of $\mu_{m;a_1,\dots,a_r}$ defined above. This leads us to prove the following 

\begin{pro}
\label{tail}
Fix an integer $r \geq 2$, and let $m \in \mathcal{M}_{q}$ be of large degree $M.$ Assume LI. Then there exists $\alpha > 0$ such that for $ R \geq \sqrt{\phi(m) M}$ we have 
\begin{equation*}
\begin{aligned}
\mu_{m;a_1,\dots,a_r}(|x|_{\infty} > R) &\leq 2 r \exp{\left( - \frac{R^{2}}{\alpha \phi(m) M}  \left( 1 + O\left( \frac{\log M}{M} \right) \right) \right)}.
\end{aligned}
\end{equation*}
\end{pro}

\begin{proof}
Notice that 
\begin{equation*}
\mu_{m;a_1,\dots,a_r}(|x|_{\infty} > R) \leq \sum\limits_{j=1}^{r} \mu_{m;a_1,\dots,a_r}(x_j > R) + \sum\limits_{j=1}^{r} \mu_{m;a_1,\dots,a_r}(x_j < - R).
\end{equation*}
Thus in order to prove this inequality, it is sufficient to prove that there exists $\alpha > 0$ such that $\mu_{m;a_1,\dots,a_r}(x_j > R)$ and $\mu_{m;a_1,\dots,a_r}(x_j < -R)$ are smaller than  $\exp{\left( - \frac{R^{2}}{\alpha \phi(m) M}  \left( 1 + O\left( \frac{\log M}{M} \right) \right) \right)}.$ \\
We know that
\begin{equation*}
\begin{aligned}
\mu_{m;a_1,\dots,a_r} \left( x_j > R\right) &= \int_{\substack{(x_1,\dots,x_r) \in \mathbb{R}^{r} \\ x_j > R}}^{} {d}\mu_{m;a_1,\dots,a_r}(x_1,\dots,x_r).
\end{aligned}
\end{equation*}
Let $s > 0$, then we get
\begin{equation*}
\begin{aligned}
\mu_{m;a_1,\dots,a_r}(x_j > R) &\leq \exp{(- s R)} \int_{(x_1,\dots,x_r) \in \mathbb{R}^{r}}^{} \exp{(s x_j)} {d}\mu_{m;a_1,\dots,a_r}(x_1,\dots,x_r).
\end{aligned}
\end{equation*}
Hence from  \eqref{similFourrier} we obtain 
\begin{equation*}
\begin{aligned}
\mu_{m;a_1,\dots,a_r}(x_j > R) &\leq \exp{(- s R)} \mathcal{A}_{m;a_1,\dots,a_r}(0,\dots,s,\dots,0) \\
& \quad \times \prod_{ \substack{\chi  \bmod m \\ \chi \ne \chi_{0} }}^{} \prod_{\Im{(\gamma_{\chi})} > 0}^{} I_{0} \left( 2 \left| \frac{\gamma_{\chi}}{\gamma_{\chi} - 1} \right| s \right).
\end{aligned}
\end{equation*}
Since $I_{0}(s) \leq \exp{\left(s^2/4 \right)}$ for all $s \in \mathbb{R}$, then we have
\begin{equation*}
\begin{aligned}
\prod_{ \substack{\chi  \bmod m \\ \chi \ne \chi_{0} }}^{} \prod_{\Im{(\gamma_{\chi})} > 0}^{} I_{0} \left( 2 \left| \frac{\gamma_{\chi}}{\gamma_{\chi} - 1} \right| s \right)  &\leq
\prod_{ \substack{\chi  \bmod m \\ \chi \ne \chi_{0} }}^{} \prod_{\Im{(\gamma_{\chi})} > 0}^{} \exp{\left( \left| \frac{\gamma_{\chi}}{\gamma_{\chi} - 1} \right|^{2} s^2 \right)}.
\end{aligned}
\end{equation*}
Thus 
\begin{equation*}
\begin{aligned}
\prod_{ \substack{\chi  \bmod m \\ \chi \ne \chi_{0} }}^{} \prod_{\Im{(\gamma_{\chi})} > 0}^{} I_{0} \left( 2 \left| \frac{\gamma_{\chi}}{\gamma_{\chi} - 1} \right| s \right) &\leq \exp{\left( s^2 \sum\limits_{\substack{\chi  \bmod m \\ \chi \ne \chi_{0} }}^{} \sum\limits_{\Im{\left( \gamma_{\chi}\right)} > 0 }^{} \left| \frac{\gamma_{\chi}}{\gamma_{\chi} - 1}\right|^{2} \right)}.
\end{aligned}
\end{equation*}
On the other hand, by Lemma \ref{formuleNm} we have 
\begin{equation*}
\begin{aligned}
\sum\limits_{\substack{\chi  \bmod m \\ \chi \ne \chi_{0} }}^{} \sum\limits_{\Im{\left( \gamma_{\chi}\right)} > 0 }^{} \left| \frac{\gamma_{\chi}}{\gamma_{\chi} - 1}\right|^{2} &=
\frac{q}{2 (q - 1)} \phi(m) M + O\left( \phi(m) \log M \right).
\end{aligned}
\end{equation*}
Then it follows that 
\begin{equation*}
\begin{aligned}
\prod_{ \substack{\chi  \bmod m \\ \chi \ne \chi_{0} }}^{} \prod_{\Im{(\gamma_{\chi})} > 0}^{} I_{0} \left( 2 \left| \frac{\gamma_{\chi}}{\gamma_{\chi} - 1} \right| s \right) &\leq \exp{\left( s^2 \frac{q}{2(q-1)} \phi(m) M \left( 1 + O\left( \frac{\log M}{M} \right) \right) \right)}.
\end{aligned}
\end{equation*}
We finally deduce that 
\begin{equation*}
\begin{aligned}
\mu_{m;a_1,\dots,a_r}(x_j > R) &\leq \exp{\left(- sR + \frac{q}{q-1} s \right)} \exp{\left( s^2 \frac{q}{2(q-1)} \phi(m) M \left( 1 + O\left( \frac{\log M}{M} \right) \right) \right)}.
\end{aligned}
\end{equation*}
So if we take $s = \frac{R}{3 \phi(m) M}$ then there exists $\alpha_{1} > 0$ such that $$\mu_{m;a_1,\dots,a_r}(x_j > R) \leq \exp{\left( - \frac{R^{2}}{\alpha_{1} \phi(m) M}  \left( 1 + O\left( \frac{\log M}{M} \right) \right) \right)}.$$
\\
To get an upper bound for $\mu_{m;a_1,\dots,a_r}(x_j < -R)$ we use similar arguments to obtain 
\begin{equation*}
\begin{aligned}
\mu_{m;a_1,\dots,a_r}(x_j < -R) &\leq \frac{1}{2} \left[ \exp{\left( -R s + \frac{q}{q-1} C_{m}(a_j) s \right)} + 
\exp{\left( -R s + \frac{\sqrt{q}}{q-1} C_{m}(a_j) s \right)} \right] \\
&\quad \times \exp{\left( s^2 \frac{q}{2(q-1)} \phi(m) M \left( 1 + O\left( \frac{\log M}{M} \right) \right) \right)}.
\end{aligned}
\end{equation*}
Then by using the fact that $C_{m}(a_j) = O_{\epsilon}\left( \left(q^M \right)^{\epsilon} \right)$ for any $\epsilon > 0$, and by taking $s = \frac{R}{3 \phi(m) M}$ it follows that there exists $\alpha_{2} > 0$ such that $\mu_{m;a_1,\dots,a_r}(x_j < -R) \leq \exp{\left( - \frac{R^{2}}{\alpha_{2} \phi(m) M}  \left( 1 + O\left( \frac{\log M}{M} \right) \right) \right)}.$ \\
Hence by taking $\alpha = \max(\alpha_{1},\alpha_{2})$ the result follows.
\end{proof}

\section{An asymptotic formula for the densities $\delta_{m;a_1,a_2}$}
\label{section6}The goal of this section is to derive an asymptotic formula for the densities when $r=2$ and prove Theorem \ref{asymDelta2}. 
We first adapt the proof of Proposition \ref{transFourrier} to establish the following key lemma. 
\begin{lem} 
\label{lemSimCase2}
For $t=(t_1,t_2)\in \mathbb{R}^2$ with $||t||\leq N_m^{1/4}$ we have
$$
\hat{\mu}_{m;a_1,a_2}
\left(\frac{t_1}{\sqrt{N_m}},\frac{t_2}{\sqrt{N_m}}\right)
=\exp\left(-\frac{t_1^2+t_2^2}{2}-\frac{B_m(a_1,a_2)}{N_m}t_1t_2\right)
F_{m;a_1,a_2}(t_1,t_2),$$
where
$$ F_{m;a_1,a_2}(t_1,t_2)=1 + \frac{i}{2 \sqrt{N_{m}}} \left( \frac{\sqrt{q} + q}{q-1} \right) \left(   C_{m}(a_1) t_1 + C_{m}(a_2) t_2 \right)
+ O\left(\frac{||t||^4}{N_m}+ \frac{||t||^2C_{m}(1)^2}{N_m}\right).$$
\end{lem}
\begin{proof} For $||t||\leq N_m^{1/4}$ the explicit formula \eqref{forExplicitCha} implies that $\log\hat{\mu}_{m;a_1,a_2}
\left(t_1N_m^{-1/2},t_2N_m^{-1/2}\right)$ equals
\begin{equation}
\label{estimate61}
\begin{aligned}
&\log \mathcal{B}_{m;a_1,a_2}\left(\frac{t_1}{\sqrt{N_m}},\frac{t_2}{\sqrt{N_m}}\right) - \frac{1}{N_m}\sum_{\substack{\chi\neq\chi_0\\ \chi \bmod m}}\sum\limits_{\Im{\left( \gamma_{\chi} \right)} > 0}^{} \left| \frac{\gamma_{\chi}}{\gamma_{\chi} - 1} \right|^{2} |\chi(a_1)t_1+\chi(a_2)t_2|^2
+O\left(\frac{||t||^4}{N_m}\right)\\
=&\log \mathcal{B}_{m;a_1,a_2}\left(\frac{t_1}{\sqrt{N_m}},\frac{t_2}{\sqrt{N_m}}\right)
-\frac{t_1^2+t_2^2}{2}-\frac{B_m(a_1,a_2)}{N_m}
t_1t_2+O\left(\frac{||t||^4}{N_m}\right).
\end{aligned}
\end{equation}
On the other hand, we have 
\begin{equation}
\label{inter1case2}
\begin{aligned}
\mathcal{B}_{m;a_1,a_2}\left(\frac{t_1}{\sqrt{N_m}},\frac{t_2}{\sqrt{N_m}}\right) &= \frac{1}{2} \left[ \exp{\left(i \frac{\sqrt{q}}{q-1} \left( C_{m}(a_1)\frac{t_1}{\sqrt{N_m}} +C_{m}(a_2)\frac{t_2}{\sqrt{N_m}} \right) \right)} \right.\\
&\left. \quad+ \exp{\left(i \frac{q}{q-1} \left( C_{m}(a_1) \frac{t_1}{\sqrt{N_m}} +C_{m}(a_2) \frac{t_2}{\sqrt{N_m}} \right) \right)} \right] \\
&= 1 + \frac{i}{2 \sqrt{N_m}} \left( \frac{\sqrt{q}+q}{q-1} \right) \left( C_{m}(a_1)t_1+C_{m}(a_2)t_2 \right) + O\left(\frac{||t||^2C_{m}(1)^2}{N_m}\right).
\end{aligned}
\end{equation}
Thus, by combining \eqref{estimate61} and \eqref{inter1case2} our lemma follows.
\end{proof}
To prove Theorem \ref{asymDelta2}, we need the following lemma of \cite{lamzouri}. 
\begin{lem} [{\cite[Lemma $8.2$]{lamzouri} }]
\label{lemInter}
Let $\rho$ be a real number such that $|\rho|\leq 1/2$, $n_1$, $n_2$ are  fixed non-negative integers and $S$ a large positive number. Then
\begin{equation*}
\begin{aligned}
&\int_{||t||\leq S}e^{i(t_1x_1+t_2x_2)}t_1^{n_1}t_2^{n_2}\exp\left(
-\frac{t_1^2+t_2^2+2\rho t_1t_2}{2}\right)dt_1dt_2\\
&= \frac{1}{i^{n_1+n_2}}\frac{\partial^{n_1+n_2}
\Phi_{\rho}(x_1,x_2)}{\partial x_1^{n_1}\partial x_2^{n_2}}      + O\left(\exp\left(-\frac{S^2}{8}\right)\right),
\end{aligned}
\end{equation*}
where
$$\Phi_{\rho}(x_1,x_2)= \frac{2\pi}{\sqrt{1-\rho^2}}
\exp\left(-\frac{1}{2(1-\rho^2)}(x_1^2+x_2^2-2\rho x_1x_2)\right).
$$
\end{lem}

\begin{proof}[Proof of Theorem \ref{asymDelta2}]
Write $\mu_{m}=\mu_{m;a_1,a_2}$ and let $R=\sqrt{N_m}M$, then we have from Proposition \ref{tail} that 
$$ \delta_{m;a_1,a_2}=\int_{-R<y_2<y_1<R}d\mu_m(y_1,y_2)+ O\left(\exp\left(-\frac{M^2}{\beta}\right)\right),$$
for some constant $\beta > 0$.
By applying the Fourier inversion formula to the measure $\mu_m$ we get
\begin{equation}
\label{densityCase21}
 \delta_{m;a_1,a_2}= \frac{1}{(2\pi)^2}\int_{-R<y_2<y_1<R}\int_{s\in \mathbb{R}^2}e^{i(s_1y_1+s_2y_2)}\hat{\mu}_m(s_1,s_2)d{\bf s}d{\bf y}+ O\left(\exp\left(-\frac{M^2}{\beta}\right)\right).
\end{equation}
Furthermore, by using Proposition \ref{majorMesur} with $\epsilon=M N_m^{-1/2}$ we obtain that
$$\int_{s\in \mathbb{R}^2}e^{i(s_1y_1+s_2y_2)}\hat{\mu}_m(s_1,s_2)d{\bf s}= \int_{||s||\leq \epsilon }e^{i(s_1y_1+s_2y_2)}\hat{\mu}_m(s_1,s_2)d{\bf s} +O\left(\exp\left(-c M^2\right)\right),
$$
for some constant $c>0$.
Inserting this estimate in \eqref{densityCase21}, and making the change of variables $t_j=\sqrt{N_m}s_j$ and $x_j=y_j/\sqrt{N_m}$ for $j=1,2$, we deduce from Lemma \ref{lemSimCase2} that
\begin{equation}
\label{estimate64}
\begin{aligned}
\delta_{m;a_1,a_2}&= \frac{1}{(2\pi)^2}
\int_{-M<x_2<x_1<M}\int_{||t||<M}e^{i(t_1x_1+t_2x_2)}\hat{\mu}_m\left(\frac{t_1}{\sqrt{N_m}},
\frac{t_2}{\sqrt{N_m}}\right)d{\bf t}d{\bf x} \\
&\quad+O\left(\exp\left(-M^{3/2}\right)\right).\\
&= I_0+ \frac{i (\sqrt{q} + q)}{2(q-1)} \frac{C_{m}(a_1)}{\sqrt{N_m}}I_1+
\frac{i (\sqrt{q} + q)}{2(q-1)} \frac{C_{m}(a_2)}{\sqrt{N_m}}I_2 +O\left(\frac{C_{m}(1)^2M^2}{N_m}\right),
\end{aligned}
\end{equation}
where
$$ I_0= \frac{1}{(2\pi)^2}
\int_{-M<x_2<x_1<M}\int_{||t||<M}e^{i(t_1x_1+t_2x_2)}
\exp\left(-\frac{t_1^2+t_2^2}{2}-\frac{B_m(a_1,a_2)}{N_m}t_1t_2\right)
d{\bf t}d{\bf x},$$
and
$$ I_j= \frac{1}{(2\pi)^2}
\int_{-M<x_2<x_1<M}\int_{||t||<M}e^{i(t_1x_1+t_2x_2)}
t_j\exp\left(-\frac{t_1^2+t_2^2}{2}-\frac{B_m(a_1,a_2)}{N_m}t_1t_2\right)
d{\bf t}d{\bf x},$$
for $j=1,2$. We shall first evaluate $I_0$. Let $\rho= B_m(a_1,a_2)/N_m$. Then Corollary \ref{majorBm} implies that $|\rho|\leq 1/2$ for $M$ large enough. Hence Lemma \ref{lemInter} gives
\begin{equation*}
\begin{aligned}
I_0&= \frac{1}{2\pi\sqrt{1-\rho^2}}
\int_{-M<x_2<x_1<M}\exp\left(-\frac{1}{2(1-\rho^2)}(x_1^2+x_2^2-2\rho x_1x_2)\right)dx_1dx_2\\
&\quad+ O\left(\exp\left(-\frac{M^2}{10}\right)\right).\\
\end{aligned}
\end{equation*}
The integral on the RHS of the last estimate equals
$$\frac{1}{2\pi\sqrt{1-\rho^2}}
\int_{x_1>x_2}\exp\left(-\frac{1}{2(1-\rho^2)}(x_1^2+x_2^2-2\rho x_1x_2)\right)dx_1dx_2 + O\left(\exp\left(-\frac{M^2}{10}\right)\right).$$
Therefore, using that the integrand is symmetric in $x_1$ and $x_2$, along with the fact that
$$ \frac{1}{2\pi\sqrt{1-\rho^2}}
\int_{-\infty}^{\infty}\int_{-\infty}^{\infty}\exp\left(-\frac{1}{2(1-\rho^2)}(x_1^2+x_2^2-2\rho x_1x_2)\right)dx_1dx_2= 1, $$ we obtain that
\begin{equation}
\label{estimate65}
 I_0= \frac{1}{2}+  O\left(\exp\left(-\frac{M^2}{10}\right)\right).
\end{equation}
Using similar ideas along with Lemma \ref{lemInter} yields
\begin{equation*}
\begin{aligned}
 I_1&= \frac{1}{(2\pi)^2 i}\int_{x_1>x_2}\frac{\partial \Phi_{\rho}(x_1,x_2)}{\partial x_1}dx_1dx_2+O\left(\exp\left(-\frac{M^2}{10}\right)\right)\\
 &=-\frac{1}{(2\pi)^2 i}\int_{-\infty}^{\infty} \Phi_{\rho}(x_2,x_2)dx_2+O\left(\exp\left(-\frac{M^2}{10}\right)\right) .\\
 \end{aligned}
 \end{equation*}
 Furthermore, we know that
 $$ \int_{-\infty}^{\infty} \Phi_{\rho}(y,y)dy= \frac{2\pi}{\sqrt{1-\rho^2}}\int_{-\infty}^{\infty} \exp\left(-\frac{y^2}{2}\left(\frac{2}{1+\rho}\right)\right)dy= \frac{2\pi^{3/2}}{\sqrt{1-\rho}}.$$
Since $2(1-\rho)= V_m(a_1,a_2)/N_m$, then upon combining the above estimates we obtain that 
\begin{equation}
I_1= - \frac{\sqrt{N_m}}{i\sqrt{2\pi V_m(a_1,a_2)}} + O\left(\exp\left(-\frac{M^2}{10}\right)\right).
\end{equation}
Similarly, it is easy to check that
\begin{equation}
\label{estimate67}
I_2= \frac{\sqrt{N_m}}{i\sqrt{2\pi V_m(a_1,a_2)}} + O\left(\exp\left(-\frac{M^2}{10}\right)\right).
\end{equation}
Finally, inserting the estimates \eqref{estimate65}-\eqref{estimate67} into  \eqref{estimate64}, and using the fact that $N_m \sim \frac{q}{q-1} \phi(m) M \log_q |m|$ the desired result follows.
\end{proof}

\section{An asymptotic formula for the densities $\delta_{m;a_1,\dots,a_r}$ for $r \geq 3$}
\label{section7}In this section, we will prove an asymptotic formula for $\delta_{m;a_1,\dots,a_r}$ (Theorem \ref{forAsymDelta3}) as well as Theorem \ref{genra212}. We have already found a bound for the tail of the measure $\mu_{m;a_1,\dots,a_r}$ in Proposition \ref{tail}. We also found a formula for the Fourier transform $\hat{\mu}_{m;a_1,\dots,a_r}$ in Proposition \ref{transFourrier}. Thus following the strategy used to obtain the asymptotic formula when $r=2$ leads us to prove Theorem \ref{forAsymDelta3}.

\begin{proof}[Proof of Theorem \ref{forAsymDelta3}]
First we shorten the notations using $\delta_{m}$ and $\mu_{m}$ to refer to $\delta_{m;a_1,\dots,a_r}$ and $\mu_{m;a_1,\dots,a_r}$ respectively. 
We have 
\begin{equation*}
\begin{aligned}
\delta_{m} &= \int_{\substack{y_1 >\dots>y_r \\ \left| y \right|_{\infty} \leq R}}^{} {d} \mu_{m}(y_1,\dots,y_r) 
+ \int_{\substack{y_1 >\dots>y_r \\ \left| y \right|_{\infty} > R}}^{} {d} \mu_{m}(y_1,\dots,y_r).
\end{aligned}
\end{equation*}
Thus from Proposition \ref{tail} it follows that
\begin{equation*}
\begin{aligned}
\mu_{m;a_1,\dots,a_r}( \left| y \right|_{\infty} > R) &\leq 
2r \exp{\left( \frac{- R^{2}}{\alpha \phi(m) M} \left( 1 + O\left( \frac{\log M}{M} \right) \right) \right)}.
\end{aligned}
\end{equation*}
In particular we take $R = \sqrt{N_{m}} M$, and hence there exists $\beta > 0$ such that 
\begin{equation*}
\begin{aligned}
\delta_{m} &=  \int_{\substack{y_1 >\dots>y_r \\ \left| y \right|_{\infty} \leq R}}^{} {d} \mu_{m}(y_1,\dots,y_r) 
+ O_{r} \left( \exp{\left( \frac{- M^{2}}{\beta} \right)} \right).
\end{aligned}
\end{equation*}
By applying the Fourier inversion formula to $\mu_{m}$ we obtain
\begin{equation*}
\begin{aligned}
\int_{\substack{y_1 >\dots>y_r \\ \left| y \right|_{\infty} \leq R}}^{} {d} \mu_{m}(y_1,\dots,y_r)  &= (2 \pi)^{-r} \int_{\substack{y_1 >\dots>y_r \\ \left| y \right|_{\infty} \leq R}}^{} \int_{s \in \mathbb{R}^{r}}^{} e^{i (s_1 y_1 +\dots+ s_r y_r)} \hat{\mu}_{m} (s_1,\dots,s_r) {d}s {d}y.  
\end{aligned}
\end{equation*}
Now, let $A = A(r) \geq r$ be a suitably large constant. Therefore, we obtain by using Proposition \ref{majorMesur} with $\epsilon = A (N_m)^{-1/2} M^{1/2}$ that 
\begin{equation*}
\begin{aligned}
  \int_{s \in \mathbb{R}^{r}}^{} e^{i (s_1 y_1 +\dots+ s_r y_r)} \hat{\mu}_{m} (s_1,\dots,s_r) {d}s  
  = \int_{\lVert s \rVert \leq \epsilon}^{} e^{i (s_1 y_1 +\dots+ s_r y_r)} \hat{\mu}_{m} (s_1,\dots,s_r) {d}s  
+ O\left( \exp{ \left( - 2 A M \right)} \right).
\end{aligned}
\end{equation*}
Since we know that $R^r \exp{\left(-2 A M \right)} \ll \exp{\left( - A M \right)}$, then it follows that 
\begin{equation*}
\begin{aligned}
\delta_{m} &= (2 \pi)^{-r} \int_{\substack{y_1 >\dots>y_r \\ \left| y \right|_{\infty} \leq R}}^{} \int_{\lVert s \rVert \leq \epsilon}^{} e^{i (s_1 y_1 +\dots+ s_r y_r)} \hat{\mu}_{m} (s_1,\dots,s_r) {d}s {d}y + O\left( \exp{ \left( - A M \right)} \right).
\end{aligned}
\end{equation*}
By making the following change of variables $t_j = \sqrt{N_{m}} s_j$ and $x_j = \frac{y_j}{\sqrt{N_{m}}}$, we deduce that 
\begin{equation*}
\begin{aligned}
\delta_{m} &= (2 \pi)^{-r} \int_{\substack{x_1 >\dots>x_r \\ \left| x \right|_{\infty} \leq M}}^{} \int_{\lVert t \rVert \leq A M^{1/2}}^{} e^{i (t_1 x_1 +\dots+ t_r x_r)} \hat{\mu}_{m} \left(\frac{t_1}{\sqrt{N_{m}}},\dots,\frac{t_r}{\sqrt{N_{m}}} \right) {d}t {d}x  + O\left( \exp{\left(- A M \right)}\right).
\end{aligned}
\end{equation*} 

We take $L = L(A) \geq 2 r$ a suitably large constant. Hence by using Proposition \ref{transFourrier} we have 
\begin{equation}
\label{A1}
\begin{aligned}
\delta_{m} &= J_0 + \frac{i}{2 \sqrt{N_{m}}} \left( \frac{q + \sqrt{q}}{q - 1}\right) \sum\limits_{j=1}^{r} C_{m}(a_j) G_{j} - \frac{1}{4 N_m} \left( \frac{q+q^2}{(q-1)^2} \right) \sum\limits_{j=1}^{r} C_m(a_j)^2 R_j \\
& \quad- \frac{1}{N_{m}} \sum\limits_{1 \leq j < k \leq r}^{} \left[ B_{m}(a_j,a_k) + \frac{1}{2} \left(\frac{q + q^2}{(q-1)^2}\right) C_{m}(a_j) C_{m}(a_k) \right] S_{j,k} \\
& \quad+ \frac{T_0}{N_{m}}  + \sum\limits_{s=0}^{1} \sum\limits_{d=0}^{2} \sum\limits_{\substack{ 0 \leq l \leq L \\ 2l \geq 3 - 2s - d}}^{} \frac{1}{2} \left( \frac{q^{d/2} + q^d}{(q-1)^{d}} \right) \frac{C_{m}^{d} B_{m}^{l}}{N_{m}^{d/2 + s + l}} Y_{s,d,l}  + E_{1},
\end{aligned}
\end{equation}
where $E_1 \ll \frac{M^r B_m^{L}}{N_m^{L}}$ and 
$$J_0 = (2 \pi)^{-r} \int_{\substack{x_1 >\dots>x_r \\ \left| x \right|_{\infty} \leq M}}^{} \int_{\lVert t \rVert \leq A M^{1/2}}^{} e^{i (t_1 x_1 +\dots+ t_r x_r)} \exp{\left( - \frac{t_1^{2} + \dots + t_r^{2}}{2} \right)} {d}t {d}x ,$$
$$ G_{j} = (2 \pi)^{-r} \int_{\substack{x_1 >\dots>x_r \\ \left| x \right|_{\infty} \leq M}}^{} \int_{\lVert t \rVert \leq A M^{1/2}}^{} t_j e^{i (t_1 x_1 +\dots+ t_r x_r)} \exp{\left( - \frac{t_1^{2} + \dots + t_r^{2}}{2} \right)} {d}t {d}x ,$$
$$
R_{j} = (2 \pi)^{-r} \int_{\substack{x_1 >\dots>x_r \\ \left| x \right|_{\infty} \leq M}}^{} \int_{\lVert t \rVert \leq A M^{1/2}}^{} t_j^{2} e^{i (t_1 x_1 +\dots+ t_r x_r)} \exp{\left( - \frac{t_1^{2} + \dots + t_r^{2}}{2} \right)} {d}t {d}x ,
$$
$$ S_{j,k} = (2 \pi)^{-r} \int_{\substack{x_1 >\dots>x_r \\ \left| x \right|_{\infty} \leq M}}^{} \int_{\lVert t \rVert \leq A M^{1/2}}^{} t_j t_k e^{i (t_1 x_1 +\dots+ t_r x_r)} \exp{\left( - \frac{t_1^{2} + \dots + t_r^{2}}{2} \right)} {d}t {d}x ,$$
$$
T_0 = (2 \pi)^{-r} \int_{\substack{x_1 >\dots>x_r \\ \left| x \right|_{\infty} \leq M}}^{} \int_{\lVert t \rVert \leq A M^{1/2}}^{} e^{i (t_1 x_1 +\dots+ t_r x_r)} \exp{\left( - \frac{t_1^{2} + \dots + t_r^{2}}{2} \right)} Q_{4}(t_1,\dots,t_r) {d}t {d}x,
$$
and 
$$
Y_{s,d,l} = (2 \pi)^{-r} \int_{\substack{x_1 >\dots>x_r \\ \left| x \right|_{\infty} \leq M}}^{} \int_{\lVert t \rVert \leq A M^{1/2}}^{}  e^{i (t_1 x_1 +\dots+ t_r x_r)} \exp{\left( - \frac{t_1^{2} + \dots + t_r^{2}}{2} \right)} P_{s,d,l}(t_1,\dots,t_r) {d}t {d}x .
$$
Furthermore, by using \cite[Lemma $4.3$]{lamzouri} we have 
\begin{equation}
\label{A2}
\begin{aligned}
& \frac{T_0}{N_m} + \sum\limits_{s=0}^{1} \sum\limits_{d=0}^{2} \sum\limits_{\substack{ 0 \leq l \leq L \\ 2l \geq 3 - 2s - d}}^{} \frac{1}{2} \left( \frac{q^{d/2} + q^d}{(q-1)^{d}} \right) \frac{C_{m}^{d} B_{m}^{l}}{N_{m}^{d/2 + s + l}} Y_{s,d,l} + E_1 \ll_r \frac{1}{N_m} + \frac{C_m B_m}{N_m^{3/2}} + \frac{B_m^2}{N_m^2}.
\end{aligned}
\end{equation}
Moreover, by \cite[Lemma $4.2$]{lamzouri} it follows that 
\begin{equation}
\begin{aligned}
\label{A3}
J_{0} &= \frac{1}{(2 \pi)^{r/2}}
 \int_{\substack{x_1 >\dots>x_r \\ \left| x \right|_{\infty} \leq M}}^{} \exp{\left( - \frac{x_1^2 +\dots+x_r^2}{2} \right)} {d}x + O\left( \exp{\left( - A M \right)}\right) \\
 &= \frac{1}{r! (2 \pi)^{r/2}} \int_{x \in \mathbb{R}^{r}}^{} \exp{\left( - \frac{x_1^{2} + \dots + x_r^{2}}{2} \right)} {d}x + O\left( \exp{\left( - A M \right)}\right) = \frac{1}{r!} + O\left( \exp{\left( - A M \right)}\right). 
\end{aligned}
\end{equation}
By using the same argument, we can see that  for all $1 \leq j \leq r$ we have  
\begin{equation}
\label{A4}
\begin{aligned}
G_j
&=  i \alpha_j(r) + O\left( \exp{\left( - A M \right)}\right), 
\end{aligned}
\end{equation}
and that for all $1 \leq j \leq r$  we have
\begin{equation}
\label{A5}
\begin{aligned}
R_j
&=  - \lambda_j(r) + O\left( \exp{\left( - A M \right)}\right). 
\end{aligned}
\end{equation}
For all $1 \leq j < k \leq r$ we deduce by analogy that
\begin{equation}
\label{A6}
\begin{aligned}
S_{j,k} &=  - \beta_{j,k}(r) + O\left( \exp{\left( - A M \right)}\right). 
\end{aligned}
\end{equation}
Hence, by combining the estimates \eqref{A1}-\eqref{A6} our theorem follows.
\end{proof}
We end this section by proving Theorem \ref{genra212}.
\begin{proof}[Proof of Theorem \ref{genra212}]
Let $M = \deg(m)$ and $S_{m}$ be the set of pairs $(a,b) \in \mathcal{A}_{2}(m)$ such that $\left| B_{m}(a,b) \right| \geq \sqrt{\phi(m)}$. Then by Theorem \ref{FirstMomBm} we have 
\begin{equation*}
\begin{aligned}
\left| S_{m} \right| \sqrt{\phi(m)} &\leq \sum\limits_{(a,b) \in \mathcal{A}_{2}(m)}^{} \left| B_m(a,b) \right| \ll \phi(m)^{2} M. 
\end{aligned}
\end{equation*}
Hence
\begin{equation*}
\begin{aligned}
\left| S_m \right| &\ll \phi(m)^{3/2} M.
\end{aligned}
\end{equation*}
Let $\Omega_{r}(m)$ be the set of $r$-tuples $(a_1,\dots,a_r) \in \mathcal{A}_{r}(m)$ such that there exist $1 \leq i \ne j \leq r$ with $(a_i,a_j) \in S_m$, hence  
\begin{equation*}
\begin{aligned}
\left| \Omega_r(m) \right| &\ll_{r} \phi(m)^{r- \frac{1}{2}} M.    
\end{aligned}
\end{equation*} 
Since $\left| \mathcal{A}_{r}(m) \right| =  \phi(m)^r + O_{r} (\phi(m)^{r-1})$, it follows that $\left| \Omega_r(m) \right| = o \left(\left| \mathcal{A}_{r}(m) \right| \right).$ 
On the other hand, if $(a_1,\dots,a_r) \in \mathcal{A}_{r}(m) \setminus{\Omega_r(m)}$ then for all $1 \leq i < j \leq r$ we have  $\left| B_{m}(a_i,a_j) \right| \leq \sqrt{\phi(m)}$, therefore we deduce from Corollary \ref{forAsymCor} that
\begin{equation*}
\begin{aligned}
\delta_{m;a_1,\dots,a_r} &= \frac{1}{r!} - \frac{q + \sqrt{q}}{2 \sqrt{N_m} (q-1)} \sum\limits_{j=1}^{r} \alpha_j(r) C_{m}(a_j) + O_{r} \left( \frac{1}{\sqrt{N_m M}} \right).
\end{aligned}
\end{equation*}
Thus for all $r$-tuples $(a_1,\dots,a_r), (b_1,\dots,b_r) \in \mathcal{A}_{r}(m) \setminus{\Omega_r(m)}$ we obtain
\begin{equation*}
\begin{aligned}
\delta_{m;a_1,\dots,a_r} - \delta_{m;b_1,\dots,b_r} &= \frac{q + \sqrt{q}}{2 \sqrt{N_m} (q-1)} \sum\limits_{j=1}^{r} - \alpha_j(r) \left( C_{m}(a_j) - C_m(b_j) \right) + O_{r} \left( \frac{1}{\sqrt{N_m M}} \right).
\end{aligned}
\end{equation*}
Moreover, since  for all $1 \leq j \leq r$ we know that $C_{m}(a_j)$ and $C_m(b_j)$ are integers, it follows that if $ - \sum\limits_{j=1}^{r} \alpha_j(r) \left( C_{m}(a_j) - C_m(b_j) \right)  > 0$ then $ - \sum\limits_{j=1}^{r}  \alpha_j(r) \left( C_{m}(a_j) - C_m(b_j) \right) \gg_r 1$, hence we finally obtain that 
$\delta_{m;a_1,\dots,a_r} > \delta_{m;b_1,\dots,b_r}$.
\end{proof}

\section{Explicit constructions of biased races}
\label{section8}The aim of this section is to prove Theorems \ref{biasedQuadratic}, \ref{biasedRaces} and \ref{generMartin}. The idea behind these explicit constructions of the $a_i$ is analogous to the ideas relative to  \cite[Section $6$]{lamzouri}. 

We first define 
\begin{equation}
\label{lambdaDef}
\begin{aligned}
\Lambda_{0}(f) := \left\{
    \begin{array}{ll}
        \frac{\Lambda(f)}{|f|} & \text{if } f \in \mathbb{F}_{q}[T], \\
        0 & \text{otherwise.}
    \end{array}
\right.
\end{aligned}
\end{equation}

Next, we will find a new estimates of the term $B_m(a,b)$ when $0 \leq \deg(a),\deg(b) < M$.
\begin{pro}
\label{moreExpBm}
Let $m \in \mathcal{M}_{q}$ be of large degree $M$. Let $(a,b) \in \mathcal{A}_{2}(m)$ such that $0 \leq \deg(a),\deg(b) < M$. 
\begin{enumerate}
    \item If $\deg(a) = \deg(b)$ then  
\begin{equation*}
B_{m}(a,b) = - \frac{q^2 + q}{2(q-1)^2} \phi(m) l(a,b) 
+ O_q\left( \left( |a| + |b| \right) M^2 \right),
\end{equation*}
where $l(a,b) = 1$ if $a=-b$ and $0$ otherwise. 
    \item If $\deg(a) \ne \deg(b) $ then 
\begin{equation*}
B_{m}(a,b) = \frac{-q}{(q-1) \log q} \phi(m) \Lambda_{0}\left( \frac{\mathbf{Pmax}(a,b)}{\mathbf{Pmin}(a,b)}\right) 
+ O_q\left( \left(|a| + |b| \right) M^2 \right), 
\end{equation*}    
\end{enumerate}
where 
\begin{equation*}
\mathbf{Pmax}(a,b) = 
\begin{cases}
a & \text{ if } \deg(a) > \deg(b), \\
b & \text{otherwise.}
\end{cases}
\quad \text{and} \quad
\mathbf{Pmin}(a,b) = 
\begin{cases}
b & \text{ if } \deg(a) > \deg(b), \\
a & \text{otherwise.}
\end{cases}
\end{equation*}
\end{pro}
To prove this result, we need the following lemma.
\begin{lem} 
\label{majorSomme62}
Let $m \in \mathcal{M}_{q}$ be of large degree $M$. Let $a,b \in \mathbb{F}_{q}[T]$  such that $a \ne b$, $(a,m) = 1$ and $(b,m) =1$ and $0 \leq \deg(a), \deg(b)  < M$. Then 
\begin{equation*}
\begin{aligned}
\sum\limits_{P^{v} || m}^{} \sum\limits_{n=1}^{\infty} \sum\limits_{\substack{ 1 \leq e \leq (9 \log q) M \\ aP^e \equiv b \bmod m/P^v \\ \deg(P^e) = n}}^{} \frac{\log |P|}{|P|^{e+v-1} (|P| - 1)}  &\ll_q  \frac{(|a| + |b|) M^2}{|m|}.   
\end{aligned}
\end{equation*}
\end{lem}

\begin{proof}

If $P | m$ and $(b,m) = 1$ then it is clear that $aP^{e} - b \ne 0$.  This implies that if $ (m/P^{v})   | (aP^e - b)$, then  we have $|m/P^v| \leq |aP^e - b| \leq |a| |P|^{e} + |b|$, hence $|m/P^v| \leq \left(|a| + |b| \right) |P|^{e}$, therefore  $\frac{1}{|P|^{e+v}} \leq \frac{|a| + |b|}{|m|}.$ 
Thus we obtain 
\begin{equation}
\label{lem62Inter1}
\begin{aligned}
\sum\limits_{P^{v} || m}^{} \sum\limits_{n=1}^{\infty} \sum\limits_{\substack{ 1 \leq e \leq 9 \log |m| \\ aP^e \equiv b \bmod m/P^v \\ \deg(P^e) = n}}^{} \frac{\log |P|}{|P|^{e+v-1} (|P| - 1)} &\ll_q M \left( \sum\limits_{P | m}^{} \frac{|P| \log|P|}{|P| - 1} \right) \frac{(|a| + |b|)}{|m|}.    
\end{aligned}
\end{equation}
On the other hand, writing $m = \alpha P_1^{t_1} \dots P_k^{t_k}$ where $\alpha \in \mathbb{F}_q^{*}$, each $P_i$ is an irreducible monic polynomial in $\mathbb{F}_q[T]$ such that $P_i \ne P_j$ for $1 \leq i \ne j \leq k$, and each $t_i$ is a positive integer, we obtain
\begin{equation}
\label{lem62Inter2}
\begin{aligned}
\sum\limits_{P | m}^{} \log |P| &\leq  \sum\limits_{i = 1}^{k} t_i \log |P_i| = \log |m|
\end{aligned}
\end{equation}
Hence by combining \eqref{lem62Inter1} and \eqref{lem62Inter2}
we deduce our result. 
\end{proof}

\begin{proof}[Proof of Proposition \ref{moreExpBm}]
Since $0 \leq \deg(a), \deg(b) < \deg(m) = M$, then $a+b \equiv 0 \bmod m$ implies $a= -b$. Furthermore by using the fact that $|(m,a-b)| \leq |a| + |b|$ and \eqref{majorSomme2}, we obtain that 
\begin{equation*}
\begin{aligned}
\frac{\Lambda(m/(m,a-b))}{\phi(m/(m,a-b))} &\ll_q \frac{M^2}{|m|} (|a| + |b|).
\end{aligned}
\end{equation*}
Hence from Proposition \ref{forExpBm}, and Lemmas \ref{lemElemMajBm} and  \ref{majorSomme62} we deduce that 
\begin{equation}
\label{formuleExpliciteBm}
\begin{aligned}
B_{m}(a,b) &= - \phi(m) \left( \frac{q^2 + q}{2(q-1)^2 }l(a,b) + \frac{q}{(q-1) \log q} \left( \frac{\Lambda(s_1)}{|s_1|} + \frac{\Lambda(s_2)}{|s_2|} \right) \right) \\ &\quad+ O_q \left( (|a| + |b|) M^2 \right),
\end{aligned}
\end{equation}
where $s_1$ and $s_2$ denote the residues of  $ba^{-1} \bmod m$, $ab^{-1} \bmod m$ with the least degree, respectively.  \\
Let us first prove the first part of our proposition. We assume that $\deg(a) = \deg(b)$. Therefore,
if $s_1 a = b$ then $\Lambda(s_1) = 0.$ Otherwise, since $m | s_1 a - b$ then $|s_1| \geq \frac{|m|}{(|a| + |b|)}$, therefore  $\frac{\Lambda(s_1)}{|s_1|} \ll_q \frac{(|a| + |b|) M}{|m|}$. Then we similarly distinguish the cases $s_2 b = a$, $s_2 b \ne a$ and it follows that $\frac{\Lambda(s_1)}{|s_1|} + \frac{\Lambda(s_2)}{|s_2|} \ll_q \frac{(|a| + |b|) M}{|m|}$. Therefore, we deduce from  \eqref{formuleExpliciteBm} that  
\begin{equation*}
B_m(a,b) = - \frac{q^2 + q}{2 (q-1)^2} \phi(m) l(a,b) +    O_{q}\left( \left(|a| + |b| \right) M^2 \right). 
\end{equation*}

Now, we assume that $\deg(a) \ne \deg(b)$. Without loss of generality we assume that $\deg(a) <  \deg(b)$.
In this case we necessarily have that $s_2 b \ne a$. This implies that $ \frac{\Lambda(s_2)}{|s_2|} \ll_q \frac{(|a| + |b|) M}{|m|}$. \\
If $a \nmid b$ then $s_1 a \ne b$ so $\frac{\Lambda(s_1)}{|s_1|} \ll_q \frac{(|a| + |b|) M}{|m|}$. Otherwise $\frac{\Lambda(s_1)}{|s_1|} = \Lambda_{0}\left(\frac{b}{a}\right)$ where $\frac{b}{a} \in \mathbb{F}_{q}[T]$. Hence  
\begin{equation*}
\frac{\Lambda(s_1)}{|s_1|} = \Lambda_{0}\left(\frac{b}{a} \right)  + O_q\left( \frac{\left( |a| + |b| \right) M}{|m|} \right).
\end{equation*}
By combining these estimates with  \eqref{formuleExpliciteBm}, we obtain 
the second part of our proposition. 
\end{proof}

\quad Let $P'$ be one of those irreducible monic polynomials in  $\mathbb{F}_{q}[T]$ with the largest degree such that $P' | m.$ Let $P_{0}$ be one of those irreducible monic polynomials in  $\mathbb{F}_{q}[T]$ with the least degree which is a quadratic non-residue modulo $P'$. Then from \cite[Corollary $2.2$]{CNH} 
we have $\deg(P_{0}) \leq 2 + 2 \log_q(1 + \deg(P))$ hence
$ \left| P_{0} \right| \leq q^2 ( 1 + M)^2.$ \\
\quad Let us consider $S$ the set of primes $P \in \mathbb{F}_{q}[T]$ such that $|P| \leq 2qM  $ and $P \nmid m$. We define the following product 
\begin{equation*}
\begin{aligned}
Q &:= \prod\limits_{|P| \leq 2qM}^{} |P| = \exp{\left( \log q \left( \sum\limits_{1 \leq \deg(P) \leq \lfloor \frac{\log(2qM)}{\log q} \rfloor}^{} \deg(P) \right) \right)}.
\end{aligned}
\end{equation*}
Therefore by using  \eqref{PNT2} it follows that
\begin{equation*}
\begin{aligned}
Q &= \exp{\left( \log q \left( \frac{q^{\lfloor \frac{\log(2qM)}{\log q} \rfloor + 1} }{q-1} + O\left(\sqrt{q}^{\lfloor \frac{\log 2qM}{\log q} \rfloor}  \right)  \right) \right)} \geq  |m|^2 (1 + o(1)).  
\end{aligned}
\end{equation*}
On the other hand, we know that $\prod\limits_{P \in S}^{} |P| = \frac{Q}{\prod\limits_{\substack{|P| \leq 2qM \\ P | m}}^{} |P|} \geq \frac{|m|^2 (1 + o(1))}{\prod\limits_{\substack{|P| \leq 2qM \\ P | m}}^{} |P|}.$
Furthermore, it is clear that $\prod\limits_{\substack{|P| \leq 2qM \\ P | m}}^{} |P| \leq \prod\limits_{P|m}^{} |P| \leq |m|$ and hence
$\prod\limits_{P \in S}^{} |P| \geq |m| (1 + o(1))$.
Moreover, since $P \in S$ it follows that $|P| \leq 2qM$, so
\begin{equation*}
(2qM)^{|S|} \geq |m| (1 + o(1)).    
\end{equation*}
Therefore 
\begin{equation*}
|S| \geq \frac{M}{\log(2qM)}.    
\end{equation*}
In particular $|S| \geq 2$, and so there exist $P_1$, $P_2 \ne P_{0}$ two distinct irreducible monic polynomials in $\mathbb{F}_{q}[T]$ such that $|P_1|,|P_2| \leq 2qM.$ \\
We end this section by proving Theorems \ref{biasedQuadratic} and \ref{biasedRaces} and \ref{generMartin}.

\begin{proof}[Proof of Theorem \ref{biasedQuadratic}]
Since we know from \cite[Theorem $2.3$]{CSK} that $\delta_{m;ba_1,\dots,ba_r} = \delta_{m;a_1,\dots,a_r}$ for any residue class modulo $m$, then it is sufficient to construct quadratic residues $a_j$ modulo $m.$ In fact, we take $b_j = b a_j$ for any $b$ quadratic non-residue modulo $m$ to get the analogous result for quadratic non-residues modulo $m$. \\

Let $a_1 = 1, a_r = P_1^2$ and $a_j =(P_1 P_2)^{2j}$ for all $2 \leq j \leq r-1$. Hence we have $ \left| a_j \right| \leq (2qM)^{4(r-1)}$ for all $1 \leq j \leq r$. Moreover, we have $P_1 P_2 | a_k/a_j$ for all $1 \leq j < k \leq r-1$. Thus from part $2$ of Proposition \ref{moreExpBm} it follows that
\begin{equation*}
\begin{aligned}
B_m(a_j,a_k) &\ll_q M^{4r}    \quad \text{for} \ 1\leq j < k \leq r-1.
\end{aligned}
\end{equation*}
On the other hand, since $a_r/a_1 = P_1^{2}$ then again by part $2$ of Proposition \ref{moreExpBm} we have 
\begin{equation*}
\begin{aligned}
B_{m}(a_1,a_r) &= - \frac{\phi(m) q}{(q-1) \log q}  \frac{\log |P_1|}{|P_1|^{2}} + O_q \left( M^{4r} \right).    
\end{aligned}
\end{equation*}
It follows from \cite[Lemma $6.3$]{lamzouri} that 
\begin{equation}
\label{lamLemme63}
\begin{aligned}
\beta_{r-1,r}(r) > 0, \quad \text{and} \quad \beta_{1,r}(r) < 0.
\end{aligned}
\end{equation}
Hence by combining these estimates with Corollary \ref{cor3}, Lemma \ref{formuleNm} and  \eqref{lamLemme63} we obtain
 \begin{equation*}
\begin{aligned}
\delta_{m;a_1,\dots,a_r} &= \frac{1}{r!} + \frac{\beta_{1,r}(r) B_{m}(a_1,a_r)}{N_{m}} + O_q \left( \frac{M^{4r}}{\phi(m)} \right)    > \frac{1}{r!} + \frac{\left| \beta_{1,r}(r) \right|}{10 q^2 M^{3}}. 
\end{aligned}
\end{equation*}
Moreover, let $\sigma$ be a permutation of the set $\{1,\dots,r \}$ defined by $\sigma(1) = r - 1$, $\sigma(r-1) = 1$ and $\sigma(j) = j$ for all the other values of $j.$ Therefore we deduce from  \eqref{lamLemme63} that
\begin{equation*}
\begin{aligned}
\delta_{m;a_{\sigma(1)},\dots,a_{\sigma(r)}} &= \frac{1}{r!} + \frac{\beta_{r-1,r}(r) B_{m}(1,P_1^{2})}{N_m} +  O_q \left( \frac{M^{4r}}{\phi(m)} \right) < \frac{1}{r!} - \frac{\left| \beta_{r-1,r}(r) \right|}{10 q^2 M^{3}}.    
\end{aligned}
\end{equation*}
\end{proof}

\begin{proof}[Proof of Theorem \ref{biasedRaces}]
Let $M = \deg(m)$. By combining the estimate $|C_m| = |m|^{o(1)}$, together with Theorem \ref{forAsymDelta3}  and Corollary \ref{majorBm} we obtain 
\begin{equation*}
\begin{aligned}
\left| \delta_{m;a_1,\dots,a_r} - \frac{1}{r!} \right| &\ll_{r}
\frac{1}{M}.
\end{aligned}
\end{equation*}

On the other hand, we choose $a_1 = 1, a_r = -1$ and $a_j = \left(P_1 P_2 \right)^{2j}$ for $2 \leq j \leq r-1.$ Hence it follows that $|a_j| \leq (2qM)^{4(r-1)}$ for all $1 \leq j \leq r$. 
Since $P_1 P_2 | a_k/a_j$ then by using part $2$ of Proposition \ref{moreExpBm} we obtain 
\begin{equation*}
\begin{aligned}
B_m(a_j,a_k) &\ll_q M^{4r}    \quad \text{for} \ 1\leq j < k \leq r-1.
\end{aligned}
\end{equation*}
Moreover, since $P_1 P_2 | a_j/a_r$ for all $2 \leq j \leq r-1$ then
\begin{equation*}
\begin{aligned}
B_m(a_j,a_k) &\ll_q M^{4r}    \quad \text{for} \ 2\leq j \leq r-1.
\end{aligned}
\end{equation*}
Furthermore, it follows from part $1$ of Proposition \ref{moreExpBm} that
\begin{equation*}
\begin{aligned}
B_m(a_1,a_r) &= - \frac{q^2 + q}{2(q-1)^2} \phi(m) + O_q \left( M^{2} \right).     
\end{aligned}
\end{equation*}
Thus, we deduce by combining Theorem \ref{forAsymDelta3} with Lemma \ref{formuleNm} and  \eqref{lamLemme63} that
\begin{equation*}
\begin{aligned}
\delta_{m;a_1,\dots,a_r} = \frac{1}{r!} + \frac{\beta_{1,r}(r) B_{m}(a_1,a_r)}{N_{m}} + O_{\epsilon,q}\left( \frac{1}{\phi(m)^{1/2 - \epsilon}} \right) > \frac{1}{r!} + \frac{|\beta_{1,r}(r)| (q^2 + q)}{6 (q-1)^2 M}.    
\end{aligned}
\end{equation*}
Therefore 
\begin{equation*}
\begin{aligned}
\delta_{m;a_1,\dots,a_r} &> \frac{1}{r!} + \frac{\left|\beta_{1,r}(r) \right|}{6 M}.    
\end{aligned}
\end{equation*}
On the other hand, if we take $b_1 = a_{r-1}$, $b_{r-1} = a_1$ and $b_j = a_j$ for all the other values of $j$ then we deduce by using  \eqref{lamLemme63} that 
\begin{equation*}
\begin{aligned}
\delta_{m;b_1,\dots,b_r} &= \frac{1}{r!} + \frac{\beta_{r-1,r}(r) B_m(b_{r-1},b_r)}{N_{m}} +  O_{\epsilon,q}\left( \frac{1}{\phi(m)^{1/2 - \epsilon}} \right) < \frac{1}{r!} - \frac{|\beta_{r-1,r}(r)| (q^2 + q)}{6(q-1)^2 M }.     
\end{aligned}
\end{equation*}
Hence 
\begin{equation*}
\begin{aligned}
\delta_{m;b_1,\dots,b_r} &< \frac{1}{r!} - \frac{\left| \beta_{r-1,r}(r) \right|}{6 M}.    
\end{aligned}
\end{equation*}
\end{proof}

\begin{proof}[Proof of Theorem \ref{generMartin}]
Let $M = \deg(m)$. Since $(\kappa_1,\dots,\kappa_r) \ne (0,\dots,0)$ then there exists $1 \leq l \leq r$ such that $\kappa_{l} \ne 0.$ \\
\noindent {\bf Case 1}: $\kappa_r\neq 0$ or $\kappa_1\neq 0$.

We only handle the case $\kappa_r\neq 0$, since the treatment of the case $\kappa_1\neq 0$ follows simply by switching $a_1$ with $a_r$, and $b_1$ with $b_r$ in every construction we make below.

Let us suppose that $\kappa_r > 0.$ In this case, we take $a_1 = 1, a_j = P_{0} \left( P_{1} P_{2} \right)^{2j}$ for $2 \leq j \leq r-1$, and $a_r = \left( P_1 P_2 \right)^{2}.$ Then $a_1$ and $a_r$ are quadratic residues modulo $m$ and $a_j$ is a quadratic non-residue modulo $m$ for $2 \leq j \leq r-1.$ Moreover choose $b_j = a_j$ for all $1 \leq j \leq r-1$ and $b_r = P_{0}.$ In this case, $b_1$ is the only quadratic residue modulo among the $b_j$ modulo $m$. Furthermore since $C_{m}(1) > -1$ it follows that 
\begin{equation*}
\sum\limits_{j=1}^{r} \kappa_j C_{m}(a_j) - \sum\limits_{j=1}^{r} \kappa_j C_m(b_j) = \kappa_r C_m(a_r) - \kappa_r C_m(b_r) = \kappa_r ( C_m(1) + 1) > 0.      
\end{equation*}
On the other hand, we know that $\left| a_j \right| \ll q^2 M^2 \left( 2 q M \right)^{4(r-1)} $ for all $1 \leq j \leq r,$ and that $P_1 P_2$ divides $\mathbf{Pmax}(a_j,a_k)/ \mathbf{Pmin}(a_j,a_k)$ for all $1 \leq j < k \leq r.$ Then by using part $2$ of Proposition \ref{moreExpBm} we get 
\begin{equation*}
\left| B_{m}(a_j,a_k) \right| \ll_q M^{4 r + 1} \ \text{for all} \ 1 \leq j < k \leq r.    
\end{equation*} 
Hence by Theorem \ref{forAsymDelta3} we deduce that 
\begin{equation}
\label{eq1}
\delta_{m;a_1,\dots,a_r} = \frac{1}{r!} + O_{\epsilon,q} \left( \frac{1}{\phi(m)^{1/2 - \epsilon}} \right) . 
\end{equation}
Similarly, by using part $2$ of Proposition \ref{moreExpBm} we obtain  $\left| B_{m}(b_j,b_k) \right| \ll_q M^{4r + 1}$ for all $1 \leq j < k \leq r$ with $\left\{ j,k \right\} \ne \left\{ 1,r \right\}$, and 
\begin{equation*}
\begin{aligned}
B_{m}(b_1,b_r) &= - \frac{q \phi(m)}{(q-1) \log q}  \frac{\log |P_{0}|}{|P_{0}|} + O_q \left( M^{4r + 1} \right).    
\end{aligned}
\end{equation*}
Hence combining Theorem \ref{forAsymDelta3} with  \eqref{lamLemme63} and \eqref{eq1} we deduce that 
\begin{equation*}
\delta_{m;b_1,\dots,b_r} = \frac{1}{r!} + \frac{\beta_{1,r}(r) B_{m}(b_1,b_r)}{N_{m}} + O_{\epsilon,q} \left( \frac{1}{\phi(m)^{1/2 - \epsilon}} \right) > \frac{1}{r!} +  \frac{\left| \beta_{1,r}(r)\right| \log |P_{0}|}{2 |P_{0}| (\log q) M} > \delta_{m;a_1,\dots,a_r}.  
\end{equation*}
\\
Now suppose that $\kappa_r < 0.$ In this case, we choose $a_1 = 1$, and $a_j = P_{0} (P_1 P_2)^{2j}$ for all $2 \leq j \leq r$ (in this case $a_1$ is the only quadratic residue $m$ among the $a_j$ modulo $m$), and $b_j = a_j$ for all $1 \leq j \leq r-1,$ and $b_r = P_{1}^{2}$ (in this case both $b_1$ and $b_r$ are quadratic residues modulo $m$). Therefore by using similar arguments to the case $\kappa_r > 0$ we get from Theorem \ref{forAsymDelta3} and part $2$ of Proposition \ref{moreExpBm} combined with  \eqref{lamLemme63} that
\begin{equation*}
\begin{aligned}
\sum\limits_{j=1}^{r} \kappa_j C_{m}(a_j) - \sum\limits_{j=1}^{r} \kappa_j C_m(b_j) = - \kappa_r \left( 1 + C_{m}(1) \right) > 0, \\
\delta_{m;a_1,\dots,a_r} = \frac{1}{r!} + O_{\epsilon,q} \left( \frac{1}{\phi(m)^{1/2 - \epsilon}} \right),
\end{aligned}
\end{equation*}
and 
\begin{equation*}
\delta_{m;b_1,\dots,b_r} = \frac{1}{r!} + \frac{\beta_{1,r}(r) B_{m}(b_1,b_r)}{N_{m}} + O_{\epsilon,q} \left( \frac{1}{\phi(m)^{1/2 - \epsilon}} \right) > \frac{1}{r!} + \frac{\left| \beta_{1,r}(r) \right| \log |P_1|}{2 (\log q) |P_1|^{2} M} > \delta_{m;a_1,\dots,a_r}.    
\end{equation*}
\\
\textbf{Case 2:} There exists $2 \leq l \leq r-1$ such that $\kappa_{l} \ne 0$.

First, assume that $\kappa_{l} > 0.$ We take $a_1 = 1, a_l = (P_1 P_2)^2$,  
$a_j = P_{0} \left( P_1 P_2 \right)^{4j}$ for $2 \leq j \ne l \leq r$ and $b_r = P_{0}$, $b_l = P_{0} \left( P_1 P_2 \right)^{4l}$, and $b_j = a_j$ for all $1 \leq j\ne l \leq r-1$.
Hence by using similar arguments to the case $1$ we obtain that if $M$ is large enough that
\begin{equation*}
\sum\limits_{j=1}^{r} \kappa_j C_{m}(a_j) - \sum\limits_{j=1}^{r} \kappa_j C_m(b_j) = \kappa_{l} \left( C_{m}(1) + 1 \right) > 0, \quad \text{and} \quad \delta_{m;b_1,\dots,b_r} > \delta_{m;a_1,\dots,a_r}.     
\end{equation*}

Now if $\kappa_{l} < 0,$ we choose $a_1 = 1, a_r = \left( P_1 P_2 \right)^{4}$, $a_j =P_{0} \left( P_1 P_2 \right)^{4j}$ for $2 \leq j \leq r-1$, $b_l = \left( P_1 P_2 \right)^{4}, b_r = P_1^{2}$ and $b_j = a_j$ for the other values of $j$, which leads to our result.
\end{proof}

\section{$m$-extremely biased races}
\label{section9}The goal of this section is to prove Theorem \ref{criteriaExtreme}. In order to do this,
we will adapt the proof of \cite[Theorem $2.6$]{lamzouri}. The idea is to
see when the term $B_m(a_i,a_j)$ have a large contribution to the density $\delta_{m;a_1,\dots,a_r}.$ By using Proposition \ref{moreExpBm} we respond to this question. We start by reducing our study to the case $r=3$ which is treated in the following lemma.
\begin{lem}
\label{lemme91}
Let $r \geq 3$ be a fixed integer, $m \in \mathcal{M}_{q}$ of large degree $M$ and $(a_1,\dots,a_r) \in \mathcal{A}_{r}(m)$. If there exist $1 \leq i_1 < i_2 <i_3 \leq r$ such that the race $\{m;a_{i_1},a_{i_2},a_{i_3} \}$ is $m$-extremely biased, then $\{ m;a_1,\dots,a_r \}$ is $m$-extremely biased.  
\end{lem}

\begin{proof}
Assume that there exist $1 \leq i_1 < i_2 <i_3 \leq r$ such that the race $\{ m;a_{i_1},a_{i_2},a_{i_3}\}$ is $m$-extremely biased. It follows that for some permutation $\nu$ of the set $\{i_1,i_2,i_3 \}$ we have
\begin{equation*}
\begin{aligned}
\left| \delta_{m;a_{\nu(i_1)},a_{\nu(i_2)},a_{\nu(i_3)}} - \frac{1}{6} \right| &\gg_{r,q} \frac{1}{M}. 
\end{aligned}
\end{equation*}
Let $j_{l} = \nu(i_l)$ for $l \in \{ 1,2,3 \}$. Let $S$ be the set of permutations $\sigma$ of $\{1,2,\dots,r \}$ such that there exist $1 \leq i < k < l \leq r$ verifying the facts that $\sigma(i) = j_1$, $\sigma(k) = j_2$, and $\sigma(l) = j_3$. Note that under LI, the density of the set of all positive integers $X$ such that 
$$
\sum\limits_{N=1}^{X}\pi_q(a_j,m,N) = \sum\limits_{N=1}^{X} \pi_q(a_k,m,N)
$$
is $0$ if $j \ne k$. Therefore 
\begin{equation*}
\begin{aligned}
\delta_{m;a_{j_1},a_{j_2},a_{j_3}} &= \sum\limits_{\sigma \in S}^{} \delta_{m;a_{\sigma(1)},\dots,a_{\sigma(r)}}.     
\end{aligned}
\end{equation*}
On the other hand, it is clear that $|S| \leq \frac{r!}{3!}$ hence 
\begin{equation*}
\begin{aligned}
\frac{1}{M} &\ll_{r,q} \left| \delta_{m;a_{j_1},a_{j_2},a_{j_3}} - \frac{1}{6} \right| \leq \sum\limits_{\sigma \in S}^{} \left| \delta_{m;a_{\sigma(1)},\dots,a_{\sigma(r)}} - \frac{1}{r!} \right| \ll_{r} \max_{\sigma \in S} \left| \delta_{m;a_{\sigma(1)},a_{\sigma(2)},\dots,a_{\sigma(r)}} - \frac{1}{r!} \right|,      
\end{aligned}
\end{equation*}
which implies that the race $\{m;a_1,\dots,a_r\}$ is $m$-extremely biased. 
\end{proof}

Next, we establish some properties related to the function $\Lambda_0(f)$ defined in \eqref{lambdaDef}.

\begin{lem}
\label{majorLog3}
The maximum of $\Lambda_{0}(f)$ in $\mathbb{F}_{q}[T]$ is less than or equal to $\log(3)/3.$ 
\end{lem}

\begin{proof}
It is clear that $\Lambda_{0}(f) \ne 0 \Leftrightarrow f= P^{l}$, where  $P \in \mathcal{P}_{q}$, and $l \in \mathbb{N}^{*}.$ In this case, since $|P|= q^{\deg(P)} \geq 3$ then it follows that 
\begin{equation*}
\begin{aligned}
\Lambda_{0}(f) = \frac{\log |P|}{|P|^{l}} &\leq \frac{\log |P|}{|P|} \leq \frac{\log 3}{3}.
\end{aligned}
\end{equation*}  
\end{proof}

\begin{lem}
\label{lem93}
Let $a_1,a_2,a_3$ be elements of $\mathbb{F}_{q}[T]$ prime to $m$ representing distinct residue classes modulo $m$ and of distinct degrees. We define 
\begin{equation*}
\begin{aligned}
X_1 = \frac{\mathbf{Pmax}(a_1,a_2)}{\mathbf{Pmin}(a_1,a_2)}, & \ X_2 =\frac{\mathbf{Pmax}(a_2,a_3)}{\mathbf{Pmin}(a_2,a_3)}, \text{  and}  & X_3 = \frac{\mathbf{Pmax}(a_1,a_3)}{\mathbf{Pmin}(a_1,a_3)}.    
\end{aligned}
\end{equation*}
If one of the values $\Lambda_{0}(X_1)$,  $\Lambda_{0}(X_2)$ , $\Lambda_{0}(X_3)$ is non-zero then there exists a permutation $\sigma$ of $\{ 1,2,3 \}$ such that
\begin{equation*}
\Lambda_{0}(X_{\sigma(1)}) + \Lambda_{0}(X_{\sigma(2)}) - 2 \Lambda_{0}(X_{\sigma(3)}) \ne 0.     
\end{equation*}
\end{lem}

\begin{proof}
We suppose without loss of generality that $|a_1| > |a_2| > |a_3|$ (since $a_1,a_2,a_3$ play a symmetric role in the proof). So in this case, we obtain that $X_1 = \frac{a_1}{a_2}$, $X_2 = \frac{a_2}{a_3}$, $X_3 = \frac{a_1}{a_3}.$  
Assume that for all permutations $\sigma$ of the set $\left\{1,2,3 \right\}$ we have $\Lambda_{0}(X_{\sigma(1)}) + \Lambda_{0}(X_{\sigma(2)}) - 2 \Lambda_{0}(X_{\sigma(3)}) = 0$, then we get $\Lambda_{0}(X_1) = \Lambda_{0}(X_2) = \Lambda_{0}(X_3)$. Since this value is non-zero, it follows that $X_1 = P^k$ and $X_2 = Q^j$, where $k,j \geq 1$ are two integers and $P,Q \in \mathcal{P}_{q}.$ Since $X_3 = X_1 X_2 = P^{k} X_2$ and $\Lambda_{0}(X_3) \ne 0$ then $X_3 = P^{k+a}$ with $a$ a non-negative integer. Hence  
$\frac{\log |P|}{|P|^{k}} = \frac{\log|P|}{|P|^{k+a}}$, which is possible if and only if $a = 0.$ Therefore $a=0$, so $X_1 = X_3$ and $X_2 = 1$, thus $\Lambda_{0}(X_2) = 0.$ However this can not hold since we assumed that this value is non-zero. 
\end{proof}

The last step is to observe what happens to the main contribution of $B_m(a,b)$ when $\deg(a),\deg(b)$ are relatively small compared to $M = \deg(m)$ and $\mathbf{Pmax}(a,b)/\mathbf{Pmin}(a,b)$ equals a prime power. This leads to our proof of Theorem \ref{criteriaExtreme}. 
\begin{proof}[Proof of Theorem \ref{criteriaExtreme}]
Suppose first that neither $(1)$ nor $(2)$ hold. In this case, from Proposition \ref{moreExpBm} we obtain that $B_{m}(a_j,a_k) = O_{A,q} \left( M^2 \right)$ for all $1 \leq j < k \leq r.$  Hence, in the case where the $a_i$ are all quadratic residues modulo $m$ (or all quadratic non-residues modulo $m$), it follows from Corollary \ref{cor3}  that  $\left| \delta_{m;a_1,\dots,a_r} - \frac{1}{r!} \right| \ll_{A,r,q} \frac{M^2}{|m|}.$ On the other hand, if the last condition is not verified then we have from Theorem \ref{forAsymDelta3} that $\left| \delta_{m;a_1,\dots,a_r} - \frac{1}{r!} \right| \ll_{\epsilon,r,q} |m|^{-1/2 + \epsilon}.$ Hence in the two cases, the race $\{m;a_1,\dots,a_r \}$ is not $m$-extremely biased.  

Next, assume that there exist $1 \leq j \ne k \leq r$ such that $a_j = - a_k = a.$ 
Since $r \geq 3$ then there exists $b \in \{a_1,\dots,a_r \}$ such that $b \ne a$ and $b \ne -a.$ By Lemma \ref{lemme91}, it is sufficient to prove that the race $\{m;a,-a,b \}$ is $m$-extremely biased. In order to prove this, we will distinguish two cases. 

\textbf{Case 1:} Assume that $\deg(a) \ne \deg(b)$ then from part $2$ of Proposition \ref{moreExpBm} and Lemma \ref{majorLog3} we have
\begin{equation*}
\begin{aligned}
B_{m}(a,b) &= - \frac{q}{(q-1) \log q} \phi(m) \Lambda_{0}\left(\frac{\mathbf{Pmax}(a,b)}{\mathbf{Pmin}(a,b)} \right) + O_{A,q} \left( M^2 \right) \geq   -\frac{\log 3}{3} \frac{q}{(q-1) \log q} \phi(m) +  O_{A,q} \left( M^2 \right).
\end{aligned}
\end{equation*}
The same formula holds for $B_{m}(b,-a)$. Moreover, since $\deg(a) = \deg(-a)$ it follows from part $1$ of Proposition \ref{moreExpBm} that
\begin{equation}
\label{expreSamedeg}
\begin{aligned}
B_{m}(a,-a) &= - \frac{q^2 + q}{2(q-1)^2} \phi(m) + O_{A,q} \left( M^2 \right).    
\end{aligned}
\end{equation}
Furthermore, from Corollary \ref{cor2} combined with the fact that $\left| C_m(a) \right| = |m|^{o(1)}$ and Lemma \ref{formuleNm} we have
\begin{equation*}
\begin{aligned}
\delta_{m;a,b,-a} &\geq \frac{1}{6} + \frac{ \frac{-2q}{(q-1) \log q}  \frac{\log 3}{3}  + \frac{q^2 + q}{(q-1)^2}}{16 \pi \sqrt{3} M}.   
\end{aligned}
\end{equation*}
Since $\frac{q^2 + q}{(q-1)^2} \geq 1$ and $\frac{q}{(q-1) \log q} \leq \frac{4}{3 \log 4}$ for $q \geq 4$ then $\frac{-2q}{(q-1) \log q}  \frac{\log 3}{3}  + \frac{q^2 + q}{(q-1)^2} \geq 1 - \frac{4 \log 3}{9 \log 2}$ for any $q \geq 4$.  Moreover, when $q=3$ we have that $\frac{-2q}{(q-1) \log q}  \frac{\log 3}{3}  + \frac{q^2 + q}{(q-1)^2} = 2$. Thus in any case 
\begin{equation*}
\begin{aligned}
\delta_{m;a,b,-a} &\geq \frac{1}{6} + \frac{1 - \frac{4 \log 3}{9 \log 2}}{16 \pi \sqrt{3} M},
\end{aligned}
\end{equation*}
which shows that the race $\{m;a,-a,b\}$ is $m$-extremely biased, if $M$ is large enough.  

\textbf{Case 2:} Assume that $\deg(a) = \deg(b).$ Since $b \ne a$ and $b \ne -a$ then from part $1$ of Proposition \ref{moreExpBm} it follows that
\begin{equation*}
B_{m}(a,b),B_{m}(b,-a) = O_{A,q}\left( M^2 \right),
\end{equation*}
and the expression of $B_{m}(a,-a)$ is the same as in \eqref{expreSamedeg}, hence  
\begin{equation*}
\begin{aligned}
\delta_{m;a,b,-a} &\geq \frac{1}{6} + \frac{1}{16 \pi \sqrt{3} M} .   
\end{aligned}
\end{equation*}
Thus the race $\{m;a,-a,b\}$ is $m$-extremely biased, if $M$ is large enough. \\

Now, let us suppose that $\ a_i \ne -a_j$ for all $1 \leq i < j \leq r$ and there exist $b_1,b_2 \in \{ a_1,\dots,a_r \}$ such that $b_1 = P^{k} b_2$ where $P \in \mathcal{P}_{q}$ and $k$ a positive integer. In this case, we obtain from part $2$ of Proposition \ref{moreExpBm} that
\begin{equation}
\label{Bm1}
\begin{aligned}
B_{m}(b_1,b_2) &= - \phi(m) \frac{q}{(q-1) \log q} \frac{\log |P|}{|P|^{k}} + O_{A,q} \left( M^2 \right).  
\end{aligned}
\end{equation}
Since $r \geq 3$ then there exists $b_3 \in \{a_1,\dots,a_r \}$ where $b_3 \ne b_{i}$ for $i = 1,2$. We will distinguish again two cases.

\textbf{Case 1:} Assume that there exist $1 \leq i \ne j \leq 3$ such that $\deg(b_i) = \deg(b_j)$. Without loss of generality, assume that $\deg(b_3) = \deg(b_1).$ Since $b_1 \ne - b_3$ then 
\begin{equation*}
\begin{aligned}
B_{m}(b_1,b_3) &\ll_{A,q} M^2.    
\end{aligned}
\end{equation*}
If $b_3/b_2$ is a prime power then we have $b_3/b_2 = Q^j$ where $j$ is a positive integer and $Q$ is a prime. In this case, from part $2$ of Proposition \ref{moreExpBm} we have 
\begin{equation}
\label{Bm2}
\begin{aligned}
B_{m}(b_2,b_3) &= - \phi(m) \frac{q}{(q-1) \log q} \frac{\log |Q|}{|Q|^{j}} + O_{A,q} \left( M^2 \right).  
\end{aligned}
\end{equation}
Hence by inserting  \eqref{Bm1} and \eqref{Bm2} in Corollary \ref{cor2}, it follows from Lemma \ref{formuleNm} that
\begin{equation*}
\begin{aligned}
\delta_{m;b_1,b_2,b_3} &= \frac{1}{6} - \frac{1}{4 \pi \sqrt{3} (\log q) M} \left( \frac{\log|P|}{|P|^{k}} + \frac{\log |Q|}{|Q|^{j}} \right) \left( 1 + o(1) \right).  
\end{aligned}
\end{equation*}
Therefore the race $\{m;b_1,b_2,b_3 \}$ is $m$-extremely biased. 

\textbf{Case 2:} Assume that $\deg(b_1),\deg(b_2),\deg(b_3)$ are pairwise distinct. Let $S$ be the set of permutation of the set $\{1,2,3 \}.$ Since $\Lambda_{0} \left( b_1/b_2 \right) \ne 0$, then Lemma \ref{lem93} shows that there exists $\sigma \in S_3$ such that 
\begin{equation*}
\Lambda_{0}(X_{\sigma(1)}) + \Lambda_{0}(X_{\sigma(2)}) - 2 \Lambda_{0}(X_{\sigma(3)}) \ne 0,     
\end{equation*}
where
\begin{equation*}
\begin{aligned}
X_1 = \frac{b_1}{b_2} \text{,} \  X_2 &= \frac{\mathbf{Pmax}(b_2,b_3)}{\mathbf{Pmin}(b_2,b_3)} \text{,}    \  X_3 = \frac{\mathbf{Pmax}(b_1,b_3)}{\mathbf{Pmin}(b_1,b_3)}.
\end{aligned}
\end{equation*}
Hence by using Corollary \ref{cor2} and Proposition \ref{moreExpBm} we obtain that
\begin{equation*}
\begin{aligned}
\max_{\nu \in S_3} \left| \delta_{m;b_{\nu(1)},b_{\nu(2)},b_{\nu(3)}} - \frac{1}{6} \right| &\gg_{r,q}  \frac{\left| \Lambda_{0}(X_{\sigma(1)}) + \Lambda_{0}(X_{\sigma(2)}) - 2 \Lambda_{0}(X_{\sigma(3)}) \right|}{M},   
\end{aligned}
\end{equation*}
which implies that the race $\{m;b_1,b_2,b_3 \}$ is $m$-extremely biased. Hence our result follows from Lemma \ref{lemme91}.
\end{proof}

\section{Examples of races where LI is false}
\label{section10}The following section provides some examples of races of primes over function fields. We illustrate that when LI is violated, the races in the function fields case behave differently than in the number field case. \\
\begin{exa}
Let $m = T^2 + T + 1 \in \mathbb{F}_{3}[T]$, we can check that $ m = (T+2)^2$ so $m$
is reducible and $\phi(m) = 9 - 3 = 6$. In this case $(\mathbb{F}_{3}[T]/(m))^{*}$ is cyclic and it is simple to verify that $T+1$ is a generator of $(\mathbb{F}_{3}[T]/(m))^{*}$. More precisely, we have $(T+1)^2 \equiv T \bmod m$, 
$(T+1)^3 \equiv -1 \bmod m \equiv 2 \bmod m$,
$(T+1)^4 \equiv 2T + 2 \bmod m$, $(T+1)^5 \equiv 2(T+1)^2 \bmod m \equiv 2T \bmod m$ and finally $(T+1)^6 \equiv -T^2 - T \bmod m \equiv 1 \bmod m$.

The Dirichlet characters modulo $m$ are then determined entirely by knowing the image of the generator $T+1$. Therefore, the character table of $\left( \mathbb{F}_{3}[T] / (m) \right)^{*}$ is  
\begin{table}[H]
\begin{tabular}{|c|c|c|c|c|c|c|}
        \hline
&$1$ & $2$ & $T$ & $T+1$ & $2T$ & $2T+2$ \\
\hline
$\chi_{1}$ & 1 & -1 & $ \frac{- 1 + \sqrt{3} i}{2}$ & $  \frac{1+\sqrt{3}i}{2}$ & $ \frac{1 -\sqrt{3}i}{2}$ & $ \frac{-1 -\sqrt{3}i}{2}$\\
\hline
$\chi_{2}$ & 1 & 1 & $ \frac{-1 -\sqrt{3}i}{2}$ & $\frac{-1 + \sqrt{3} i}{2}$ & $ \frac{-1 -\sqrt{3}i}{2}$ & $ \frac{-1  + \sqrt{3} i}{2}$\\
\hline
$\chi_{3}$ & 1 & -1 & 1 & -1 & -1 & 1 \\
\hline
$\chi_{4}$ & 1 & 1 & $\frac{- 1 + \sqrt{3} i}{2}$ & $\frac{-1 -\sqrt{3}i}{2}$ & $\frac{- 1 + \sqrt{3} i}{2}$ & $ \frac{-1 -\sqrt{3}i}{2}$ \\
\hline
$\chi_{5}$ & 1 & -1 & $\frac{-1 - \sqrt{3}i}{2}$ & $\frac{1 -\sqrt{3}i}{2}$ & $\frac{1 + \sqrt{3} i}{2}$ & $ \frac{-1 +\sqrt{3}i}{2}$ \\
\hline
\end{tabular}    
\end{table}
Thus by using the fact that we have
$\mathcal{L}(u,\chi_{i}) = 1 + ( \chi_{i}(T) + \chi_{i}(T+1) ) u$ for all $1 \leq i \leq 5$, we obtain that  $\mathcal{L}(u,\chi_{1}) = (\sqrt{3}i) u + 1$, $\mathcal{L}(u,\chi_{2}) = \mathcal{L}(u,\chi_{4}) = -u + 1$,
$\mathcal{L}(u,\chi_{3}) = 1$ and $\mathcal{L}(u,\chi_{5}) = (- \sqrt{3}i) u + 1$.  Hence the only inverse zero of $L$-functions associated to the characters $\chi_2$ and $\chi_{4}$ is $1$ and there are
no inverse zeros of $L$-functions associated to $\chi_{3}$. Finally the only inverse zero of $L$-functions associated to the characters $\chi_1$ and $\chi_{5}$ are respectively $\sqrt{3} e^{-i \frac{\pi}{2}}$ and $\sqrt{3} e^{i \frac{\pi}{2}}$.
Hence the only inverse zero of all the $L$-functions above whose modulus is $\sqrt{3}$ and with argument between $0$ and $\pi$ is $\gamma_{5} = \sqrt{3} e^{i \frac{\pi}{2}}$, thus LI does not hold. We have from \cite[Theorem $2.5$]{BCha} that for all $a \in \mathbb{F}_{3}[T]$ such that $(a,m) =1$ 
\begin{equation}
\label{formEma}
\begin{aligned}
E_{m,a}(X)= -C_m(a) \mathcal{B}_{q}(X) &- \sum_{\chi \ne \chi_{0}} \overline{\chi}(a) \sum_{\gamma_{\chi}} e^{i \theta(\gamma_{\chi}) X} \frac{\gamma_{\chi}}{\gamma_{\chi} - 1} + o(1),   
\end{aligned}
\end{equation}
where we denote by $\gamma_{\chi} = \sqrt{3} e^{i \theta(\chi)}$ any inverse zero of $\mathcal{L}(u,\chi)$ whose modulus is $\sqrt{3}$ and
$X \in \mathbb{N}^{*}.$
In this case, we can easily verify that $E_{m,a}(X+4) = E_{m,a}(X)$ for $X \in \mathbb{N}^{*}$. Then by using the previous formula we get that \\
\begin{table}[H]
    \centering
    \begin{tabular}{|c|c|c|c|c|}
       \hline
        X mod 4 & $E_{m,T}(X)$ (mod $o(1)$) & $E_{m,T+1}(X)$ (mod $o(1)$) &  $E_{m,2T}(X)$ (mod $o(1)$) & $E_{m,2}(X)$ (mod $o(1)$) \\
        \hline
        1  & $\sqrt{3}/2$ & $\sqrt{3}$ & $-\sqrt{3}/2$ & $\sqrt{3}$ \\
        \hline
        2 & $-3/2$ & 3 & $3/2$ & $0$ \\
        \hline
        3 & $-3\sqrt{3}/2$ & 0 & $3 \sqrt{3}/2$ & $0$ \\
        \hline
        4 & $-3/2$ & 0 & $3/2$ & $3$  \\
        \hline
    \end{tabular}
    \caption{}
    \label{tab:table2}
\end{table}
We deduce then that $\delta_{m,T+1,2T,2} = \frac{1}{4}$. Notice that 
$2 \equiv -1 \bmod m$ so $2T \equiv - T \bmod m$ and we have that $2 T^2 \equiv -T^2 \bmod m \equiv T+1 \bmod m$  so $- T^3 \equiv T (T+1) \bmod m \equiv -1 \bmod m$. Hence by taking $\rho = T \in \mathbb{F}_{3}[T]$, $a_1 = 2 \in \mathbb{F}_{3}[T] , a_2 = 2T \in \mathbb{F}_{3}[T]$ and $a_3 = T+1 \in \mathbb{F}_{3}[T]$, we deduce the following : 
\begin{equation*}
\begin{aligned}
\rho^3 \equiv 1 \bmod m,\quad a_2 \equiv a_1 \rho \bmod m,  \quad a_3 \equiv a_1 \rho^2 \bmod m, \quad \text{and} \quad \rho \not\equiv 1 \bmod m.
\end{aligned}
\end{equation*}
On the other hand we showed that $$\delta_{m,a_3,a_2,a_1} = \frac{1}{4} \ne \frac{1}{6}.$$ 
This implies that the races $\{ m;a_1,a_2,a_3 \}$ is biased. Therefore, \cite[Theorem $6.1$]{BCha} does not hold when LI is false.
\end{exa}
\begin{exa}
\label{example2}
With the same $m$ we can find $a_1, a_2 \in \mathbb{F}_{3}[T]$ representing distinct classes in $(\mathbb{F}_{3}[T]/(m))^{*}$ such that $\delta_{m,a_1,a_2} = 0.$ 
In fact, by Table \ref{tab:table2} we can see that $E_{m,T}(X) < E_{m,T+1}(X)$ for all large enough positive integers $X$, which implies 
 $\sum\limits_{N=1}^{X} \pi_q(T,m,N) < \sum\limits_{N=1}^{X} \pi_q(T+1,m,N)$.
 Thus $$\delta_{m,T,T+1} = 0.$$ In particular this trivially implies  $\delta_{m,T,T+1,2T} = 0$. 
\begin{rem}
In $1914$, Littlewood proved that the difference between the number of primes of the form $4n+3$ up to $x$ and the number of primes of the form $4n+1$ up to $x$ changes signs for infinitely many positives integers $x$. The Example \ref{example2} disproves a generalization of Littlewood's theorem in function fields since $\sum\limits_{N=1}^{X} \pi_q(T,m,N) < \sum\limits_{N=1}^{X} \pi_q(T+1,m,N)$ for all large enough positive integers $X.$
\end{rem}
Furthermore, choosing $a_1 = T+1$, $a_2 = T$ and $a_3 = 2$ in $\left( \mathbb{F}_{3}[T] / (m) \right)^{*}$ with $ m = T^2 + T + 1 \in \mathbb{F}_{3}[T]$, we deduce from Table \ref{tab:table2} that $\delta_{m;a_1,a_2} = 1$ but $\delta_{m;a_1,a_2,a_3} = 0$. 
\end{exa}
\begin{exa}
Let $m = T (T+1) (T+2) \in \mathbb{F}_{3}[T].$ We give the table of all non-principal Dirichlet characters modulo $m$ below: \\

\begin{tabular}{|c|c|c|c|c|c|c|c|c|}
    \hline
 &  $1$ & $2$ & $T^2 +1$ & $T^2+T+2$ & $T^2+ 2T + 2$ & $2 T^2 + 2$ & $2T^2 + T + 1$ & $2T^2+ 2T + 1$ \\
\hline 
$\chi_{1}$  & $1$ & $-1$ & $-1$ & $1$ & $-1$ & $1$ & $1$ & $-1$ \\
\hline
$\chi_{2}$  & $1$ & $1$ & $1$ & $-1$ & $-1$ & $1$ & $-1$ & $-1$ \\
\hline
$\chi_{3}$ & $1$ & $-1$ & $-1$ & $-1$ & $1$ & $1$ & $-1$ & $1$ \\
\hline
$\chi_{4}$ & $1$ & $-1$ & $1$ & $-1$ & $-1$ & $-1$ & $1$ & $1$ \\
\hline
$\chi_{5}$ & $1$ & $1$ & $-1$ & $1$ & $-1$ & $-1$ & $-1$ & $1$ \\
\hline
$\chi_{6}$ & $1$ & $-1$ & $1$ & $1$ & $1$ & $-1$ & $-1$ & $-1$ \\
\hline
$\chi_{7}$ & $1$ & $1$ & $-1$ & $-1$ & $1$ & $-1$ & $1$ & $-1$ \\
\hline
\end{tabular}
\medskip
\\

We then directly deduce that for all $1 \leq i \ne 6 \leq 7$  we have $\mathcal{L}(u,\chi_{i}) = 1 - u^2$ and $\mathcal{L}(u,\chi_{6}) = 1 + 3 u^2 = (1 - \sqrt{3} i u) (1 + \sqrt{3} i u)$. Thus the only inverse zero of all $L$-functions above whose modulus is $\sqrt{3}$ and with argument between $0$ and $\pi$ is $\sqrt{3} e^{i \frac{\pi}{2}}.$ Therefore by using \eqref{formEma} we obtain the following table: \\
\begin{center}
    \begin{tabular}{|c|c|c|c|}
     \hline
 X mod 4 & $E_{m,1}(X)$ (mod $o(1)$) & $E_{m,T^2+1}(X)$ (mod $o(1)$)  & $E_{m,2T^2 + 2}(X)$ (mod $o(1)$) \\
\hline
1  & $- 4 \sqrt{3}$ & $0$ & $\sqrt{3}$  \\
\hline
2 & $- 9$ & $3$ & $0$  \\
\hline
3 & $-3 \sqrt{3}$ & $\sqrt{3}$  & $0$  \\
\hline
4 & $- 12$ & $0$ & $3$   \\
\hline
    \end{tabular}
\end{center}

\medskip
Therefore $$\delta_{m,1,T^2+1} = \delta_{m,1,2T^2+2} = 0,$$ and in this case we
can prove easily that if $a$ is a quadratic non-residue modulo $m$ then
$\delta_{m,1,a} = 0.$
\end{exa}

\section*{Acknowledgments}
I would like to thank my advisor Youness Lamzouri for his many helpful advises, remarks and suggestions and for his encouragement. I thank also my co-advisor Manfred Madritsch for his remarks.

\end{document}